\pdfoutput=1
\documentclass[english]{shinybook}
\usepackage{tikz}
\usepackage{pgfplots}
\usepgfplotslibrary{groupplots}
\pgfplotsset{compat=1.18}
\usetikzlibrary{intersections,calc}
\usetikzlibrary{arrows.meta, decorations.markings}
\usepackage{glossaries-extra}

\newcommand{\N}{\mathbb{N}}

\newcommand{\R}{\mathbb{R}}
\newcommand{\Rb}{\overline{\mathbb{R}}}
\newcommand{\bR}{\overline{\mathbb{R}}}
\newcommand{\F}{\mathbb{F}}
\newcommand{\1}{\mathbb{1}}

\newcommand{\E}{\mathbb{E}}
\usepackage{pgffor}
\foreach \x in {A,...,Z}{%
    \expandafter\xdef\csname cal\x\endcsname{\noexpand\ensuremath{\noexpand\mathcal{\x}}}
}

\renewcommand{\phi}{\varphi}

\newcommand{\dom}{\operatorname{\mathrm{dom}}}
\renewcommand{\ker}{\operatorname{\mathrm{ker}}}
\newcommand{\conv}{\operatorname{\mathrm{conv}}}

\newcommand{\tr}{\operatorname{\mathrm{tr}}}
\newcommand{\epi}{\operatorname{\mathrm{epi}}}
\renewcommand{\span}{\operatorname{\mathrm{span}}}
\newcommand{\sign}{\operatorname{\mathrm{sign}}}
\DeclareMathOperator{\Id}{\mathrm{Id}}
\newcommand{\prox}{\mathrm{prox}}
\newcommand{\proj}{\mathrm{proj}}

\newtheorem{theorem}{Theorem}[chapter]
\newtheorem{lemma}[theorem]{Lemma}
\newtheorem{cor}[theorem]{Corollary}

\theoremstyle{definition}
\newtheorem{defn}[theorem]{Definition}
\newtheorem{prop}[theorem]{Proposition}
\newtheorem{example}[theorem]{Example}

\newtheorem{exercise}[theorem]{Exercise}
\newtheorem*{rem}{Remark}

\crefname{theorem}{Theorem}{Theorems}
\crefname{cor}{Corollary}{Corollaries}
\crefname{prop}{proposition}{Proposition}
\crefname{defn}{Definition}{Definitions}
\crefname{example}{Example}{Examples}
\crefname{assumption}{Assumption}{Assumptions}
\crefname{rem}{Remark}{Remarks}
\crefname{exercise}{Exercise}{Exercises}

\newcommand{\norm}[1]{\|#1\|}

\newcommand{\inner}[2]{\left\langle#1 , #2\right\rangle}

\usepackage{enumerate}

\AtBeginEnvironment{remark}{\small}

\usepackage{mdframed}
\mdfdefinestyle{examplebg}{
    backgroundcolor=structure!7,%
    hidealllines=true,%
    splittopskip=\topskip,
    frametitleaboveskip=0,
    frametitlebelowskip=0,
    innertopmargin=-.3em,%
    innerbottommargin=.7em,%
    innerleftmargin=.7em,%
    innerrightmargin=.7em,%
}
\surroundwithmdframed[style=examplebg]{example}

\usepackage{tablefootnote}
\makeatletter
\AfterEndEnvironment{mdframed}{%
 \tfn@tablefootnoteprintout%
 \gdef\tfn@fnt{0}%
}
\makeatother

\usepackage{enumitem}
\setlist[enumerate]{label=(\roman*)}
\usepackage{scrhack}

\newcommand{\ie}{\textit{i.e.}}
\newcommand{\eg}{\textit{e.g.}}
\newcommand{\cf}{\textit{cf.}}

\newcommand{\emptyarg}{\,\cdot\,}

\newcommand{\dd}{\mathrm{d}}

\newcommand{\Oc}{\mathcal{O}}
\newcommand{\Lc}{\mathcal{L}}
\newcommand{\Gc}{\mathcal{G}}
\newcommand{\Ic}{\mathcal{I}}
\newcommand{\Pc}{\mathcal{P}}
\newcommand{\Mc}{\mathcal{M}}

\renewcommand{\epsilon}{\varepsilon}

\newcommand{\diag}{\mathrm{diag}}
\title{Convex Optimization}
\subtitle{lecture notes\\
Graz University of Technology\\
Institute of Visual Computing}
\author{Andreas Habring}
\publishers{\small\email{andreas.habring@tugraz.at}}
\hypersetup{
    pdftitle={Convex Optimization},
    pdfauthor={Andreas Habring},
}

\begin{document}
\maketitle

\frontmatter
\tableofcontents

\mainmatter

\chapter*{Preface}
\markboth{Preface}{Preface}
\enlargethispage*{2cm}

These notes build on lecture slides created by Prof.~Thomas Pock.
This is not a final version of the notes and may contain typos/errors as well as a lack of \emph{connecting text} between results. These notes primarily serve as a collection of content of the corresponding lecture I taught during the summer term of 2026 at Graz University of Technology.

I thank Prof.~Christian Clason for this beautiful latex template.


\chapter{Preliminaries}

\section{Vector spaces, norms, and scalar products}
\begin{defn}[Vector space]\label{preliminaries:defn:vectorspace}
    A vector space over a field $\mathbb{F}$ is a set $V$ with two operations $+:\mathbb{F}\times V\rightarrow V$, $(u,v)\mapsto u+v$, $\cdot:\mathbb{F}\times V\rightarrow V$, $(\lambda,v)\mapsto\lambda \cdot v$ such that
    \begin{enumerate}
        \item $(V,+)$ is an abelian group, that is, the following hold for all $u,v,w\in V$:
        \begin{enumerate}
            \item Associativity: $(u+v)+w = u+(v+w)$
            \item Neutral element: There exists $e\in V$ such that $v+e = e+v = v$. We will denote $e=0$
            \item Inverse element: There exists $\tilde{v}$ such that $v+\tilde{v}=0$ We will denote $\tilde{v} = -v$.
            \item Commutativity: $v+w=w+v$
        \end{enumerate}
        \item The distributive laws hold, that is, for all $u,v\in V$, $\lambda, \mu\in \mathbb{F}$ the following hold
        \begin{enumerate}
            \item $\lambda\cdot(\mu\cdot v) = (\lambda\mu)\cdot v$
            \item $\lambda\cdot(u+v) = (\lambda\cdot v) + (\lambda\cdot\mu)$
            \item $(\lambda+\mu)\cdot v = (\lambda\cdot v) + (\mu\cdot u)$
            \item $1\cdot v = v$.
        \end{enumerate}
    \end{enumerate}
    We will denote $\lambda\cdot v = \lambda v$.
\end{defn}

From now on we will restrict to the setting $\F = \R$.
\begin{example}[Important vector spaces]\
    \begin{enumerate}
        \item $\R^d$: The vector space consisting of vectors of the form 
        \begin{equation}
            v = \begin{pmatrix}
                v_1\\
                \vdots\\
                v_d
            \end{pmatrix}
        \end{equation}
        where $v_i\in \R$ and addition is element-wise and multiplication applied to each element. An important mathematical result states that any finite dimensional vector space can be identified with $\R^d$ where $d$ is the dimension. We can therefore usually think of $\R^d$ when working with finite dimensional spaces.
        \item Function spaces: Let $\Omega\subset\R^d$ and $f,g:\Omega\rightarrow \R$. We can define $f+g$ and for $\lambda\in \R$, $\lambda f$ pointwise as
        \begin{gather}
            (f+g)(x) = f(x) + g(x)\\
            (\lambda f)(x) = \lambda f(x).
        \end{gather}
        This way we can define all sorts of function spaces. For instance the space of linear functions (Why is this a well-defined vector space?)
    \end{enumerate}    
\end{example}

\begin{defn}[Norm]
    A norm is a function $\|\emptyarg\|:V\rightarrow [0,\infty]$ such that the following hold for all $\lambda\in \R$, $u,v\in V$
    \begin{enumerate}
        \item Positivity: $\|v\| = 0$ iff $v=0$.
        \item Homogenity: $\|\lambda u\| = |\lambda|\| u\|$.
        \item Triangle inequality: $\|u+v\|\leq \|u\| + \|v\|$.
    \end{enumerate}
    We call the tuple $(V,\|\emptyarg\|)$ a normed space.
\end{defn}
\begin{rem}
    Note that by definition of the image space the norm is always non-negative.
\end{rem}

From now on all vector spaces are assumed to be normed. If there is no risk of ambiguity, we will denote the norm simply as $\|\emptyarg\|$. To denote a specific norm we will use, \eg, $\|\emptyarg\|_\bigstar$ where $\bigstar$ will be a symbol indicating the specific norms, \cf~\cref{preliminiaries:example:norms}.

\begin{example}[Common norms]\
    \label{preliminiaries:example:norms}
    \begin{enumerate}
        \item $\ell^p$ norms: For $p\in [1,\infty)$ we denote 
        \begin{equation}
            \|v\|_p = \bigg( \sum_{i=1}^d |v_i|^p\bigg)^{1/p}
        \end{equation}
        For $p=\infty$ we denote 
        \begin{equation}
            \|v\|_\infty = \max_{i=1,\dots d}|v_i|
        \end{equation}
        \item Consider the space of $n\times m$ matrices, that is, $\R^{n\times m}$. Every matrix $A\in\R^{n\times m}$ can be identified with a linear function via
        \begin{equation}
            (Av)_i = \sum_{j=1}^m A_{i,j}v_j,\quad i=1,\dots n.
        \end{equation}
        The following norms are frequently used on $\R^{n\times m}$
        \begin{itemize}
            \item Frobenius: $\|A\|_F = \sqrt{\sum_{i,j}|A_{i,j}|^2}$.
            \item Induced norms: For every pair of norms $\|\emptyarg\|_a$, $\|\emptyarg\|_{\Delta}$ on $\R^m$ and $\R^n$, respectively, we can define the induced norm
            \begin{equation}
                \|A\|_{a,\Delta} \coloneq \sup_{\|v\|_a\leq 1}\|Av\|_\Delta = \sup_{v\neq 0}\frac{\|Av\|_\Delta}{\|v\|_a}.
            \end{equation}
            Induced norms satsify $\|Av\|_b\leq \|v\|_a$. If $a=b$ we write $\|\emptyarg\|_{a,a} = \|\emptyarg\|_a$. In particular, we have
            \begin{gather}
                \|A\|_{1} = \max_j \sum_i |A_{i,j}|\\
                \|A\|_{\infty} = \max_i \sum_j |A_{i,j}|\\
                \|A\|_{2} = \sqrt{\lambda_{\text{max}}(A^\top A)}
            \end{gather}
        \end{itemize}
        where $\lambda_{\text{max}}$ denotes the largest eigenvalue.
        \item For $f:\Omega\rightarrow\R$ with $\Omega\subset \R^d$ we may define $L^p$ norms as
        \begin{equation}
            \|f\|_p = \bigg(\int_\Omega |f(x)|^p \dd x\bigg)^{1/p},\quad 1\leq p<\infty
        \end{equation}
        and\footnote{This definition of the $L^\infty$ norm is not entirely correct. One actually has to take care of so-called \emph{null sets} leading to $\|f\|_\infty = \inf_{\substack{N\subset \Omega\\ |\Omega|=0}}\sup_{x\in\Omega}|f(x)|$. But this should not be of your concern in this class.}
        \begin{equation}
            \|f\|_\infty = \sup_{x\in\Omega}|f(x)|.
        \end{equation}
        Then, $\|\emptyarg\|_p$ defines a norm on the \textbf{vector space}
        \begin{equation}
            L^p(\Omega) = \{f:\Omega\rightarrow \R\;|\; \|f\|_p<\infty\}.
        \end{equation}
    \end{enumerate}
\end{example}

\begin{theorem}[Equivalence of norms]
    All norms on $\R^d$ are equivalent, that is for any two norms $\|\emptyarg\|_a$, and $\|\emptyarg\|_b$ there exist $m,M>0$ such that for any $v\in\R^d$
    \begin{equation}
        m\|v\|_a\leq \|v\|_b\leq M\|v\|_a.
    \end{equation}
\end{theorem}

\begin{defn}[Scalar product]
    A mapping $\langle\emptyarg,\emptyarg\rangle:V\times V\rightarrow \R$ is called scalar product or inner product if the following hold for all
    \begin{enumerate}
        \item Bilinearity: $w\mapsto \langle w, u\rangle$ and $w\mapsto \langle u, w\rangle$ are linear maps.
        \item Positivity: $\langle u, u\rangle\geq 0$ and equality holds iff $u=0$
        \item Symmetry: $\langle u,v \rangle = \langle v,u\rangle$.
    \end{enumerate}
    Every inner product defines a norm via $\|v\| \coloneq\sqrt{\inner{v}{v}}$. In this case we refer to the the space as an inner product space.
\end{defn}

\begin{cor}[Cauchy-Schwartz]
    In every inner product space $V$ it holds
    \begin{equation}
        \inner{v}{w}\leq |\inner{v}{w}| \leq \|v\|\|w\|,\quad v,w\in V
    \end{equation}
    and the inequality is strict whenever $v,w$ are not colinear.
\end{cor}
\begin{proof}
    We may assume that $v,w\neq 0$ as otherwise the result holds trivially. We have for any $\lambda$
    \begin{equation}\label{preliminaries:eq:CS}
        0\leq \inner{v-\lambda w}{v-\lambda w} = \|v\|^2 -2\lambda \inner{v}{w} + \lambda^2\|w\|^2.
    \end{equation}
    If we choose $\lambda = \frac{\|v\|}{\|w\|}$ we find
    \begin{equation}
        2\frac{\|v\|}{\|w\|} \inner{v}{w}\leq 2\|v\|^2.
    \end{equation}
    Moreover, whenever $v,w$ are not colinear, $v-\lambda w\neq 0$ for any $\lambda$ and the inequality in~\eqref{preliminaries:eq:CS} is strict.
\end{proof}

\begin{theorem}[Hölder]
    Let $p,q\in [1,\infty]$ such that $\frac{1}{p} + \frac{1}{q}=1$\footnote{We call $p,q$ \emph{Hölder conjugate exponents}.}. Then for any $u,v\in \R^d$ equipped with the standard inner product we have
    \begin{equation}
        \inner{u}{v} \leq |\inner{u}{v}|\leq \|u\|_p \|v\|_q.
    \end{equation}
\end{theorem}
\begin{proof}
    We distinguish two cases:
    \begin{itemize}
        \item \text{p=1}: In this case one easily finds
        \begin{equation}
            |\inner{u}{v}| 
            \leq \sum_i |u_i||v_i|
            \leq\sum_i |u_i|\|v\|_\infty
            \leq\|v\|_\infty\sum_i |u_i| = \|v\|_\infty\|u\|_1.
        \end{equation}
        \item \text{$1<p<\infty$}: First note that the inequality is trivially true if $u=0$ or $v=0$ so assume now $u,v\notin\{0\}$. Note that Young's inequality states that for any Hölder conjugate exponents $1<p,q<\infty$ and $a,b\in\R$ it holds $ab \leq \frac{a^p}{p} + \frac{b^q}{q}$. It follows for any $u,v$
        \begin{equation}\label{eq:prelim:hoelder}
            |\inner{u}{v}| \leq \sum_i |u_i||v_i|\leq \sum_i \frac{|u_i|^p}{p} + \frac{|v_i|^q}{q} \leq \frac{\|u\|_p^p}{p} + \frac{\|v\|_q^q}{q}.
        \end{equation}
        Replacing $u$ by $\frac{u}{\|u\|_p}$ and $v$ by $\frac{v}{\|v\|_q}$ in~\eqref{eq:prelim:hoelder} leads to 
        \begin{equation}
            \begin{aligned}
                \frac{|\inner{u}{v}|}{\|u\|_p\|v\|_q} 
                = \left|\inner{\frac{u}{\|u\|_p}}{\frac{v}{\|v\|_q}}\right|
                \leq \frac{\left\|\frac{u}{\|u\|_p}\right\|_p^p}{p} + \frac{\left\|\frac{v}{\|v\|_q}\right\|_q^q}{q} 
                = \frac{1}{p} + \frac{1}{q} 
                = 1
            \end{aligned}
        \end{equation}
    \end{itemize}
\end{proof}

Using Hölder's inequality, we can show that $\ell^p$ norms satsify the triangle inequality, which is the only more difficult aspect of showing that they are, in fact, norms.

\begin{theorem}
    For any $p\in[1,\infty]$ it holds
    \begin{equation}
        \|x+y\|_p\leq \|x\|_p+\|y\|_p,\quad x,y\in\R^d.
    \end{equation}
\end{theorem}
\begin{proof}
    The cases $p\in\{1,\infty\}$ are left as an easy exercise. Thus, assume $1<p<\infty$. Assume without loss of generality that $\|x+y\|_p$ as otherwise the result holds trivially. We find using Hölder's inequality and the fact that $p-1 = p/q$
    \begin{equation}
        \begin{aligned}
            \|x+y\|_p^p 
            = \sum_i |x_i+y_i||x_i+y_i|^{p-1}
            &\leq \sum_i |x_i||x_i+y_i|^{p-1} + \sum_i |y_i||x_i+y_i|^{p-1}\\
            &\leq (\|x\|_p + \|y\|_p) \bigg(\sum_i |x_i+y_i|^{(p-1)q}\bigg)^{1/q}\\
            &\leq (\|x\|_p + \|y\|_p) \|x+y\|_p^{p/q}.
        \end{aligned}
    \end{equation}
    The proof follows from rearranging terms.
\end{proof}

\section{Sequences and basic topology}

\begin{defn}[Convergence of a sequence]
    A sequence $(x_n)_n\subset\R^d$ is called convergent if there exists $x\in \R^d$ such that for every $\epsilon>0$ there exists $N\in\N$ such that 
    \begin{equation}
        n\geq N\Rightarrow \|x-x_n\|\leq \epsilon.
    \end{equation}
\end{defn}

\begin{rem}
    By equivalence of norms, convergence is independent of the choice of the norm.
\end{rem}

\begin{defn}[Cauchy sequence]
    We call a sequence $(x_n)_n\subset V$ Cauchy iff for any $\epsilon>0$ there exists $N\in \N$ such that for any $n,m\geq N$, it holds $\|x_n-x_m\|<\epsilon$.
\end{defn}

Being Cauchy is similar to being convergent. Note, however, that the Cauchy property can be formulated without the notion of a limit point. One can easily show the following:

\begin{theorem}
    Every convergent sequence is Cauchy.
\end{theorem}
\begin{proof}
    Let $x_n\rightarrow\hat x$. The proof follows from the fact that by the triangle inequality
    \begin{equation}
        \|x_n-x_m\|\leq \|x_n-\hat x\| + \|x_m - \hat x\|
    \end{equation}
    where the latter two terms can be made small by choosing $n,m$ sufficiently large.
\end{proof}

\begin{defn}[Banach and Hilbert space]
    A normed space $(V,\|\emptyarg\|)$ is called Banach space if every Cauchy sequence is convergent. That is, if $(x_n)_n$ is Cauchy, then there exists $x\in V$ such that $\|x_n-x\|_n\rightarrow 0$ as $n\rightarrow \infty$. If the norm is derived from an inner product, we call the space a Hilbert space.
\end{defn}

\begin{example}
    The spaces $(\R^d,\|\emptyarg\|_p)$ are Banach spaces for any $1\leq p\leq \infty$.
\end{example}

\begin{defn}[Basic topology]
    We call a set $A\subset\R^d$ 
    \begin{itemize}
        \item open, if for every $x\in A$ there exists $\epsilon>0$ such that $B_\epsilon (x)\subset A$,
        \item closed, if $A^c$ is open; closedness is equivalent to the following property: If $(x_n)_n\subset A$ and $x_n\rightarrow x\in\R^d$ then $x\in A$ as well.
    \end{itemize}
\end{defn}

\begin{defn}[Sequential copmpactness]
    We call a set $A\subset \R^d$ (sequentially) compact if every sequence in $A$ admits a convergent (in $A$!) subsequence. That is, for every sequence $(x_n)_n\subset A$ there exists a subsequence $(x_{n(k)})_k$ and a point $x\in A$ such that $x_{n(k)}\rightarrow x$.
\end{defn}

\begin{theorem}[Bolzano-Weierstraß]\label{thm:bolzano}
    In a finite-dimensional vector space every bounded sequence admits a convergent subsequence.    
\end{theorem}
\begin{proof}
    Assume first the dimesnion is one, that is, the space is simply $\R$. Let $(x_n)_n\subset V$ be bounded. We will in the following inductively construct two sequences $(a_n)_n, (b_n)_n$ such that $a_n\leq b_n$, $[a_n,b_n]\subset [a_{n-1},b_{n-1}]$, $b_n-a_n = (b_{n-1}-a_{n-1})/2$, and importantly, infinitely many $x_k$ are contained in $[a_n,b_n]$ for any $n$.
    
    We choose $a_0,b_0$ such that $(x_n)_n\subset [a_0,b_0]$ by boundedness of the sequence. Obviously infinitely many $x_k$ are contained in $[a_0,b_0]$. Now let $a_n$, $b_n$ be given. Define the midpoint $m_n = (a_n+b_n)/2$. Since infinitely many $a_k$ are contained in $[a_n,b_n]$ we know that in at least one of the intervals $[a_n,m_n]$, $[m_n,b_n]$ there have to be infinitely many $x_k$. Thus, set either $a_{n+1} = a_n$ and $b_{n+1} = m_n$ or $a_{n+1} = m_n$ and $b_{n+1} = b_n$ so that infinitely $x_k$ are contained in $[a_{n+1},b_{n+1}]$.

    Now we are in the position to construct a convergent subsequence $(x_{n(k)})_k$ of $(x_n)_n$. Again we proced inductively. Choose $n(0)=0$. Let $(n(k))_{k=0}^m$ be given. Since $[a_{m+1},b_{m+1}]$ contains infinitely many elements of $(x_k)_k$ we can pick, \eg, $n(m+1) = \min\{k\in\N\;|\; k>n(m),\; x_k\in [a_{m+1},b_{m+1}]\}$.

    Lastly, we show that the sequence is indeed convergent. Let $k,\ell\in\N$ arbitrary. Assume without loss of generality that $k\leq \ell$. Since $x_{n(k)}, x_{n(\ell)}\in [a_{k},b_{k}]$ we have 
    \begin{equation}
        |x_{n(k)}-x_{n(\ell)}|\leq |b_{k}-a_{k}| = 2^{-k}
    \end{equation}
    it follows that $(x_{n(k)})_k$ is Cauchy and, thus, convergent.

    For dimesnions greater than one, that is $\R^d$ we may proceed as follows. We can first pick a subsequence $(n_1(k))_k$ such that the sequence of the first entries of the vectors converges. Next we may pick a subsequence of this subsequence $(n_2(k))_k$ such that also the second entries converge. Proceeding like this, the subsequence $(n_d(k))_k$ leads to convergence of all components of the vectors which, in particular, implies convergence in, \eg, the $1$-norm. By equivalence of norms we have convergence in all norms.
\end{proof}

As a consequence we have the following.

\begin{cor}
    In $\R^d$ a set is compact iff it is closed and bounded.
\end{cor}
\begin{proof}
    Assume $A\subset \R^d$ is closed and bounded. Let $(x_n)_n\subset A$. By boundedness and \cref{thm:bolzano} there exists a convergent subsequence and by closedness its limit is in $A$, thus, $A$ is compact. 

    To prove the converse assume $A$ is compact. If $A$ was not bounded, for any $n\in\N$ we could find $x_n\in A$ such that $\|x_n\|\geq n$. Then the sequence $(x_n)_n$ does not admit a convergent subsequence as any subsequence \emph{converges to infinity}. Thus, $A$ has to be bounded. Now let $(x_n)_n\subset A$ be such that $x_n\rightarrow x\in\R^d$. We need to show that $x\in A$. By compactness, $(x_n)_n$ admits a convergent subsequence $(x_{n(k)})_k$, which is \textbf{convergent in $A$}. Any subsequence of a convergent sequence admits the same limit. Thus, $x = \lim_k x_{n(k)}\in A$, concluding the proof.
\end{proof}

\section{Functions}

\begin{defn}[Continuity]
    Let $f:V\rightarrow E$ be a map between two normed spaces. Then we call $f$ continuous in $v\in V$ iff one of the following equivalent conditions is satisfied:
    \begin{itemize}
        \item For every sequence $v_n\rightarrow v$ it holds true that $f(v_n)\rightarrow f(v)$.
        \item For every $\epsilon>0$ there exists $\delta>0$ such that 
        \begin{equation}
            \|u-v\|\leq \delta\Rightarrow \|f(u)-f(v)\|\leq \epsilon.
        \end{equation}
    \end{itemize}
    If $f$ is coontinuos at every $v\in V$ we simply say, $f$ is continuous.
\end{defn}
\begin{defn}[Lipschitz continuity]
    Let $f:V\rightarrow E$ be a map between two normed spaces. Then we call $f$ Lipschitz continuous if there exists $L\geq 0$ such that 
    \begin{equation}
        \|f(u)-f(v)\|\leq L\|u-v\|,\quad u,v\in V
    \end{equation}
\end{defn}

\subsection{Linear maps and the dual space}

\begin{defn}[Linear maps]
    We call a function $f:\R^n\rightarrow \R^m$ linear if 
    \begin{equation}
        f(\lambda x + \mu y) = \lambda f(x) + \mu f(y)\quad x,y\in\R^n, \; \mu,\lambda\in\R.
    \end{equation}
\end{defn}

Indeed, we can identify linear mappings with matrices!
\begin{theorem}
    The space of all linear functions $\R^n\rightarrow \R^m$ is isomorphic\footnote{in layman's terms: the same} to the space of all matrices $A\in \R^{m\times n}$.
\end{theorem}
\begin{proof}
    We can define a one-to-one correspondence via defining for each $f:\R^n\rightarrow\R^m$ a matrix $A_f$ as
    \begin{equation}
        (A_f)_{i,j} = f(e_j)_i.
    \end{equation}
    It is left as an exercise to show that this correspondence is bijective and linear.
\end{proof}

\begin{defn}[Dual space]
    Let $V$ be an arbitrary vector space. The dual space $V^*$ is the \textbf{vector space} of all linear maps $V\rightarrow\R$, \ie,
    \begin{equation}
        V^* = \{f:V\rightarrow \R\;|\; f\text{ is linear}\}.
    \end{equation}
    We refer to elements in $V^*$ as linear functionals on $V$. The canoncical norm on $V^*$ is the dual norm
    \begin{equation}
        \|v^*\|_{V^*} \coloneqq \sup_{\substack{v\in V\\\|v\|_V\leq 1}} v^*(v).
    \end{equation}
    It is common, in a slight abuse of notation, to use the notation of the scalar product also for the application of a dual element, that is, 
    \begin{equation}
        \inner{v^*}{v} \coloneqq v^*(v).
    \end{equation}
\end{defn}

Note that it directly follows that 
\begin{equation}
    |v^*(v)|\leq \|v^*\| \|v\|.
\end{equation}

The following result is extremely important: It states that in a Hilbert space, the dual space can be identified with the space itself! Moreover, it \emph{justifies} the notation of the application of a dual element using the scalar product.

\begin{theorem}[Riesz]
    Let $V$ be a Hilbert space. Then, for every $v^*\in V$ there exists a unique $v\in V$ such that
    \begin{equation}
        v^*(w) = \langle v,w\rangle,\quad w\in V.
    \end{equation}
    The mapping $v^*\mapsto v$ is linear and $\|v\| = \|v^*\|$.
    In particular, $V^*\cong V$. 
\end{theorem}
\begin{proof}
    If $v^*= 0$, $v=0$, so let us assume $v^*\neq 0$. Let $z\in \ker(v^*)^\perp$ be such that $v^*(z)\neq 0$ and $\|z\|=1$. 
    Let $w\in V$ arbitrary. Note that 
    \begin{equation}
        v^*(zv^*(w) - wv^*(z)) = 0
    \end{equation}
    so that 
    \begin{equation}
        \inner{zv^*(w) - wv^*(z)}{z} = 0
    \end{equation}
    by orthogonality of $z$ to $\ker(v^*)$. It follows
    \begin{equation}
        \begin{aligned}
            v^*(w) = v^*(w)\inner{z}{z}
            =& \inner{zv^*(w)}{z}\\
            =& \inner{zv^*(w) - wv^*(z) + wv^*(z)}{z}\\
            =& \inner{zv^*(w) - wv^*(z) + wv^*(z)}{z}\\
            =& \inner{zv^*(w) - wv^*(z)}{z} + \inner{wv^*(z)}{z}\\
            =& \inner{wv^*(z)}{z}\\
            =& \inner{w}{v^*(z)z}
        \end{aligned}
    \end{equation}
    so that $v = v^*(z)z$ is the desired element. The remainder of the proof is left as an exercise.
\end{proof}

\begin{example}[Dual]
    The dual norm of $\|\emptyarg\|_p$ is precisely $\|\emptyarg\|_q$ with $\frac{1}{p}+\frac{1}{q}=1$.
\end{example}

Most readers are probably familiar with the adjoint or transpose of a matrix, which is simply obtained by swapping rows and colums. That is, for $A\in \R^{n\times m}$ we define the adjoint $A^T\in \R^{m\times n}$ via
\begin{equation}
    (A^T)_{i,j} = A_{j,i}.
\end{equation}
One can easily check that with the standard scalar product the adjoint satsifies for any $x\in \R^m$, $y\in \R^n$
\begin{equation}
    \inner{Ax}{y} = \inner{x}{A^Ty}.
\end{equation}
The latter property is in fact the characterizing property of the formal definition of the adjoint of a linear mapping. Note, in particular, \textbf{that the adjoint thus depends on the scalar product}.
\begin{defn}
    Let $A:V\rightarrow W$ be a linear mapping between two inner product spaces. We define the adjoint operations $A^*:W\rightarrow V$ via
    \begin{equation}
        \inner{Ax}{y} = \inner{x}{A^*y},\quad x\in V,\; y\in W
    \end{equation}
\end{defn}
\begin{exercise}
Show that the adjoint is a well-defined linear operator.
\end{exercise}

\begin{lemma}
    The following properties are satisfied for $A,B:V\rightarrow W$ linear and $\lambda\in \R$
    \begin{enumerate}
        \item $(A^*)^* = A$
        \item $(\lambda A + B)^* = \lambda A^* + B^*$
        \item $(AB)^*=B^*A^*$.
    \end{enumerate}
\end{lemma}
\begin{proof}
    Exercise!
\end{proof}

\begin{defn}[Extended real-valued functions]
    Let $V$ be a vector space. We call a function $f:V\rightarrow [-\infty,\infty]$ extended real-valued. Moreover, we call $f$ proper if $f>-\infty$ and there exists $x\in V$ such that $f(x)<\infty$. The domain of a proper extended real-valued function $f$ is defined as the set
    \begin{equation}
        \dom(f) = \{x\in V\;|\; f(x)<\infty \}.
    \end{equation}
\end{defn}

\begin{example}
    A frequently used extended real-valued function is the indicator function. For a set $A\subset V$ the indicator function os defined as 
    \begin{equation}
        \delta_A(x) = 
        \begin{cases}
            0\quad &x\in A\\
            \infty\quad &\text{else.}
        \end{cases}
    \end{equation}
    Using the indicator function we can re-write constraint optimization problems as unconstrained ones via
    \begin{equation}
        \min_{x\in C} f(x) = \min_{x\in V} f(x)+\delta_C(x)
    \end{equation}
\end{example}

\begin{defn}[Epigraph and closed functions]
    The epigraph of a function $f:V\rightarrow (-\infty,\infty]$ is defined as 
    \begin{equation}
        \epi(f) = \{ (x,t)\in V\times \R\; |\; f(x)\leq t\}.
    \end{equation}
    A function is called closed if its epigraph is a closed set.
\end{defn}

\begin{exercise}
    If $A$ is closed then so is $\delta_A$.
\end{exercise}

\begin{exercise}
    Is the domain of a closed function closed? Prove or disprove!
\end{exercise}

\begin{defn}[Lower semi-continuity]
    We call $f:V\rightarrow(-\infty,\infty]$ lower semi-continuous (lsc) iff for every sequence $x_n\rightarrow x$ in $V$ it holds
    \begin{equation}
        f(x)\leq \liminf_{n\rightarrow \infty}f(x_n).
    \end{equation}
\end{defn}

\begin{lemma}
    A function $f:V\rightarrow (-\infty,\infty]$ is closed iff it is lsc.
\end{lemma}
\begin{proof}
    $\Rightarrow$: Assume $f$ is closed. Let $x_n\rightarrow x$. It follows that $(x_n,f(x_n))\in \epi(f)$. Let Now $n(k)$ be a subsequence realizing the liminf, that is,
    \begin{equation}
        \liminf_n f(x_n) = \lim_{k} f(x_{n(k)})
    \end{equation}
    It follows that 
    \begin{equation}
        (x_{n(k)},f(x_{n(k)}))\rightarrow (x,\liminf_n f(x_{n}))
    \end{equation}
    as $k\rightarrow \infty$. By closedness we have $(x,\liminf_n f(x_{n}))\in \epi(f)$, that is $f(x)\leq \liminf_n f(x_{n})$. 

    $\Leftarrow$: Assume now $f$ is lsc and let $(x_n,t_n)_n\subset \epi(f)$ with $(x_n,t_n)\rightarrow (x,t)$ in $V\times \R$. In particular, it follows $x_n\rightarrow x$. By lower semicontinuity we have
    \begin{equation}
        f(x)\leq \liminf_n f(x_n) \leq \liminf_n t_n = \lim_n t_n = t
    \end{equation}
    where the second inequality follows from $(x_n,t_n)\in \epi(f)$. This implies $f(x)\leq t$ and, hence, $(x,t)\in \epi(f)$ concluding the proof.
\end{proof}

\begin{lemma}
    If $f$ is continuous and $\dom(f)$ is closed, then $f$ is closed.
\end{lemma}
\begin{proof}
    Exercise.
\end{proof}

\begin{figure}
    \centering
    \begin{tikzpicture}
        \node at (0,0) {\includegraphics[width=0.4\textwidth]{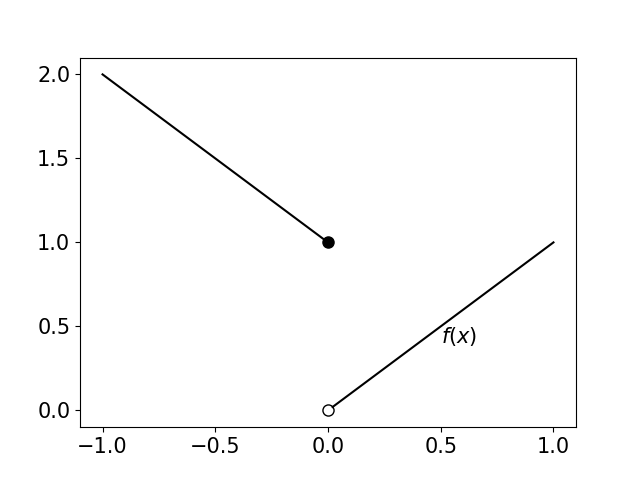}};
        \node at (7,0) {\includegraphics[width=0.4\textwidth]{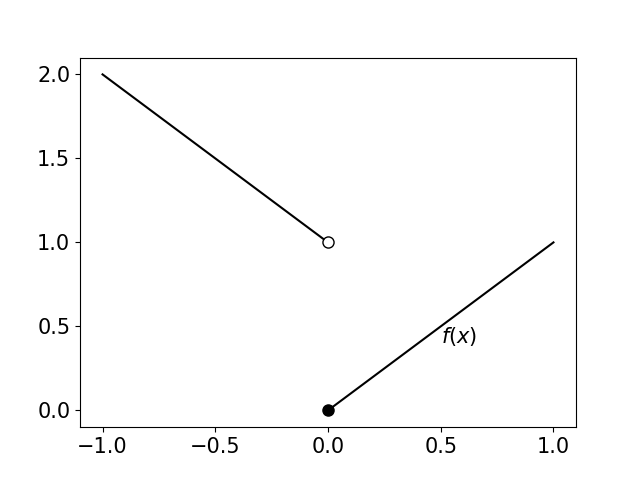}};
        \node at (0,-3) {No minimum attained};
        \node at (7,-3) {Minimum attained};
    \end{tikzpicture}
    \caption{Lack of lower-semicontinuity may lead to infima not being attained!}
    \label{prelim:fig:lsc}
\end{figure}

\begin{defn}[Coercive]
    We call $f:V\rightarrow \R$ coercive if 
    \begin{equation}
        \lim_{\|x\|\rightarrow\infty} f(x) = \infty.
    \end{equation}
\end{defn}

\begin{theorem}\label{preliminaries:thm:direct_method}
    Let $f:V\rightarrow (-\infty,\infty]$ be proper, coercive, and lsc and $V$ be finite-dimensional\footnote{This is not necessary but assumed to avoid more complicated arguments.}. Then $f$ admits a minimium on $V$.
\end{theorem}
\begin{proof}
    The proof follows the so-called \emph{direct method}.
    \paragraph{Step 1: Take a minimizing sequence}
    Let $(x_n)_n$ be a minimizing sequence\footnote{As an exercise: Why does such a sequence always exist?}, that is, 
    \begin{equation}
        \lim_n f(x_n) = \inf_{x\in V} f(x).
    \end{equation}
    \paragraph{Step 2: Bounedness of the sequence}
    Since $f$ is coercive, the minimizing sequence must be bounded. Indeed, assume the contrary. Then we can find a subsequence $(x_{n(k)})_k$ such that $\|x_{n(k)}\|\rightarrow\infty$. By coercivity, this implies
    \begin{equation}
        f(x_{n(k)})\rightarrow\infty
    \end{equation}
    which is a contradiction to
    \begin{equation}
        f(x_{n(k)})\rightarrow \inf f<\infty.
    \end{equation}
    \paragraph{Step 3: Apply a copmpactness argument to extract a convergent subsequence}
    From \cref{thm:bolzano} we know that every bounded sequence in a finite dimensional vector space admits a convergent subsequence. Thus, we can take such a subsequence $(x_{n(k)})_k$ with limit $\lim_kx_{n(k)}=\hat x$.
    \paragraph{Step 4: Conclude using lsc}
    By lsc of $f$ we obtain
    \begin{equation}
        f(\hat x) \leq \liminf_{k\rightarrow \infty}x_{n(k)} = \lim_n f(x_n) = \inf_x f(x) \leq f(\hat{x}).
    \end{equation}
\end{proof}

\chapter{Convexity}

Our goal in this lecture is to solve problems of the form 
\begin{equation}
    \min_{x\in C} f(x)
\end{equation}
where $C\subset V$ is a \emph{convex subset} of the space $V$ and $f:C\rightarrow (-\infty,\infty]$ is a \emph{convex function}. Therefore, in the following we will analyze in more detail what convexity of sets and functions precisely means and what consequences we can draw from those properties.

\section{Convex sets}
A set is called convex if for any two points within the set, the entire straight line connecting the points as contained in the set as well. More formally:
\begin{defn}
    A set $C\subset V$ is called convex iff for every $x,y\in C$, $\lambda\in (0,1)$ it holds
    \begin{equation}
        \lambda x + (1-\lambda) y\in C.
    \end{equation}
\end{defn}

Note that
\begin{equation}
    \lambda x + (1-\lambda) y = y + \lambda(x-y)
\end{equation}
which shows more clearly that $\lambda x + (1-\lambda) y$ is the line connecting $x$ and $y$.

Moreover, we can easily show the following using more general \emph{convex combinations}.

\begin{lemma}
    The set $C\subset V$ is convex if and only if for every $n\in \N$, $x_i\in C$, $\lambda_i\in [0,1]$, $i=1,\dots,n$ with $\sum_{i=1}^n\lambda_i=1$ it holds
    \begin{equation}
        \sum_{i=1}^n\lambda_i x_i \in C.
    \end{equation}
\end{lemma}
\begin{proof}
    Exercise.
\end{proof}

For notational convenience we define the \emph{unit simplex} as 
\begin{equation}
    \Delta^k = \{(\lambda_1,\dots,\lambda_k)\in [0,1]\;|\;\sum_{i=1}^n\lambda_i = 1\}
\end{equation}
We may now define the smalles convex set containing a given set, the \emph{convex hull}.

\begin{defn}[Convex hull]
    Let $C\subset V$. The convex hull of $C$ is defined as the set
    \begin{equation}\label{eq:convex_sets:defin_conv_hull}
        \conv(C) \coloneqq \bigcap_{\substack{S\subset V \text{ convex}\\C\subset S}}S
    \end{equation}
\end{defn}

\begin{lemma}
    It holds true that
    \begin{equation}\label{eq:convex_sets:conv_hull}
        \conv(C) = \bigg\{ \sum_{i=1}^n\lambda_i x_i\;|\; n\in \N,\;x_i\in C, \; \lambda\in \Delta^n\bigg\}
    \end{equation}
\end{lemma}
\begin{proof}
    Let us denote the right-hand side of \eqref{eq:convex_sets:conv_hull} as $A$.
    We will show that $\conv(C)\subset A$ and $A\subset\conv (C)$.

    \underline{$A\subset\conv (C)$:}
    Let $x=\sum_{i=1}^n\lambda_i x_i\in A$. Note that for any $S\subset V$ such that $S$ is convex and $C\subset S$ it holds that $x\in S$. Therefore, $x\in \conv(C)$ by its definition in~\eqref{eq:convex_sets:defin_conv_hull}.

    \underline{$\conv (C)\subset A$:} 
    One may easily check that $A$ is, indeed, a convex set. Moreover $C\subset A$. Therefore, $A$ appears as one of the sets $S$ in~\eqref{eq:convex_sets:defin_conv_hull} and, thus, $\conv(C)\subset A$.
\end{proof}

\begin{figure}
    \centering
    \includegraphics[width=.7\textwidth]{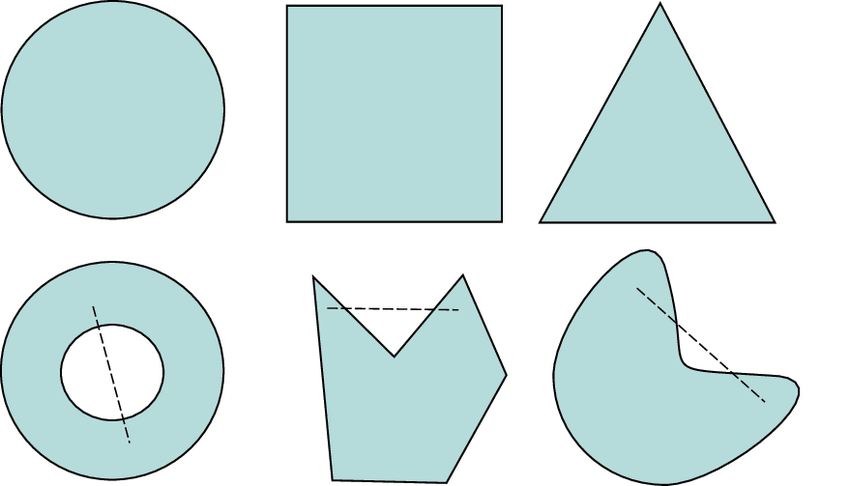}
    \caption{Top row: convex sets. Bottom row: non-convex sets.}
\end{figure}

\begin{example}
    The following are convex sets (the proofs are left as exercises):
    \begin{itemize}
        \item The empty set $\emptyset$
        \item Every vector space
        \item Every norm ball, that is, for every norm $\|\emptyarg\|$ and every $c\in V$ and $R\geq 0$ the sets
        \begin{equation}
            B_{\|\emptyarg\|,R}(c)\{x\in V\;|\; \|x-c\|< R\}
        \end{equation}
        and
        \begin{equation}
            \overline{B}_{\|\emptyarg\|,R}(c)\{x\in V\;|\; \|x-c\|\leq R\}
        \end{equation}
        \item Every affine subspace, that is, for any set of vectors $z, x_1,\dots,x_k\in V$ the set
        \begin{equation}
            \bigg\{ x\in V\;|\; x = z + \sum_{i=1}^k\lambda_i x_i,\; \lambda_i\in \R \bigg\}
        \end{equation}
    \end{itemize}
\end{example}

\begin{figure}
    \centering
	\includegraphics[width=1\textwidth]{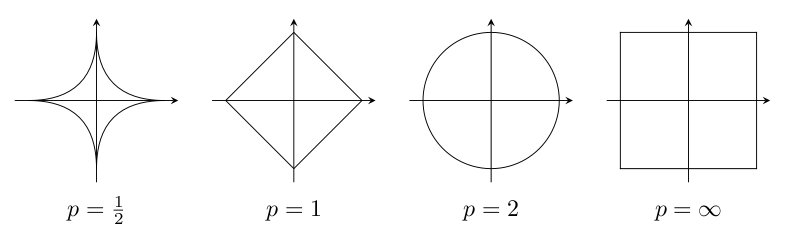}
    \caption{Norm balls for $\overline{B}_{\|\emptyarg\|_p,R}(0)$ for different $\ell^p$ norms. Note that the set is not convex for $p=1/2$ as this choise does not lead to a norm (no triangle inequality).}
\end{figure}

\begin{lemma}[Operations that preserve convexity]\label{convexity:lemma:ops preserving convexity}
    The following hold true:
    \begin{enumerate}
        \item Intersection. Let $C_i\subset V$ be convex for $i\in I$ with $I$ an arbitrary index set. Then the intersection
        \begin{equation}
            \bigcap_{i\in I} C_i \coloneqq \{x\in V\;|\; x\in C_i\; \forall i\in I\}
        \end{equation}
        is convex.
        \item Weighted sum. Let $C_i\subset V$ be convex for $i=1,\dots, N$. Moreover let $\lambda_i \in \R$ for $i=1,\dots, N$. Then the set
        \begin{equation}
            \sum_{i=1}^n \lambda_i C_i \coloneqq \bigg\{\sum_{i=1}^N \lambda_i x_i\;|\; x_i\in C_i\; \forall i\in I\bigg\}
        \end{equation}
        is convex.
        \item Cartesian product sum. Let $C_i\subset V$ be convex for $i\in I$ with $I$ an arbitrary index set. Then the cartesian product
        \begin{equation}
            \bigotimes_{i\in I} C_i \coloneqq \{(x_i)_{i\in I}\;|\; x_i\in C_i\; \forall i\in I\}
        \end{equation}
        is convex.
        \item Images and preimages\footnote{The preimage is well-defined also for non-invertible maps!} of linear maps are convex, that is, for $A:V\rightarrow W$ linear and $S\subset V$, $T\subset W$ both convex the following sets are convex:
        \begin{equation}
            A(S) = \{Ax \;|\; x\in S\},\quad A^{-1}= \{y\in V\;|\; Ax \in T\}
        \end{equation}
    \end{enumerate}
\end{lemma}
\begin{proof}
    We only prove the first part and leave the remainder as an exercise. The proof is straight-forward. Let $x,y\in \bigcap_{i\in I} C_i$ and $\lambda\in (0,1)$. Since $x,y\in C_i$ for every $i$ and each $C_i$ is convex, it follows that $\lambda x + (1-\lambda)y \in C_i$ for every $i$ and, therefore, $\lambda x + (1-\lambda)y \in \bigcap_{i\in I} C_i$.
\end{proof}

We want to formally define the notion of a hyperplane, that is, the generalization to arbitrary dimensions of a 2D plane in $\R^3$.

\begin{defn}[Hyperplanes]
    Let $V$ be an inner product space, a hpyerplane is a set $H$ of the form
    \begin{equation}
        H = \{x\in V\;|\; \inner{x}{a}=b\}
    \end{equation}
    for some $a\in V\setminus\{0\}$, $b\in\R$.
\end{defn}

The parameter $a$ defines the angle of the plane and $b$ its offset. In particular, for $0\in H$ iff $b=0$ in which case $H$ is a subspace. Whenever $b\neq 0$, $H$ is merely an affine subspace.

Note that if $x_0\in H$ it holds that
\begin{equation}
    H = \{x\in V\;|\; \inner{x-x_0}{a}=0\}.
\end{equation}
It follows that $a$ is the normal vector to the plane. Moreover one can check (exercise) that $\frac{|b|}{\|a\|}$ is the distance of the plane from the origin.

Note that every hyperplane separates the space into two halfs. 
\begin{defn}[Halfspaces]
    Let $V$ be an inner product space, a halfspace is a set $H^-$ of the form
    \begin{equation}
        H^- = \{x\in V\;|\; \inner{x}{a}\leq b\}
    \end{equation}
    for some $a\in V\setminus\{0\}$, $b\in\R$.
\end{defn}

The following theorem, which states that this separation can be chosen in a particular manner, is in fact surprisingly crucial within optimization and functional analysis!

\begin{theorem}[Hahn-Banach seperation theorem]
    Let $S,T$ be two nonempty, disjoint, convex subsets of a vector space $V$ with $S$ open. Then there exists $a\in V^*$, $a\neq 0$ and $\alpha\in \R$ such that
    \begin{equation}
        \inner{x}{a}<\inner{y}{a},\quad x\in S, \; y\in T.
    \end{equation}
    If $S,T$ are both closed and one of them is compact, there exists $\alpha\in \R$ and $\epsilon>0$ such that 
    \begin{equation}
        \inner{x}{a}\leq \alpha\leq \alpha+\epsilon\leq \inner{y}{a},\quad x\in S, \; y\in T.
    \end{equation}
\end{theorem}
The second version allows to find a strictly positive difference between 
\[
    \sup_{x\in S}\inner{x}{a}
    \quad 
    \text{and}
    \quad
    \inf_{y\in T}\inner{y}{a}.
\]

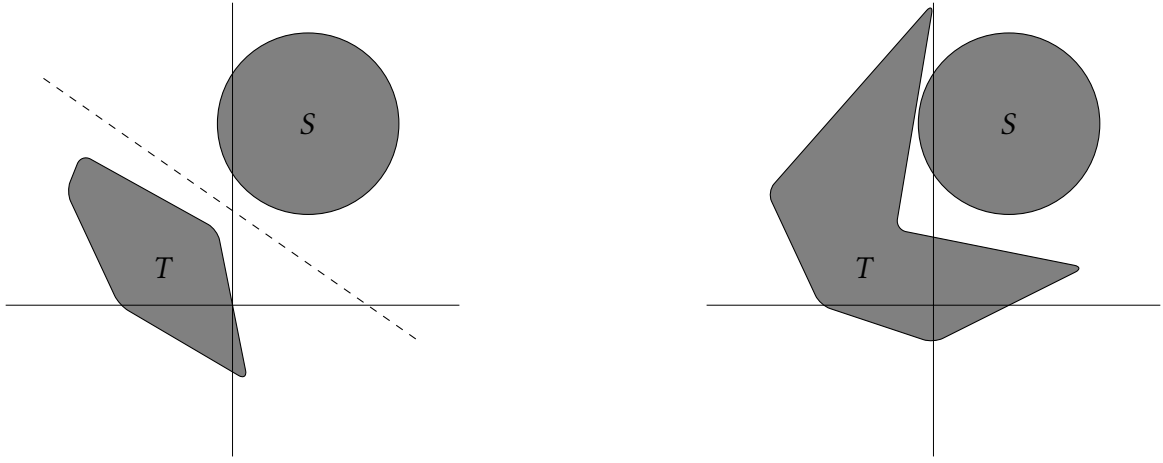
\begin{figure}
    \centering
			\begin{tikzpicture}[font=\sffamily]
				\draw[dashed] (-2.5,3) -- (2.5,-0.5);
				\node[circle,draw,fill=gray,minimum size=2.4cm] at (1,2.4) {$S$};
				\draw[fill=gray,rounded corners] (-2,2) -- (-0.2,1) -- (0.2,-1) -- (-1.5,0) -- (-2.2,1.5) --cycle;
				\path (-2,2) --  (0.2,-1)  node[midway] {$T$};
				\draw[ultra thin] (-3,0) -- (3,0) (0,-2) -- (0,4);
			\end{tikzpicture}\hfill%
			\begin{tikzpicture}[font=\sffamily]
				\node[circle,draw,fill=gray,minimum size=2.4cm] at (1,2.4) {$S$};
				\draw[fill=gray,rounded corners] (0,4) -- (-0.5,1) -- (.5,.8) -- (2,.5) -- (0,-.5) -- (-1.5,0) -- (-2.2,1.5) --cycle;
				\path (-2,2) --  (0.2,-1)  node[midway] {$T$};
				\draw[ultra thin] (-3,0) -- (3,0) (0,-2) -- (0,4);
			\end{tikzpicture}
    \caption{Illustration of the Hahn-Banach seperation theorem and why it may fail without convexity.}
\end{figure}

\section{Convex functions}
As shown in~\cref{fig:convexity:1d_example} convexity has crucial impacts on the questions of existence and uniqueness of minima.
\begin{figure}
    \centering
    \includegraphics[width=0.5\textwidth]{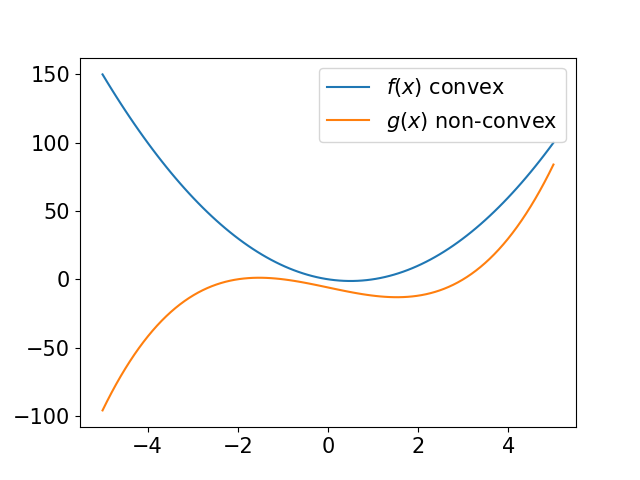}
    \caption{Example of a convex vs a non-convex function in 1D.}
    \label{fig:convexity:1d_example}
\end{figure}

\begin{defn}[Zero-order condition for convexity]
    Let $C\subset V$ be convex and $f:C\rightarrow (-\infty,\infty]$.
    \begin{itemize}
        \item We call $f$ convex iff for any $x,y\in C$ and $\lambda\in (0,1)$ it holds
        \begin{equation}
            f(\lambda x + (1-\lambda)y)\leq \lambda f( x ) + (1-\lambda)f(y).
        \end{equation}
        \item We call $f$ strictly convex iff for any $x,y\in C$ and $\lambda\in (0,1)$ it holds
        \begin{equation}
            f(\lambda x + (1-\lambda)y) < \lambda f( x ) + (1-\lambda)f(y).
        \end{equation}
        \item We call $f$ $\mu$-strongly convex iff for any $x,y\in C$ and $\lambda\in (0,1)$ it holds
        \begin{equation}
            f(\lambda x + (1-\lambda)y) \leq \lambda f( x ) + (1-\lambda)f(y) - \frac{\lambda(1-\lambda)\mu}{2}\|x-y\|^2
        \end{equation}
    \end{itemize}
\end{defn}
The zero-order condition states that the straight line connecting to points on the graph (the secant) is always above the function. In the case of a strictly convex function the secant is strictly above the function and in the strongly convex case we may even \emph{squeeze} a quadratic inbetween.

\begin{figure}
    \centering
    \begin{tikzpicture}[inner sep=.5mm,
			declare function={
				f(\x)= (\x-3)^2+3*\x-4;
			}
			]
			\begin{axis}[axis lines=middle,domain=-1:4,samples=200,xtick={0.5,3},xticklabels={$x$,$y$},ytick=\empty,xmin=-.2,xmax=3.7,ymin=-.1,ymax=10]
				\addplot[blue,very thick,name path=f] {f(x)};
				\pgfplotsinvokeforeach{.5,3}{
					\path[name path=x] (#1,0) -- (#1,6);
					\draw[black!50,thick,dashed,name intersections={of=f and x, name=i}] (#1,0) -- (i-1);
				}
				\draw[very thick] ($(.5,{f(.5)})$) node[circle,fill=blue,label=left:{$f(x)$}]{} -- ($(3,{f(3)})$) node[circle,fill=blue,label=right:{$f(y)$}]{};
				\node[text=blue] at (1.75,2) {$f(\lambda x+(1-\lambda)y)$};
				\node[rotate=10] at (1.75,5) {$\lambda f(x)+(1-\lambda)f(y)$};
			\end{axis}
		\end{tikzpicture}
        \caption{Geometric interpretation of the zero-order condition for convexity of a function.}
\end{figure}
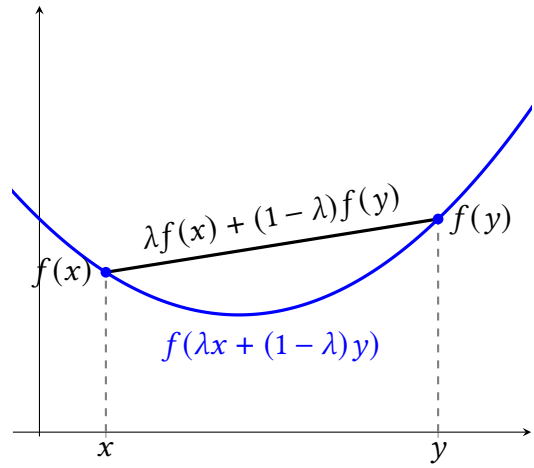

\begin{example}
    The following are convex (proof: exercise):
    \begin{enumerate}
        \item Affine functions $f(x)=\langle y, x\rangle + b$ with $a\in V$, $b\in \R$ (in particular, linear functions for $b=0$).
        \item All norms are convex functions.
    \end{enumerate}
\end{example}

Similar to the corresponding result for convex sets we can extend convexity to arbitrarily \emph{large} convex combinations.

\begin{theorem}[Jensen]
    Let $C\subset V$ be convex and $f:C\rightarrow \bR$. Then $f$ is convex iff for any $k\in \N$, $\lambda\in \Delta^k$, and $x_i\in C$ it holds
    \begin{equation}
        f(\sum_{i=1}^k\lambda_i x_i)\leq \sum_{i=1}^k\lambda_i f(x_i)
    \end{equation}
\end{theorem}

We have seen above that convexity means that the secant stays above the function. It also means, that the tangent stays below the function (\cf~\cref{convexity:fig:gradients}). 

\begin{theorem}[First-order characterization of convexity]\label{convexity:thm:first_order}
    Let $f:C\rightarrow\R$ continuously differentiable with $C\subset\R^d$ convex.
    \begin{itemize}
        \item $f$ is convex iff for all $x,y\in C$
        \[
            f(x) + \langle \nabla f(x),y-x\rangle \leq f(y).
        \]
        \item $f$ is strictly convex iff the above inequality holds strict for all $x\neq y$.
    \end{itemize}
\end{theorem}
\begin{proof}
    We only show the first item as the second follows trivially.
    
    $\Rightarrow$: Let $f$ be convex. then we have
    \begin{equation}
        f(x+\lambda (y-x)) = f(\lambda y + (1-\lambda)x)\leq \lambda f(y) + (1-\lambda) f(x).
    \end{equation}
    Rearranging yields
    \begin{equation}
        \frac{f(x+\lambda (y-x))-f(x)}{\lambda} \leq f(y) - f(x).
    \end{equation}
    Letting $\lambda\rightarrow 0$ yields the result.

    $\Leftarrow$: Assume now the gradient condition holds true. Let us denote $x_\lambda = x+\lambda (y-x)$
    \begin{equation*}
        f(x_\lambda) + \langle \nabla f(x_\lambda), y-x_\lambda\rangle = f(x_\lambda) + \langle \nabla f(x_\lambda), (1-\lambda) (y-x)\rangle \leq f(y)
    \end{equation*}
    and
    \begin{equation*}
        f(x_\lambda) + \langle \nabla f(x_\lambda), x-x_\lambda\rangle = f(x_\lambda) + \langle \nabla f(x_\lambda), -\lambda(y-x)\rangle \leq f(x).
    \end{equation*}
    Multiplying the first inequality with $\lambda$, the second with $1-\lambda$ and adding both leads to
    \begin{equation*}
        f(x_\lambda) \leq \lambda f(y) +(1-\lambda) f(x).
    \end{equation*}
\end{proof}

\begin{figure}
    \centering
    \begin{tikzpicture}[inner sep=.5mm,
			declare function={
				f(\x)=  (\x-3)^2+3*\x-4;
			}
			]
			\begin{axis}[axis lines=middle,domain=-1:4,samples=200,xtick={2},xticklabels={$x$},ytick=\empty,xmin=-.2,xmax=3.7,ymin=-.1,ymax=10]
				\addplot[blue,very thick,name path=f] {f(x)};
				\pgfplotsinvokeforeach{2}{
					\path[name path=x] (#1,0) -- (#1,6);
					\draw[black!50,thick,dashed,name intersections={of=f and x, name=i}] (#1,0) -- (i-1);
				}
				\addplot[very thick] {f(2)+1*(x-2)};
				\node[circle,fill=blue,label=above:{$f(x)$}] at ($(2,{f(2)})$) {};
				\node[text=blue,yshift=5mm] at ($(.5,{f(.5)})$) {$f$};
			\end{axis}
			\node[rotate=17,yshift=-3mm] at (6, 2.5) {$\textcolor{black}{y \mapsto f(x)+\inner{\nabla f(x)}{y-x}}$};
		\end{tikzpicture}
        \caption{Geometric interpretation of a convex function.}
        \label{convexity:fig:gradients}
\end{figure}
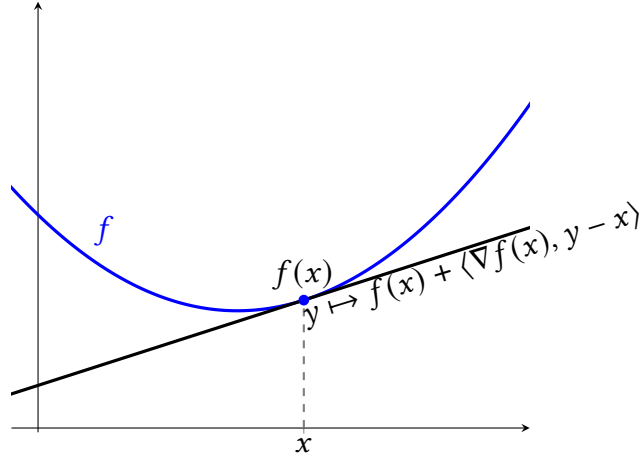

\begin{defn}[Minimum]
    Let $f:C\rightarrow \R$. 
    \begin{itemize}
        \item We call $x^*$ a local minimum iff there exists $\epsilon>0$ such that 
        \[
            f(x^*)\leq f(x),\quad \forall x\in B_\epsilon(x^*)
        \]
        \item We call $x^*$ a global minimum iff
        \[
            f(x^*)\leq f(x),\quad \forall x\in C
        \]
    \end{itemize}
    The minimum is called \emph{strict} if the above inequalities hold strict.
\end{defn}

We are now in the position to prove a first crucial result of convex optimization.
\begin{prop}\label{convexity:prop:OC1}
    Let $f:C\rightarrow \R$ continuously differentiable over a convex set. If $\nabla f(x^*)=0$ then $x^*$ is a \textbf{global} minimizer.
\end{prop}
\begin{proof}
    By the first-order characterization of convexity we have
    \begin{equation}
        f(x)\geq f(x^*) + \langle \nabla f(x^*),x-x^*\rangle
        = f(x^*) + 0.
    \end{equation}
\end{proof}

Note, that the converse of~\cref{convexity:prop:OC1}, that is, the implication
\begin{equation}
    \text{$x^*$ is minimum} \Rightarrow \nabla f(x^*) = 0
\end{equation}
is in general only true if $x^*$ is contained in the interior of $C$ denoted as $C^\circ$, that is, there exists $\epsilon>0$ such that $B_\epsilon(x^*)\subset C$.

\begin{theorem}
    Let $f:C\rightarrow \R$ continuously differentiable over a convex set and $x^*\in C^\circ$. Then $\nabla f(x^*)=0$ iff $x^*$ is a global minimizer.
\end{theorem}
\begin{proof}
    One direction follows from~\cref{convexity:prop:OC1}. For the other direction, assume that $x^*$ is a minimum. Therefore, for any $v\in\R^d$, if $t>0$ is sufficiently small, it follows by optimality
    \begin{equation}
        f(x^*+tv) - f(x^*)\geq 0.
    \end{equation}
    Dividing by $t$, since $t>0$ yields
    \begin{equation}
        \frac{f(x^*+tv) - f(x^*)}{t}\geq 0.
    \end{equation}
    Letting $t\rightarrow 0$ yields 
    \begin{equation}
        \langle\nabla f(x^*), v\rangle\geq 0.
    \end{equation}
    Since $v$ was arbitrary, we may repate the same argument with $-v$ leading to $\langle\nabla f(x^*), v\rangle\leq 0$ and, therefore, in total to $\langle\nabla f(x^*), v\rangle = 0$. Again, since $v$ was arbitrary, we have $\nabla f(x^*)=0$ (why?).
\end{proof}

You may remember that in 1D, if the function is differentiable., convexity can be characterized by monotonicity of the derivative. In arbitrary dimensions we have the following result.

\begin{theorem}[Monotonicity of the gradient]
    Let $f:C\rightarrow\R$ be continuously differentiable over a convex set $C$. Then $f$ is convex iff
    \begin{equation}
        \langle\nabla f(x) - \nabla f(y),x-y\rangle\geq 0.
    \end{equation}
\end{theorem}
\begin{proof}
    $\Rightarrow$: Assume $f$ is convex. By the first-order characterization of convexity we have 
    \begin{equation}
        \begin{aligned}
            &f(x) + \langle\nabla f(x), y-x\rangle\leq f(y),\\
            &f(y) + \langle\nabla f(y), x-y\rangle\leq f(x).
        \end{aligned}
    \end{equation}
    Adding both equation yields the result.

    $\Leftarrow$: Assume $\nabla f$ is monotone. By the fundamental theorem of calculus we have
    \begin{equation}
        \begin{aligned}
            f(y) - f(x) 
            &= \int_0^1 \frac{\dd}{\dd t}f(x+t(y-x))\dd t\\
            &= \int_0^1 \langle\nabla f(x+t(y-x)), y-x\rangle\dd t\\
            &= \int_0^1 \langle\nabla f(x+t(y-x))-\nabla f(x), y-x\rangle\dd t + \langle\nabla f(x), y-x\rangle\\
            &= \int_0^1 \frac{1}{t}\underbrace{\langle\nabla f(x+t(y-x))-\nabla f(x), t(y-x)\rangle}_{\geq 0}\dd t + \langle \nabla f(x), y-x\rangle\\
            &\geq \langle \nabla f(x), y-x\rangle
        \end{aligned}
    \end{equation}
\end{proof}

Finally, we can also use the second derivative to characterize convexity. Intuitivelly, convexity means \emph{positive curvature}. To formalize this, we recall, how matrices are \emph{ordered}. For two matrices $A,B\in \R^{d\times d}$ we write
\begin{equation}
    A\succeq B \Longleftrightarrow \langle x, A x\rangle \geq \langle x, B x\rangle,\; \forall x\in\R^d
\end{equation}
and, similarly,
\begin{equation}
    A\succ B \Longleftrightarrow \langle x, A x\rangle > \langle x, B x\rangle,\; \forall x\in\R^d.
\end{equation}

\begin{theorem}[Second order characterization of convexity]
    Let $f:C\rightarrow \R$ be twice continuously differentiable over a convex set $C$. 
    \begin{enumerate}
        \item $f$ is convex iff $\nabla^2 f(x)\succeq 0$.
        \item $f$ is strictly convex iff $\nabla^2 f(x)\succ 0$.
    \end{enumerate}
\end{theorem}
\begin{proof}
    We only proof the first assertion.

    $\Rightarrow$: Assume $f$ is convex. By monotonicity of the gradient it follows for $t>0$
    \begin{equation}
        \frac{\inner{\nabla f(x+tv)-\nabla f(x)}{v}}{t}\geq 0.
    \end{equation}
    Letting $t\rightarrow 0$ it follows 
    \begin{equation}
        \langle v, \nabla^2 f(x)v\rangle\geq 0.
    \end{equation}
    Since $v$ was arbitrary, the result follows.

    $\Leftarrow$: Assume $\nabla^2 f\succeq 0$. By the fundamental theorem of calculus we have
    \begin{equation}
        \begin{aligned}
            \nabla f(y) - \nabla f(x)
            =& \int_0^1 \frac{\dd}{\dd t} \nabla f(x+t(y-x))\dd t\\
            =& \int_0^1 \nabla^2 f(x+t(y-x))(y-x)\dd t.
        \end{aligned}
    \end{equation}
    Multiplying by $y-x$ yields
    \begin{equation}
        \begin{aligned}
            \langle \nabla f(y) - \nabla f(x),y-x\rangle
            =& \int_0^1 \frac{\dd}{\dd t} \nabla f(x+t(y-x))\dd t\\
            =& \int_0^1 \underbrace{\langle y-x,\nabla^2 f(x+t(y-x))(y-x)\rangle}_{\geq 0}\dd t\\
            \geq& 0.
        \end{aligned}
    \end{equation}
\end{proof}


\begin{lemma}[Operations that preserve convexity]
    Let $f, f_1,\dots, f_n:C\rightarrow\R$ be convex on a convex set. 
    \begin{enumerate}
        \item For any $\alpha\geq 0$, $\alpha f$ is convex.
        \item $\sum_i f_i$ is convex.
    \end{enumerate}
\end{lemma}
\begin{proof}
    Exercise!
\end{proof}

\begin{lemma}[Operations that preserve convexity cont'd]
    Let $f:\R^n\supset C\rightarrow\R$ be convex and $A\in \R^{n\times m}$, $b\in \R^n$, then 
    \begin{equation}
        g(x) = f(Ax+b)
    \end{equation}
    is convex on $\{x\in \R^m\;|\; Ax+b\in C\}$.
\end{lemma}
\begin{proof}
    Exercise!
\end{proof}

\begin{lemma}[Operations that preserve convexity cont'dd]
    Let $f:C\rightarrow\R$ be convex and $g:I\rightarrow \R$ be convex and non-decreasing with $I\subset\R$ an interval such that $f(C)\subset I$. Then $g\circ f$ is convex.
\end{lemma}
\begin{proof}
    Exercise!
\end{proof}

\begin{example}
    \begin{enumerate}
        \item The logsumexp function is convex
        \begin{equation}
            f(x) = \log\bigg(\sum_{i=1}^n e^{x_i}\bigg).
        \end{equation}
        \item The quadratic-over-linear function
        \begin{equation}
            f(x) = \frac{x_1^2}{x_2}
        \end{equation}
        is convex on $\R\times (0,\infty)$.
        \item The function $h(x) = e^{\|x\|^2}$ is convex over $\R^d$ as $f(x) = \|x\|^2$ is convex and $g(t)=e^t$ is convex and non-decreasing.
		\item Consider the $h(x) = (\|x\|^2+1)^2$ over $\R^d$. We have $f(x)=\|x\|^2+1$ and $g(t)=t^2$ are convex. The function \( g \) not non-decreasing on $\R$ but it is non-decreasing on \( f(\R^n) = [1, \infty) \). Consequently, \( h \) is convex.
        \item In the previous example, if we replace \( f \) with \( \|\emptyarg \|^2 - 1 \), then \( g \) is not non-decreasing on \( f(\R^n) = [-1, \infty) \), \cf~\cref{convexity:fig:compositions}.
    \end{enumerate}
\end{example}

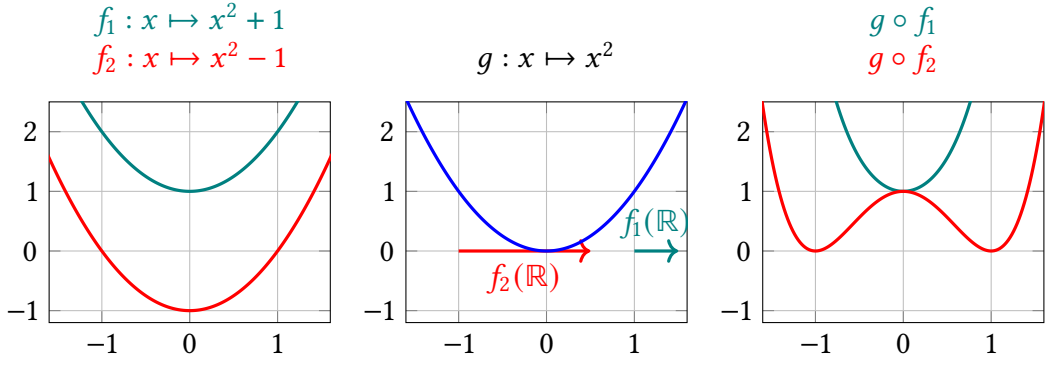
\begin{figure}
    \centering
    \begin{tikzpicture}[declare function={
			f1(\x) =  \x^2 + 1;
			f2(\x) =  \x^2 - 1;
			g(\x) = \x^2;
			h1(\x) = g(f1(\x));
			h2(\x) = g(f2(\x));
		}
		]
		\begin{groupplot}[
			group style={group size= 3 by 1,vertical sep=1.5cm}, height=4.5cm, width=5.3cm,ymin=-1.2,ymax=2.5,domain=-2:2,xmin=-1.6,xmax=1.6,samples=200,grid
		]
			\nextgroupplot[
				align=center,
				title={%
					\( \textcolor{teal}{f_1 : x \mapsto x^2 + 1} \) \\
					\( \textcolor{red}{f_2 : x \mapsto x^2 - 1} \)
				}
			]
				\addplot[teal, very thick] {f1(x)};
				\addplot[red, very thick] {f2(x)};
			\nextgroupplot[title={\( g : x \mapsto x^2\)}]
				\draw[->, teal, very thick] (axis cs:1, 0) -- (axis cs:1.5, 0) node [midway, above] {\( \textcolor{teal}{f_1(\R)} \)};
				\draw[->, red, very thick] (axis cs:-1, 0) -- (axis cs:.5, 0) node [midway, below] {\( \textcolor{red}{f_2(\R)} \)};
				\addplot[blue, very thick] {g(x)};
			\nextgroupplot[
				align=center,
				title={%
					\( \textcolor{teal}{g \circ f_1 } \) \\
					\( \textcolor{red}{g \circ f_2 } \)
				}
			]
				\addplot[teal, very thick] {h1(x)};
				\addplot[red, very thick] {h2(x)};
		\end{groupplot}
	\end{tikzpicture}
    \caption{Convexity of compositions.}
    \label{convexity:fig:compositions}
\end{figure}

\begin{lemma}[Operations that preserve convexity cont'ddd]
    Let $f_i:C\rightarrow\R$, $i=1,\dots, n$ be convex over the convex set $C$. Then
    \begin{equation}
        f(x) =\max_{i=1,\dots,n} f_i(x)
    \end{equation}
    is convex.
\end{lemma}
\begin{proof}
    We simply compute
    \begin{equation}
        \begin{aligned}
            f(\lambda x + (1-\lambda)y)
            = \max_{i}f_i(\lambda x + (1-\lambda)y)
            \leq& \max_{i}\lambda f_i(x) + (1-\lambda) f_i(y)\\
            \leq& \max_{i}\bigg\{\lambda\max_{j} f_j(x) +  (1-\lambda)\max_jf_j(y)\bigg\}\\
            =& \max_{i}\bigg\{\lambda f(x) + (1-\lambda) f(y)\bigg\}\\
            =& \lambda f(x) + (1-\lambda) f(y)
        \end{aligned}
    \end{equation}
\end{proof}

\begin{theorem}\label{convexity:thm:minimum}
    Let $f:C\times D\rightarrow \R$ be convex over $C\times D$ where both $C$ and $D$ are convex. Define
    \begin{equation}
        g(x) = \inf_{y\in D}f(x,y)
    \end{equation}
    where we asume the above infimum is finite for all $x\in C$. Then $g$ is convex.
\end{theorem}
\begin{proof}
    Let $y_n^x$ and $y_n^z$ be such that 
    \begin{equation}
        g(x) = \lim_n f(x,y_n^x),\quad g(z) = \lim_n f(z,y_n^z)
    \end{equation}
    It follows
    \begin{equation}
        \begin{aligned}
            g(\lambda x + (1-\lambda)z)
            \leq f(\lambda x + (1-\lambda)z, \lambda y_n^x + (1-\lambda)y_n^z)
            \leq \lambda f(x, y_n^x) + (1-\lambda) f(z,y_n^z)
        \end{aligned}
    \end{equation}
    Taking the limit as $n\rightarrow \infty$ concludes the proof.
\end{proof}

\begin{example}
    The distance of a point to a set is a convex function, \ie, for any $C\subset\R^d$ the following map is convex,
    \begin{equation}
        \begin{aligned}
            x\mapsto d(x,C) \coloneqq \inf_{y\in C}\|x-y\|.
        \end{aligned}
    \end{equation}
\end{example}

\begin{theorem}
    Let $f:\R^d\rightarrow\R$ be convex over a convex set $C$. Then $f$ is convex iff for any $x,v\in \R^d$ the function
    \begin{equation}
        \begin{aligned}
            g:\R&\rightarrow\R\\
            t&\mapsto g(t) = f(x+tv)
        \end{aligned}
    \end{equation}
\end{theorem}
\begin{proof}
    $\Leftarrow$: Assume $g$ is convex for any $x,v$. Let $x,y\in \R^d$ and $\lambda\in (0,1)$. Consider
    \begin{equation}
        g(t) = f(y+ t(x-y)).
    \end{equation}
    Since $g$ is convex, we have
    \begin{equation}
        \begin{aligned}
            f(\lambda x + (1-\lambda)y) 
            = g(\lambda)
            =& g(\lambda*1 + (1-\lambda)*0)\\
            \leq& \lambda g(1) + (1-\lambda)g(0)\\
            =& \lambda f(x) + (1-\lambda)f(y).
        \end{aligned}
    \end{equation}
    The converse direction is left as an exercise.
\end{proof}

\begin{example}[Common convex functions]\
	\begin{itemize}
        \item Examples on $\R$
        \begin{itemize}
            \item Convex functions
            \begin{itemize}
                \item Exponential function $f(x)=\exp(ax)$ on $\R$, $a\in\R$
                \item Powers $f(x)=x^p$ on $(0,\infty)$ for $p\geq 1$ or $p\leq 0$
                \item Powers of absolute functions $f(x)=\vert x\vert^p$ on $\R$ for $p\geq 1$
                \item Negative entropy $f(x)=x\log(x)$ on $(0,\infty)$
            \end{itemize}
            \item Concave functions
            \begin{itemize}
                \item Powers $f(x)=x^p$ on $(0,\infty)$ for $0\leq p\leq1$
                \item Logarithm $f(x)=\log(x)$ on $(0,\infty)$
            \end{itemize}
        \end{itemize}
        \item Examples on $\R^n$
        \begin{itemize}
            \item Convex functions:
            \begin{itemize}
                \item $p$-Norms $f(x)=\|x\|_p$
                \item Power of $p$-norms $f(x)=\|x\|_p^p$
                \item Least squares $f(x)=\frac12\|Ax-b\|_2^2$, $A\in\R^{m\times n}$, $b\in \R^m$
                \item Maximum over affine function
                    \[
                        f(x)=\max\{\inner{a_1}{x}+b_1,\ldots,\inner{a_m}{x}+b_m\},
                    \]
                for $a_i\in\R^n$, $b_i\in\R$
                \item Perspective of a function $f(x,y)=yg(x/y)$ with $g(x)$ a convex function, $x \in \R^n$, $y \in (0,\infty)$
            \end{itemize}
            \item Concave functions:
            \begin{itemize}
                \item Minimum over affine functions
                    \[
                        f(x)=\min\{\inner{a_1}{x}+b_1,\ldots,\inner{a_m}{x}+b_m\}
                    \]
            \end{itemize}
        \end{itemize}
        \item Examples on $\R^{m\times n}$
        \begin{itemize}
            \item Convex functions:
            \begin{itemize}
                \item Affine function
                    \[
                        f(X)=\tr(A^\top X)+b=\sum_{i=1}^m\sum_{j=1}^n A_{ij}X_{ij}+b
                    \]
                    on $\R^{m\times n}$, $A\in\R^{n\times m}$, $b\in \R$
                \item Spectral norm $f(X)=\|X\|_2=\sigma_{\text{max}}(X)=\sqrt{\lambda_{\text{max}}(X^\top X)}$
                \item Barrier for positive definite matrices $f(X)=-\log\det X$ on $\mathbb{S}_{++}^n$
            \end{itemize}
        \end{itemize}
    \end{itemize}
\end{example}

Similarly to the above results we can show the following characterizations of strong convexity:

\begin{theorem}
    A function $f:C\rightarrow \R$ with $C$ convex is $\mu$-strongly convex if and only if
    \begin{itemize}
        \item  $f - \frac12 \|\emptyarg\|^2$ is convex
        \item \emph{First-order condition}: in the case $f$ is continuously differentiable; for any $x,y\in C$
        \[
            f(x) + \inner{\nabla f(x)}{y-x}\leq f(y) - \frac{\mu}{2}\|x-y\|^2
        \]
        \item \emph{Second-order condition}: in the case $f$ is twice continuously differentiable; for any $x$
        \[
            \nabla^2 f(x)\succeq \mu.
        \]
    \end{itemize}
\end{theorem}
\begin{proof}
    The proof is a simple adaptation of the proof for regular convexity and left as an exercise.
\end{proof}


\chapter{Subgradients}

Recall the definition of the (Frechet) derivative.

\begin{defn}[Frechet derivative]
    Let $f:X\rightarrow Y$. Then $f$ is (Frechet) differentiable at $x\in X$ if there exists a linear, continuous mapping $A:X\rightarrow Y$ such that 
    \begin{equation}
        \lim_{\|h\|\rightarrow 0}\frac{\| f(x+h) - f(x)  -Ah\|}{\|h\|}= 0.
    \end{equation}
    We denote $A=Df(x)$.
\end{defn}

\begin{defn}[Gradient]
    Let $f:X\rightarrow \R$ be differentiable with $X$ a Hilbert space. Then The gradient of $f$ at the point $x$, denoted as $\nabla f(x)$, is the unique element in $X$ such that 
    \begin{equation}
        Df(x)h =  \inner{\nabla f(x)}{h}, \quad h\in X.
    \end{equation}
    We denote $A=Df(x)$.
\end{defn}

The existence and uniqueness of the gradient is a direct consequence of the Riesz representation theorem since $Df(x)\in X^*$ if $f$ maps into the real numbers.
Note moreover, that \textbf{the gradient depends on the specific scalar product}.

While the above general definition might be a little abstract, in most cases we will be dealing with the gradient is quite simple. Specifically, whenever $X=\R^d$ equipped with the standard Euclidean scalar product, then 
\begin{equation}
    \nabla f(x) = Df(x)^T,
\end{equation}
\ie, the gradient is simply the transposed derivative.

We can interpret the derivative as the best linear approximation of the function at a point, that is,
\begin{equation}
    f(y) = f(x) + Df(x)(y-x) + o(\|y-x\|)
\end{equation}
where the small $o$ denotes a function tending to zero \emph{faster than linearly}. Most optimization methods can be interpreted as iteratively minimizing a easier approximation of the function. However, many functions of interest do not admit a gradient in the above sense, for instance we might want to solve problems of the form 
\begin{equation}
    \min\limits_x \frac{1}{2}\|Ax-b\|^2_2 + \lambda\|K x\|_1
\end{equation}
with $A$ and $K$ linear. These types of composite problems are extremely popular in imaging.
Thus, we may ask the question if it is possible to define a more general notion of a linear approximation which fulfills similar properties as the gradient. The next definition answers this question positively.

\begin{defn}[Subdifferential]
    Let $f:X\rightarrow Y$. We call $g\in X^*$ a subgradient of $f$ at $x\in X$ if 
    \begin{equation}
        f(x) + \langle g, y-x\rangle\leq f(y), \quad \forall y\in Y.
    \end{equation}
    The subdifferential $\partial f(x)$ at $x\in X$ is the set of all subgradients at $x$, \ie, 
    \begin{equation}
        \partial f(x) =\{ g\in X^*\;|\; f(x) + \langle g, y-x\rangle\leq f(y), \quad \forall y\in Y \}.
    \end{equation}
\end{defn}

Note, that the subdifferential might be empty. It is empty at every point $x\in X$ where $f(x)=\infty$. 

\begin{example}\
    \begin{enumerate}
        \item Norms: Let $f = \norm{\cdot}$, then
			\[
				\partial f(0) = \overline{B}_{\|\cdot\|_*, 1}(0),
			\]
            \ie, the unit ball of the dual norm.
            where we recall the definition of the dual norm $\|\cdot\|_*$
            \[
                \|y\|_* = \sup_{\|x\|\leq 1} \langle x,y\rangle.
            \]
			As an example, if \( f = \|\cdot\|_1\) on \( \R^n \), then
			\[
				\partial f(0) = \overline{B}_{||\cdot\|_\infty,1}(0) = [-1,1]^n.
			\]
        \item Given a nonempty set $S \subseteq X$ and a point $x \in S$, we consider the indicator function $\delta_S$. The subdifferential is given by
			\[
				\partial\delta_S(x)=\left\{y\in X^\ast:\inner{y}{z-x}\leq 0,\;\forall z\in S\right\}\eqqcolon N_S(x),
			\]
			which is the so-called \emph{normal cone} of $S$ at $x$.
			The subdifferential of the indicator function of the unit norm ball is given by
			\[
				N_{\overline{B}_{\|\cdot\|, 1}(0)}(x)=
				\begin{cases}
					\{y\in X^\ast|\|y\|_*\leq \inner{y}{x}\}&\text{if }\norm{x}\leq 1,\\
					\emptyset&\text{if }\norm{x}>1.
				\end{cases}
            \]
    \end{enumerate}
\end{example}

We want to make sure, that the subdifferential is, in fact, a generalization of the regular gradient:

\begin{lemma}
    Let $f:V\rightarrow \R$ be convex and differentiable at $x\in \dom(f)^\circ$. Then $\partial f(x) = \{\nabla f(x)\}$.
\end{lemma}
\begin{proof}
    In order to proof this result, we need to show (i) $\nabla f(x)\in \partial f(x)$ and (ii) if $g\in\partial f(x)$ then $g=\nabla f(x)$. The point (i) directly follows from the first order characterization of the gradient~\cref{convexity:thm:first_order}, which yields
    \begin{equation}
        f(x) + \inner{\nabla f(x)}{y-x}\leq f(y),\quad y\in V
    \end{equation}
    showing that $\nabla f(x)\in\partial f(x)$. For (ii) we note that by definition of the subgradient, for any $g\in\partial f(x)$, $h\in V$, and $t>0$ small enough such that $x+th\in\dom(f)$ we have
    \begin{equation}
        \inner{g}{h}\leq \frac{f(x+th)-f(x)}{t}.
    \end{equation}
    By letting $t\rightarrow 0$ it follows $\inner{g}{h}\leq \inner{\nabla f(x)}{h}$. Since this is true for all $h$, it follows $\inner{g}{h}= \inner{\nabla f(x)}{h}$ and, thus, $g=\nabla f(x)$.
\end{proof}
In fact, we also have the converse result, which we will not prove, however. That is, if the subdifferential is single-valued at a point $x$, then $f$ is differentiable in $x$ and $\partial f(x) = \{\nabla f(x)\}$.

\begin{figure}
    \centering
    \includegraphics[scale=0.4]{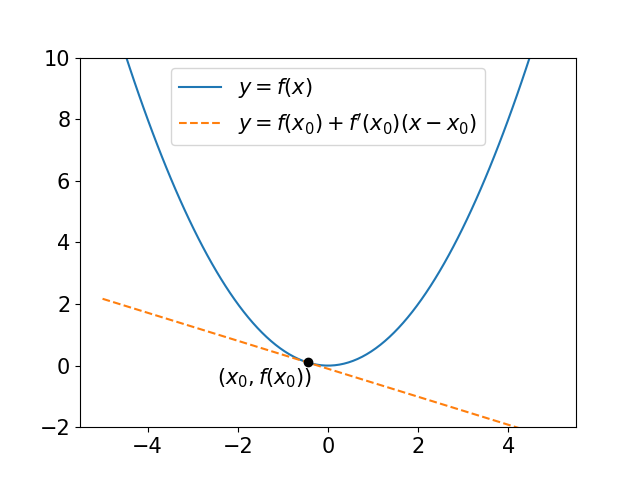}
    \includegraphics[scale=0.4]{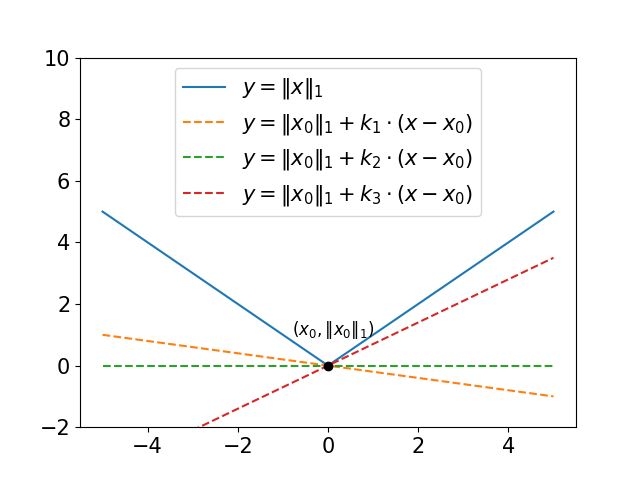}
    \caption{Left: Illustration of the gradient as the slope of the tangent. Right: Illustration of the subgradient.}
\end{figure}

\begin{example}\
    \begin{enumerate}
        \item For $f(x)=\|x\|_2$ it follows
        \[
            \partial f(x) 
            = \begin{cases}
                \{\frac{x}{\|x\|_2}\}\quad& x\neq 0\\
                \overline{B}_{\|\cdot\|_2,1}(0)\quad&\text{else.}
            \end{cases}
        \]
        \item In particular, for $d=1$ and $f(x)=|x|$ it follows
        \[
            \partial f(x) 
            = \begin{cases}
                \{\sign(x)\}\quad& x\neq 0\\
                [-1,1]\quad&\text{else.}
            \end{cases}
        \]
    \end{enumerate}
\end{example}

\begin{lemma}
    If the domain of $f:X\rightarrow Y$ is convex and the subdifferential is nonempty at every point in the domain, then $f$ is convex.
\end{lemma}
\begin{proof}
    Let $x,y\in \dom f$, $\lambda\in (0,1)$ and $g\in \partial f(x+\lambda (y-x))$.
    By definition of the subdifferential we have
    \begin{gather}
        f(x+\lambda(y-x)) + \langle g, x - (x+\lambda(y-x))\rangle \leq f(x)\\
        f(x+\lambda(y-x)) + \langle g, y - (x+\lambda(y-x))\rangle \leq f(y)
    \end{gather}
    that is
    \begin{gather}
        f(x+\lambda(y-x)) + \langle g, -\lambda(y-x)\rangle \leq f(x)\\
        f(x+\lambda(y-x)) + \langle g, (1-\lambda)(y - x) \rangle \leq f(y).
    \end{gather}
    It follows
    \begin{equation}
        \begin{aligned}
            \frac{1}{\lambda} f(x+\lambda(y-x)) + \frac{1}{1-\lambda} f(x+\lambda(y-x))\leq \frac{1}{\lambda} f(x) + \frac{1}{1-\lambda} f(y).
        \end{aligned}
    \end{equation}
    Multyplying by $\lambda (1-\lambda)$ leads to 
    \begin{equation}
        \begin{aligned}
            f(x+\lambda(y-x)) \leq (1-\lambda) f(x) + \lambda f(y).
        \end{aligned}
    \end{equation}
\end{proof}

As stated above, the subdifferential may be empty. However, for convex functions we can guarantee the following:

\begin{theorem}\label{convexity:thm:nonempty_subgrad}
    Let $f:X\rightarrow\overline{\R}$ be proper and convex. Then $\partial f(x)\neq \emptyset$ for every $x\in \dom(f)^\circ$.
\end{theorem}
Before proving this result, we need some preperation:

\begin{lemma}
    The interior of a convex set is convex.
\end{lemma}
\begin{proof}
    Let $C$ be convex. If $C^\circ=\emptyset$ there is nothing to prove so assume $C^\circ\neq \emptyset$. 
    Let $x,y\in C^\circ$ and $\lambda\in (0,1)$. We have to show that $x_\lambda\coloneqq\lambda x + (1-\lambda)y\in C^\circ$.
    Then there exists $\epsilon$ such $B_\epsilon(x),B_\epsilon(y)\subset C^\circ$. Now let $z\in B_{\epsilon}(x_\lambda)$. That is, $z=x_\lambda + h$ with $\|h\|<\epsilon$. We may write
    \begin{equation}
        z = \lambda x + (1-\lambda)y + h = \lambda (x+h) + (1-\lambda)(y + h)
    \end{equation}
    and since $x+h\in B_\epsilon(x)\subset C$, $y+h\in B_\epsilon(y)\subset C$ it follows that $z\in C$ by convexity of $C$. Thus, $B_\epsilon(x_\lambda)\subset C$ and $x_\lambda\in C^\circ$ concluding the proof.
\end{proof}

\begin{lemma}
    If $C\subset\R^d$ is convex and $C^\circ=\emptyset$ then there exists a hyperplane $H$ such that $C\subset H$.
\end{lemma}
\begin{proof}
    We may assume without loss of generality $0\in C$. Otherwise, consider for any $x_0\in C$, $\tilde{C} = C-x_0$. 
    
    We will show that $C$ is a subset of a $d-1$ dimensional subspace. Assume to the contrary that $C$ contains $d$ linearly independent elements, \ie, there exist $x_1,\dots,x_d\in C$ which are linearly independent. By convexity, also the convex hull
    \begin{equation}
        \conv((x_i)_i)=\bigg\{ \sum_{i=1}^d\lambda_i x_i\;\bigg|\; \lambda\in \Delta^d\bigg\}
    \end{equation}
    is a subset of $C$. We claim that $\frac{1}{2d}\sum_{i=1}^d x_i\in C^\circ$. First note that by convexity
    \[
        \frac{1}{2d}\sum_{i=1}^d x_i = \sum_{i=1}^d \frac{1}{2d} x_i + \frac{1}{2}0 \in C.
    \]
    Note that, since the $x_i$ are linearly independent, we may write any $h\in\R^d$ as $h=\sum_i \lambda_i x_i$. Indeed, the mapping $\lambda\mapsto h$ is a bijective linear mapping $\R^d\rightarrow\R^d$. Therefore, there exists $\epsilon>0$ such that if $\|h\|<\epsilon$ it follows $|\lambda_i|<\frac{1}{2d}$ for all $i$. In particular, we have $\frac{1}{2d} + \lambda_i \in [0,1]$ and $\sum_{i=1}^d (\frac{1}{2d} + \lambda_i)\in [0,1]$.
    As a consequence, for any $\|h\|<\epsilon$ we find
    \begin{equation}
        x+h = \sum_{i=1}^d \big(\frac{1}{2d} + \lambda_i\big) x_i = \sum_{i=1}^d \big(\frac{1}{2d} + \lambda_i\big) x_i + \bigg(1 - \sum_i\big(\frac{1}{2d} + \lambda_i\big) \bigg)0\in C.
    \end{equation}
    Therefore, if $C$ contained $d$ linear independent elements, its interior would not be empty as $B_\epsilon(0)\subset C$. Thus, $C$ contains at most $d-1$ linearly independent elements, say $x_1,\dots,x_{d-1}$. In particular, 
    \[
        C\subset \span \{ x_i\;|i=1,\dots,d-1\}
    \]
    concluding the proof.
\end{proof}
Using the two previous results we can prove the supporting hyperplane theorem.
\begin{theorem}[Supporting hyperplane theorem]
    Let $C\subset V$ be nonempty and convex and $x_0\in \partial C = \overline{C}\setminus(C^\circ)$\footnote{the boundary of $C$}. Then there exists $a\in V^*$, $a\neq 0$ such that
    \begin{equation}\label{subgradient:eq:supporting}
        \inner{a}{v}\leq \inner{a}{x_0},\quad v\in C.
    \end{equation}
\end{theorem}
\begin{proof}
    We consider only the finite dimensional case.
    The result is a direct consequence of Hahn-Banach's separation theorem. We make a case distinction. Assume first, $C^\circ\neq \emptyset$. Then $C^\circ$ is open, convex, and non-empty. We may apply Hahn-Banach separation theorem to the sets $C^\circ$ and $\{x_0\}$ yielding the result. In the converse case, if $C^\circ$ is empty, then $C$ is contained in an affine subset $H$ of at most dimension $d-1$. This affine subset also contains $x_0$ by closedness. Let $a\in V^*$ and $b\in \R$ such that 
    \[
        H=\{x\;|\; \inner{x}{a} = b\}.
    \]
    Then this $a$ satisfies~\eqref{subgradient:eq:supporting}.
\end{proof}

We can now prove existence of a subgradient.

\begin{proof}[Proof of \cref{convexity:thm:nonempty_subgrad}]
    We want to show that there exists $g\in X^*$ such that
    \begin{equation}
        f(x)+\inner{g}{y-x}\leq f(y)
    \end{equation}
    for all $y\in X$. This is equivalent to showing that
    \begin{equation}\label{subgrad:eq:nonempty_subdiff1}
        \inner{g}{y} - f(y)\leq \inner{g}{x} - f(x).
    \end{equation}
    This already looks quite similar to the structure of Hahn-Banach. In order to be able to use it, we switch to the epigraph. Note that $(x,f(x))\in \partial \epi(f)$. Moreover, by convexity of $f$, $\epi(f)$ is convex as well. We can therefore apply the supporting hyperplane theorem and find an element $(g,-\alpha)\in (V\times\R)^* = V^*\times\R$ such that\footnote{writing it as $-\alpha$ is just for convenience as we will see in a few seconds}
    \begin{equation}\label{subgrad:eq:nonempty_subdiff2}
        \inner{(g,-\alpha)}{(y,t)}\leq \inner{(g,-\alpha)}{(x,f(x))},\quad (y,t)\in \epi(f).
    \end{equation}
    which is equivalent to
    \begin{equation}
        \inner{g}{y}-\alpha t \leq \inner{g}{x} - \alpha f(x),\quad (y,t)\in \epi(f).
    \end{equation}
    We want to show now, that $\alpha>0$. Assume $\alpha<0$. Since $(x,t)\in \epi(f)$ for any $t>f(x)$ we may let $t\rightarrow \infty$ leading to a contradiction. Now assume next that $\alpha=0$. This would imply 
    \begin{equation}
        \inner{g}{y-x} \leq 0
    \end{equation}
    for all $y\in\dom(f)$ and therefore, since $x\in \dom(f)^\circ$, $g=0$. Since Hahn-Banach provides a non-trivial \emph{separating element} we also have a contradiction. Therefore, $\alpha>0$ and we may divide~\eqref{subgrad:eq:nonempty_subdiff2} by $\alpha$ to obtain
    \begin{equation}
        \inner{\tilde{g}}{y} - t \leq \inner{\tilde{g}}{x} - f(x),\quad (y,t)\in \epi(f).
    \end{equation}
    with $\tilde{g} = g/\alpha$. Picking $t=f(y)$ it follows that $\tilde{g}\in\partial f(x)$ and concluding the proof.
\end{proof}

Similar to the case of the regular derivative, in order to compute subgradients in practice, we can derive several computation rules.

\begin{theorem}
    Let $f:V\rightarrow\R$ proper and convex and $\alpha\geq 0$. Then for any $x\in\dom(f)$ we have
    \[
        \partial(\alpha f)(x) = \alpha\partial f(x).
    \]
\end{theorem}
\begin{proof}
    Exercise.
\end{proof}

\begin{theorem}
    Let $f,g:V\rightarrow\R$ be proper and convex. Then for any $x\in\dom(f)\cap\dom(g)$ we have
    \[
        \partial f(x) + \partial g(x)\subset \partial (f+g).
    \]
    Moreover, if there exists an $x_0\in \dom(f)^\circ \cap\dom(g)$ then
    \[
        \partial f(x) + \partial g(x) = \partial (f+g).
    \]
\end{theorem}
\begin{proof}
    We only proof the first inclusion. The proof of the second one is a little more involved and relies on the Hahn-Banach separation theorem similar to the proof of existence of subgradients.

    Assume $p_1\in \partial f(x)$ and  $p_2\in \partial g(x)$, that is
    \begin{equation}
        \begin{aligned}
            f(x) + \langle p_1, y-x\rangle&\leq f(y)\\
            g(x) + \langle p_2, y-x\rangle&\leq g(y)
        \end{aligned}
    \end{equation}
    Adding both inequalities yields
    \begin{equation}
        \begin{aligned}
            f(x)+ g(x) + \langle p_1+p_2, y-x\rangle \leq f(y) + g(y)
        \end{aligned}
    \end{equation}
    so that $p_1+p_2\in \partial (f+g)(x)$.
\end{proof}

Finally, we list some more \emph{computation rules} for the subdifferential without proof.
\begin{theorem}
    Let $f:W\rightarrow\Rb$ be a proper, convex, and lower semi-continuous function and $A:V\rightarrow W$ a linear transformation.
    Define $h(x)=f(Ax+b)$ with $b\in W$. For any $x\in\dom(h)$ (weak rule)
    \[
        A^\ast(\partial f(Ax+b))\subseteq\partial h(x).
    \]
    with equality if there exists $x_0\in \dom(h)^\circ$ (strong rule).
\end{theorem}

The derivative of a composition of differentiable functions is computed by using the chain rule, \ie, the derivative of the function $h(x) = g(f(x))$ is given by $D h(x) = Dg(f(x))Df(x)$.
This formula can be extended to the subdifferential calculus:
\begin{theorem}
    Let $f:V\rightarrow \R$ be a convex function, $g:\R\rightarrow \R$ be a non-decreasing convex function and $h=g\circ f$.
    Further, let $x\in V$ and suppose that $g$ is differentiable at the point $f(x)$.
    Then,
    \[
        \partial h(x)= g'(f(x))\partial f(x).
    \]
\end{theorem}

\begin{theorem}
    Let $f:V\to(-\infty, \infty]$ be a proper convex function. Then
    \[
        x^\ast\in\arg\min_{x\in V} f(x)
    \]
    if and only if $0\in\partial f(x^\ast)$.
\end{theorem}
\begin{proof}
    The proof directly follows from the subgradient inequality
    \[
        f(x)\geq f(x^\ast)+\inner{0}{x-x^\ast}
    \]
    for any $x\in\dom f$.
\end{proof}

\begin{theorem}
    Let $f:V\rightarrow(-\infty,\infty]$ be a proper and convex function and $C\subseteq V$ be a convex set for which $\dom(f)^\circ\cap C^\circ \neq\emptyset$.
    Then $x^\ast\in C$ is a solution of the constrained optimization problem if and only if
    there exists $g\in \partial f(x^\ast)$ such that
    \[
        \inner{g}{x-x^\ast}\geq 0
    \]
    for all $x\in C$.
\end{theorem}

\chapter{Projected subgradient descent}
Recall that we are interested in solving the problem
\begin{equation}
    \min_{x\in C} f(x).
\end{equation}
where $f:C\rightarrow\R$ is convex.
Recall also, that we may always write a constraint problem as above as an unconstraint one via\footnote{potentially extending $f$ arbitrarily outside of $\R^d$}
\begin{equation}
    \min_{x\in \R^d} f(x) + \delta_C(x).
\end{equation}
Let us recall the most basic numerical optimization method for such problems in the case $C=\R^d$: gradient (or \emph{steepest}) descent. We choose an initial value $x_0$ and then perform for $k=0,1,\dots$ the updates
\begin{equation}
    x_{k+1} = x_k-\tau_k \nabla f(x_k)
\end{equation}
where $\tau_k>0$ denotes the step size. As a first \emph{new} algorithm, we consider the subgradient version of steepest descent, for $k=0,1,\dots$
\begin{equation}
    \begin{cases}
        g_k &\in \partial f(x_k)\\
        x_{k+1} &= x_k-\tau_k g_k.
    \end{cases}
\end{equation}
Whenever $C\neq\R^d$ we may quickly run into a problem: It might happen that $x_{k+1}\notin C$ in which case the method is not well defined as the subgradient of $\delta_C(x)$ is empty for $x\notin C$. Therefore, we somehow need to make sure, that each new iterate stays in $C$ for which we will use \emph{projections}.

\begin{theorem}[Projection theorem]
    Let $V$ be an inner product space and $C\subset V$ closed, convex, and non-empty. Then the following function is well-defined
    \begin{equation}
        \begin{aligned}
            \proj_C:\R^d&\rightarrow C\\
            x&\mapsto\arg\min_{y\in C}\|x-y\|.
        \end{aligned}
    \end{equation}
\end{theorem}
\begin{proof}
    We need to show that the minimum is uniquely attained.
    Let $(y_n)_n$ be a minimizing sequence, \ie,
    \begin{equation}
        \lim_n \|y_n-x\| = \inf_{y\in C}\|y-x\|.
    \end{equation}
    Then $(y_n)_n$ is of course bounded. Therefore, there exists a convergent subsequence $y_{n(k)}\rightarrow \hat y$. By closedness of $C$ it follows $\hat y \in C$. Moreover, by continuity of the norm we have
    \begin{equation}
        \|\hat y-x\| = \lim_n \|y_n-x\| = \inf_{y\in C}\|y-x\|.
    \end{equation}
    This yields existence of the minimizer. Regarding uniqueness, let us assume that $y_1\neq y_2$ were both minimizers. Then, using Young's inequality, we find 
    \begin{equation}
        \begin{aligned}
            \| 1/2 (y_1+y_2) - x\|^2 
            =& \frac{1}{4}\|y_1-x\|^2 + \frac{1}{4}\|y_2-x\|^2 + \frac{1}{2}\inner{y_1-x}{y_2-x}\\
            \leq& \frac{1}{4}\|y_1-x\|^2 + \frac{1}{4}\|y_2-x\|^2 + \frac{1}{4}\|y_1-x\|^2 + \frac{1}{4}\|y_2-x\|^2\\
            \leq& \frac{1}{2}\|y_1-x\|^2 + \frac{1}{2}\|y_2-x\|^2\\
            =& \min_{y\in C}\|y-x\|^2.
        \end{aligned}
    \end{equation}
    Note that the above inequality is strict, whenever $y_1-x$ and $y_2-x$ are not colinear and pointed in the same direction. 
    The strict inequality would imply a contradiction to optimality of $y_1$ and $y_2$ as in this case $1/2(y_1+y_2)\in C$ would yield an even smaller value. So it follows that there exists $\lambda\geq 0$ such that
    \[
        y_2-x = \lambda (y_1-x)
    \]
    But since $\|y_1-x\|=\|y_2-x\|$ it follows that $\lambda=1$, \ie, $y_1=y_2$ concluding the proof.
\end{proof}
Equipped with the projection we have a remedy for the problem of \emph{escaping} $C$: the projected subgradient method, for $k=0,1,\dots$
\begin{equation}\label{subgradient:eq:projectedSG}
    \begin{cases}
        g_k &\in \partial f(x_k)\\
        x_{k+1} &= \proj_C( x_k-\tau_k g_k) .
    \end{cases}
\end{equation}
In order to prove convergence of the method we require some auxiliary results.

\begin{lemma}
    Let $C$ be convex. Then it holds
    \begin{equation}
        \inner{x-\proj_C(x)}{y-\proj_C(x)}\leq 0,\quad y\in C.
    \end{equation}
    In particular, if $C$ is a subspace\footnote{this means $C$ fulfills the conditions of a vector space, \cf~\cref{preliminaries:defn:vectorspace}}, it follows
    \begin{equation}
        \inner{x-\proj_C(x)}{y}= 0,\quad y\in C.
    \end{equation}
\end{lemma}
\begin{proof}
    We have for any $y\in C$
    \begin{equation}
        \begin{aligned}
            \|\proj_C(x)-x\|^2\leq& \|y-x\|^2\\
            =& \|y - \proj_C(x) + \proj_C(x)-x\|^2\\
            =& \|y-\proj_C(x)\|^2 + 2 \inner{y-\proj_C(x)}{\proj_C(x)-x} + \|\proj_C(x)-x\|^2.
        \end{aligned}
    \end{equation}
    As a consequence we have
    \begin{equation}
        \begin{aligned}
            0\leq \|y-\proj_C(x)\|^2 + 2 \inner{y-\proj_C(x)}{\proj_C(x)-x}
        \end{aligned}
    \end{equation}
    Now let $z\in C$ arbitrary and plug in $y=\proj_C(x) + t(z-\proj_C(x))$ for $t\in (0,1)$ above. Then we find
    \begin{equation}
        \begin{aligned}
            0\leq t^2 \|z-\proj_C(x)\|^2 + 2 t\inner{z-\proj_C(x)}{\proj_C(x)-x}.
        \end{aligned}
    \end{equation}
    Dividing by $t>0$ we find
    \begin{equation}
        \begin{aligned}
            0\leq t \|z-\proj_C(x)\|^2 + 2 \inner{z-\proj_C(x)}{\proj_C(x)-x}.
        \end{aligned}
    \end{equation}
    Letting $t\rightarrow 0$ yields the desired result. The equality in the case of a subspace is left as an exercise.
\end{proof}

\begin{prop}
    The projection onto a convex set is nonexpansive, \ie, Lipschitz continuous with Lipschitz constant one, that is
    \begin{equation}
        \|\proj_C(x_1)-\proj_C(x_2)\|\leq \|x_1-x_2\|.
    \end{equation}
\end{prop}
\begin{proof}
    Assume without loss of generality $\|\proj_C(x_2)-\proj_C(x_1)\|^2 \neq 0$ as otherwise there is nothing to prove. We find
    \begin{equation}
        \begin{aligned}
            \|\proj_C(x_2)-\proj_C(x_1)\|^2 
            =& \inner{\proj_C(x_2)-\proj_C(x_1)}{\proj_C(x_2)-x_1+x_1-\proj_C(x_1)}\\
            =& \inner{\proj_C(x_2)-\proj_C(x_1)}{\proj_C(x_2)-x_1} \\
            &+ \inner{\proj_C(x_2)-\proj_C(x_1)}{x_1-\proj_C(x_1)}\\
            \leq& \inner{\proj_C(x_2)-\proj_C(x_1)}{\proj_C(x_2)-x_1}\\
            =& \inner{\proj_C(x_2)-\proj_C(x_1)}{\proj_C(x_2)-x_2+x_2-x_1}\\
            \leq& \inner{\proj_C(x_1)-\proj_C(x_2)}{x_2-\proj_C(x_2)}\\
            &+ \inner{\proj_C(x_2)-\proj_C(x_1)}{x_2-x_1}\\
            \leq& \inner{\proj_C(x_2)-\proj_C(x_1)}{x_2-x_1}\\
            \leq& \|\proj_C(x_2)-\proj_C(x_1)\|\|x_2-x_1\|
        \end{aligned}
    \end{equation}
    Dividing by $\|\proj_C(x_2)-\proj_C(x_1)\|\neq 0$ yields the result.
\end{proof}

\begin{lemma}[Fundamental inequality of the projected subgradient method]\label{lemma:subgradient:fundamental}
    Let $x_k$ denote the iterates of the projected subgradient method~\eqref{subgradient:eq:projectedSG} and $g_k\in\partial f(x_k)$. Then it holds true that 
    \begin{equation}
        \|x_{k+1}-x^*\|^2 
            \leq \|x_k- x^*\|^2 - 2\tau_k (f(x_k)-f(x^*)) + \tau_k^2\|g_k\|^2
    \end{equation}
\end{lemma}
\begin{proof}
    By definition of the subgradient we have $f(x_k)+ \inner{g_k}{x^*-x_k}\leq f(x^*)$. It follows
    \begin{equation}\label{eq:subgradient:fundamental}
        \begin{aligned}
            \|x_{k+1}-x^*\|^2 
            =& \|\proj_C(x_k-\tau_kg_k) - \proj_C(x^*)\|^2\\
            \leq& \|x_k-\tau_k g_k - x^*\|^2\\
            \leq& \|x_k- x^*\|^2 - 2\tau_k \inner{g_k}{x_k-x^*} + \tau_k^2\|g_k\|^2\\
            \leq& \|x_k- x^*\|^2 - 2\tau_k (f(x_k)-f(x^*)) + \tau_k^2\|g_k\|^2
        \end{aligned}
    \end{equation}
\end{proof}
A natural thing to do is now to use the previous result to deduce an optimal step size. More precisely, minimizing the right-hand side of~\eqref{eq:subgradient:fundamental} with respect to $\tau_k$ we find that the optimal step-size is
\begin{equation}
    \tau_k = \frac{f(x_k)-f(x^*)}{\|g_k\|^2}.
\end{equation}
whenever $g_k\neq 0$, and arbitrary (\eg, $\tau_k=1$) whenever $g_k=0$ as in the latter case the iteration stagnates anyway.
This step size is referred to as Polyak's step size rule. We obtain the following result.

\begin{theorem}
    Assume that $f:C\rightarrow \R$ is Lipschitz continuous with Lipschitz constant $L$. With Polyak's step size rule 
    \begin{equation}
    \tau_k = \frac{f(x_k)-f(x^*)}{\|g_k\|^2}
\end{equation}
    the projected subgradient method~\eqref{subgradient:eq:projectedSG} satisfies
    \begin{enumerate}
        \item $\|x_{k+1}-x^*\|^2\leq \|x_{k}-x^*\|$ with strict inequality whenever $g_k\neq 0$,
        \item $f(x_K)\rightarrow f(x^*)$ as $k\rightarrow\infty$, and
        \item $f_{\mathrm{best}}^N-f(x^*)\leq \frac{L\|x_1-x^*\|}{\sqrt{N}}$ where $f_{\mathrm{best}}^N = \min_{k=1,\dots,N}f(x_k)$.
    \end{enumerate}
\end{theorem}
\begin{proof}
    Without loss of generality, we may assume $f(x^*)=0$. Otherwise simply consider $\tilde{f}(x) \coloneq f(x)-f(x^*)$. By~\cref{lemma:subgradient:fundamental} and the step size rule we obtain
    \begin{equation}
        \begin{aligned}
            \|x_{k+1}-x^*\|^2 
            \leq& \|x_{k} - x^*\|^2 - 2\tau_k f(x_k) +\tau_k^2\|g_k\|^2\\
            =& \|x_{k} - x^*\|^2 - a_k\\
        \end{aligned}
    \end{equation}
    where $a_k=\frac{f(x_k)^2}{\|g_k\|^2}$ if $g_k\neq 0$ and zero else.
    By rearranging and summing over $k$ we find
    \begin{equation}
        \sum_{k=1}^{N-1} a_k \leq \|x_1-x^*\|^2 - \|x_N-x^*\|^2\leq \|x_1-x^*\|^2
    \end{equation}
    Lipschitz continuity of $f$ implies that $\|g_k\|\leq L$ (exercise!) so that we obtain
    \begin{equation}
        \sum_{k=1}^{N-1} f(x_k)^2\leq \sum_{k=1}^{N-1} L^2 a_k \leq L^2 \|x_1-x^*\|^2
    \end{equation}
    Since the above is true for any $N$ we may let $N\rightarrow \infty$. As a consequence we find $f(x_k)\rightarrow 0$. Moreover,
    \begin{equation}
        N (f_{\mathrm{best}}^N)^2\leq \sum_{k=1}^{N} f(x_k) \leq L^2\|x_1-x^*\|^2
    \end{equation}
    \ie, 
    \begin{equation}
        f_{\mathrm{best}}^N\leq \frac{L\|x_1-x^*\|}{\sqrt{N}}.
    \end{equation}
\end{proof}

\begin{cor}
    We require $N=\Oc(\frac{L^2\|x_1-x^*\|^2}{\epsilon^2})$ iterations in order to reach accuracy $f_{\mathrm{best}}^N - f(x^*)\leq \epsilon$.
\end{cor}

\begin{rem}
    Recall that standard smooth gradient methods achieve complexity $\Oc(\frac{1}{\epsilon})$.
\end{rem}

While the convergent result using Polyak's step size rule is theoretically interesting and yields a best-case complexity result, in practice we will not be able to compute the step size as it requires knowledge of $f(x^*)$. In the following, we provide a more general convergence result.

\begin{theorem}
    Let $f$ be $L$-Lipschitz and assume the step sizes satsify
    \[
        \frac{\sum_{k=1}^n\tau_k}{\sum_{k=1}^n\tau_k^2}\rightarrow \infty
    \]
\end{theorem}
as $n\rightarrow\infty$. Then it holds $f^n_{\mathrm{best}}\rightarrow f(x^*)$.
\begin{proof}
    We assume again for simplicity $f(x^*)=0$.
    Using again $\|g_k\|\leq L$ and summing the fundamental inequality~\cref{lemma:subgradient:fundamental} over $k=1,\dots, N$ yields
    \begin{equation}
        2\sum_{k=1}^N\tau_k f(x_k) \leq \|x_1-x^*\|^2 + L^2 \sum_{k=1}^N \tau_k^2 
    \end{equation}
    Let us denote
    \[
        \sigma_n = \frac{\sum_{k=1}^n\tau_k}{\sum_{k=1}^n\tau_k^2}
    \]
    It follows
    \begin{equation}
        \sigma_n f^n_{\mathrm{best}} 
        \leq \frac{\sum_{k=1}^n\tau_kf^n_{\mathrm{best}}}{\sum_{k=1}^n\tau_k^2}
        \leq \frac{\sum_{k=1}^n\tau_kf(x_k)}{\sum_{k=1}^n\tau_k^2}
        \leq \frac{\|x_1-x^*\|^2}{2\sum_{k=1}^n\tau_k^2} + L^2/2.
    \end{equation}
    Since the right-hand side is bounded as $n\rightarrow\infty$, $\sigma_n\rightarrow\infty$ implies $f^n_{\mathrm{best}}\rightarrow 0$.
\end{proof}

Variants of the convergence results can moreover be obtained if we assume, \eg, that $C$ is compact and we refer to~\cite{beck2017first} for details.

The convergence can be improved and also \emph{transferred} to the iterates $(x_k)_k$ is we additionally assume that $f$ is strongly convex.

\begin{theorem}
    Assume that $f$ is $\mu$ strongly convex and $L$-Lipschitz over $C$\footnote{This implies that $C$ is bounded as every strongly convex function on an unbounded domain admits an unbeunded gradient!}. Then with the step size choice $\tau_k = \frac{2}{\mu k}$ the projected subgradient method~\eqref{subgradient:eq:projectedSG} satisfies
    \begin{enumerate}
        \item $f^n_{\mathrm{best}}-f(x^*)\leq \frac{L^2}{\mu (n-1)}$ and
        \item $\|x_n-x^*\|\leq \frac{2L}{\mu \sqrt{n-1}}$.
    \end{enumerate}
\end{theorem}
\begin{proof}
    As always, without loss of generality $f(x^*)=0$. Under strong convexity the fundamental inequality can in fact be improved to 
    \begin{equation}
        \begin{aligned}
            \|x_{k+1}-x^*\|^2 
            &\leq (1-\mu\tau_k)\|x_k- x^*\|^2 - 2\tau_k f(x_k) + \tau_k^2\|g_k\|^2\\
            &\leq (1-\mu\tau_k)\|x_k- x^*\|^2 - 2\tau_k f(x_k) + \tau_k^2L^2
        \end{aligned}
    \end{equation}
    Rearranging and dividing by $2\tau_k$ yields
    \begin{equation}
        \begin{aligned}
            f(x_k)
            \leq -\frac{\tau_k^{-1}}{2}\|x_{k+1}-x^*\|^2 + \frac{\tau_k^{-1}-\mu}{2}\|x_k- x^*\|^2  + \frac{\tau_k}{2}L^2.
        \end{aligned}
    \end{equation}
    Inserting the step size choice $\tau_k = \frac{2}{\mu k}$ then yields
    \begin{equation}
        \begin{aligned}
            f(x_k)
            \leq -\frac{\mu k}{4} \|x_{k+1}-x^*\|^2 + \frac{\mu(k-2)}{4}\|x_k- x^*\|^2  + \frac{1}{\mu k}L^2.
        \end{aligned}
    \end{equation}
    Multyplying by $k-1$ and summing over $k$ yields
    \begin{equation}
        \sum_{k=1}^{n} (k-1)f(x_k) \leq -\frac{\mu n(n-1)}{4} \|x_{n+1}-x^*\|^2 + \frac{1}{\mu}\sum_{k=1}^{n}\frac{k-1}{k}L^2.
    \end{equation}
    Using that $\sum_{k=1}^{n} (k-1) = n(n-1)/2$ and that $\sum_{k=1}^{n}\frac{k-1}{k}\leq n$ we obtain
    \begin{equation}\label{subgradient:eq:strongly}
        n(n-1)/2 f^n_{\mathrm{best}}\leq \sum_{k=1}^{n} (k-1)f(x_k) = -\frac{\mu n(n-1)}{4} \|x_{n+1}-x^*\|^2 + \frac{1}{\mu} n L^2.
    \end{equation}
    Thus
    \begin{equation}
        f^n_{\mathrm{best}} \leq \frac{"L^2}{\mu(n-1)}.
    \end{equation}
    On the other hand-side, rearranging \eqref{subgradient:eq:strongly} and using that $f^n_{\mathrm{best}}\geq 0$ yields
    \begin{equation}
        \frac{\mu n(n-1)}{4} \|x_{n+1}-x^*\|^2 \leq \frac{1}{\mu} n L^2,
    \end{equation}
    hence, $\|x_{n+1}-x^*\|^2\leq \frac{4L^2}{\mu^2 (n-1)}$.
\end{proof}

\chapter{Proximal Gradient Methods}

We have seen in the proofs (and practical exercises) that the subgradient method does not provide ideal convergence. The convergence of function values is of order $\Oc(1/\sqrt{k})$ and we require diminishing step sizes for convergence. In this chapter, we will introduce an alternative to (explicit) subgradients for optimization.

Consider the unconstrained problem
\begin{equation}\label{proximal:eq:problem}
    \min_x f(x)
\end{equation}
where $f$ is allowed to take on the value $+\infty$. Recall the subgradient method's update $x_{k+1} = x_k - \tau_k g_k$, $g_k\in\partial f(x_k)$, or, equivalently, $x_{k+1}\in x_k -\tau_k\partial f(x_k)$. In the following we instead propose
\begin{equation}\label{proximal:eq:intro}
    x_{k+1}\in x_k - \tau_k\partial f(x_{k+1}).
\end{equation}
At this point it is not clear that such a method is even well-defined. However, in case it is, we may expect on a high level that it is beneficial in terms of stability to \emph{look into the future} when choosing the update direction. This is also in line with the stability of implicit methods for discretizing differential equations (\cf~\emph{implicit Euler}).

Taking a closer look at~\eqref{proximal:eq:intro} we find that it can be written as
\begin{equation}
    0 \in \partial \biggl(\frac{1}{2}\|\cdot - x_k\|^2 + \tau_k f\biggr)(x_{k+1})
\end{equation}
which amounts to the optimality condition of
\begin{equation}
    \min_x \frac{1}{2}\|x - x_k\|^2 + \tau_k f(x)
\end{equation}
and leads to the following
\begin{defn}[Proximal operator]
    We define the proximal operator/mapping of $f:\R^d\rightarrow (-\infty,\infty]$ as the map
    \begin{equation}
        \begin{aligned}
            \prox_{f}:\R^d&\rightarrow 2^{\R^d}\\
            x&\mapsto \arg\min_{y\in\R^d}\frac{1}{2}\|y - x\|^2 + f(y).
        \end{aligned}
    \end{equation}
\end{defn}

We will usually refer to $\prox_f$ as \enquote{the prox}. Note that the prox as a set-valued mapping is always well-defined. However, it will only be interesting when $\prox_f\neq\emptyset$. We can prove the following.

\begin{lemma}
    Let $f:\R^d\rightarrow (-\infty,\infty]$ be proper, convex, and lower semi-continuous, then the proximal mapping is single-valued and we have $y = \prox_f(x)$ iff $0\in (y-x) + \partial f(y)$.
\end{lemma}
\begin{rem}
    Due to the characterization $0\in (y-x) + \partial f(y)$ we often write the prox as $\prox_f(x) = (\Id + \partial f)^{-1}(x)$.
\end{rem}
\begin{proof}
    We assume for simplicity that $\dom(f)^\circ\neq \emptyset$ and refer to~\cite{beck2017first} for the general case. 
    
    Let $g\in \partial f(y_0)$ for some $y_0\in \dom(f)^\circ$. We have that 
    \begin{equation}
        f(y_0) + \inner{g}{y-y_0} + \frac{1}{2}\|x-y\|^2\leq f(y) + \frac{1}{2}\|x-y\|^2.
    \end{equation}
    Therefore, (why?) $y\mapsto f(y) + \frac{1}{2}\|x-y\|^2$ is coercive. Moreover, this map inherits lower semi-continuity and properness from $f$. By the direct method~\cref{preliminaries:thm:direct_method} there exists a solution to 
    \begin{equation}\label{proximal:eq:prox_proof}
        \min_{y\in\R^d}\frac{1}{2}\|y - x\|^2 + f(y).
    \end{equation}
    Uniqueness follows by the usual strategy, for if, $y_1,y_2$ where two distinct solutions, then by strict convexity of $y\mapsto \frac{1}{2}\|y - x\|^2$ and convexity of $f$ we would obtain for $\bar{y} = \frac{y_1+y_2}{2}$
    \begin{equation}
        \begin{aligned}
            \frac{1}{2}\|\bar y - x\|^2 + f(\bar y)
            <& \frac{1}{2}\bigg(\frac{1}{2}\| y_1 - x\|^2 + f(y_1) + \frac{1}{2}\|y_1 - x\|^2 + f(y_1)\bigg)\\
            =&\arg\min_{y\in\R^d}\frac{1}{2}\|y - x\|^2 + f(y)
        \end{aligned}
    \end{equation}
    which is a contradiction. The last assertion is simply the optimality condition of~\eqref{proximal:eq:prox_proof}.
\end{proof}

Equipped with well-definedness of the prox we can now define the proximal-gradient method. More specifically, we consider problems of the form 
\begin{equation}\label{proximal:eq:composite}
    \min_{x\in \R^d} F(x)\coloneqq f(x) + g(x)
\end{equation}
where $f$ is convex and $L$-smooth, that is, differentiable with $L$-Lipschitz continuous gradient and $g$ is proper, convex, and lower semi-continuous. Such problems arise frequently in practice and can be reduced to the originally considered problems by setting $f\equiv 0$. We define the proximal gradient method as for $k=0,1,\dots$
\begin{equation}\label{proximal:eq:prox_grad}
    \begin{aligned}
        x_{k+1} = \prox_{\tau_k g}(x_k - \tau_k \nabla f(x_k)).
    \end{aligned}
\end{equation}
That is, we perform gradient/explicit steps with respect to $f$ and proximal/implicit steps with respect to $g$.

As a first simple result we show that the proximal-gradient method is reasonable in the sense that stationary points are indeed optimal. 
\begin{lemma}
    Assume $x\in\R^d$ is a stationary point of the proximal-gradient method, \ie, 
    \begin{equation}
        x = \prox_{\tau_k g}(x - \tau_k \nabla f(x)).
    \end{equation}
    Then $x$ solves~\eqref{proximal:eq:composite}.
\end{lemma}
\begin{proof}
    Let $x$ be a stationary point as above. Inserting the optimlity conditions for the prox yields
    \begin{equation}
        x \in x - \tau_k \nabla f(x) - \tau_k \partial g(x).
    \end{equation}
    By the sum and the scaling rule for the subdifferential we have
    \begin{equation}
        0 \in \tau_k \partial (f + g)(x),
    \end{equation}
    \ie, $0 \in \partial (f+g)(x)$ which implies optimality of $x$.
\end{proof}

In addition to the prox we also define the Moreau envelope as the objective value realized by the prox.
\begin{defn}[Moreau envelope]
    Let $f:\R^d\rightarrow (-\infty,\infty]$. We define the Moreau envelope $M_f^\tau:\R^d\rightarrow [-\infty,\infty]$ as 
    \begin{equation}
        M_f^\tau(x) \coloneq \min_{y\in \R^d}\frac{1}{2\tau} \|x-y\|^2 + f(y).
    \end{equation}
\end{defn}

\begin{lemma}
    Let $f:\R^d\rightarrow (-\infty,\infty]$ be proper, convex, and lower semi-continuous. Then the Moreau envelope is always a real (\ie, finite) number and we have
    \begin{equation}
        M_f^\tau(x) = \frac{1}{2\tau} \|x-\prox_{\tau f}(x)\| + f(\prox_{\tau f}(x)).
    \end{equation}
\end{lemma}
\begin{proof}
    Exercise.
\end{proof}

\begin{example}
    We consider a few frequently occurring examples of proximal maps/Moreau envelopes:
    \begin{enumerate}
        \item Let $C\subset\R^d$ be closed and convex and consider the indicator function $\delta_C$. Then 
        \begin{equation}
            \prox_{\tau \delta_C}(x) = \proj_C(x),\quad M^{\tau}_{\delta_C}(x) = \frac{1}{2\tau}\mathrm{dist}(x,C)^2.
        \end{equation}
        \item In particular, if $C=\{x\in\R^d\,|\, \|x\|_\infty\leq 1\}$ then
        \begin{equation}
            (\prox_{\tau \delta_C}(x))_i = \frac{x_i}{\max\{1,|x_i|\}}
        \end{equation}
        \item Let $f(x) = \|x\|_1$ then the proximal map is the so-called \emph{soft-thresholding} operator
        \begin{equation}
            (\prox_{\tau f}(x))_i = 
            \begin{cases}
                x_i-\tau\quad &x_i>\tau\\
                x_i+\tau\quad &x_t<-\tau\\
                0\quad &x_i\in[-\tau,\tau]
            \end{cases}
        \end{equation}
        For $d=1$, moreover, the Moreau envelope is exactly the Huber functional
        \begin{equation}
            M^\tau_{|\cdot|}(x) = 
            \begin{cases}
                \frac{|x|^2}{2\tau}\quad &|x|\leq \tau\\
                |x|-\frac{\tau}{2} \quad &\text{else.}
            \end{cases}
        \end{equation}
    \end{enumerate}
\end{example}

\begin{lemma}[Computation rules for the prox and the Moreau envelope]
    The following rules apply
    \begin{enumerate}
        \item $f(x) = \sum_i f_i(x_i)$ with $x = (x_1,x_2,\dots,x_k)\R^d$, $x_i\in\R^{d_i}$, $\sum_i d_i = d$, then $\prox_{f}(x) = ((\prox_{f_i}(x_i))_i)$.
        \item $f(x)=\alpha g(x) + b$ with $\alpha>0$, $b\in \R$, then $\prox_f = \prox_{\alpha g}$.
        \item $f(x)=g(\alpha x+ b)$ with $\alpha\neq 0$, $b\in\R^d$, then $\prox_f(x) = \frac{1}{\alpha}(\prox_{\alpha^2 g}(\alpha x + b)-b)$
        \item $f(x) = g(Qx)$ with $Q\in\R^{d\times d}$ orthonormal, then $\prox_f(x) = Q^T\prox_{g}(Qx)$
        \item $f(x) = g(x) + \inner{a}{x} + b$ with $a\in\R^d$, $b\in R$, then $\prox_{f}(x) = \prox_g (x-a)$
        \item $f(x) = g(x) + \frac{\gamma}{2}\|x-a\|^2$ with $a\in\R^d$, then $\prox_f(x) = \prox_{\tilde{\gamma}g}(\tilde{\gamma}x + \gamma\tilde{\gamma}a)$ with $\tilde{\gamma} = \frac{1}{1+\gamma}$.
    \end{enumerate}
\end{lemma}

The Moreau envelope has a smoothing effect on the function $f$ as will be shown in the following result.
\begin{lemma}
    Let $f:\R^d\rightarrow (-\infty,\infty]$ be proper, convex, and lower semi-continuous. Then the Moreau envelope is convex and differentiable. Moreover, its gradient is $\nabla M^\tau_f$ is $\frac{1}{\tau}$-Lipschitz and can be expressed as
    \begin{equation}
        \nabla M^\tau_f(x) = \frac{1}{\tau} (x-\prox_{\tau f}(x)).
    \end{equation}
\end{lemma}
\begin{proof}
    Convexity of the Moreau envelope follows as in~\cref{convexity:thm:minimum}. Regarding the gradient of the Moreau envelope, fix $\tau$ and denote for simplicity in the following $p(x) = \prox_{\tau f}(x)$. We have by definition of the prox
    \begin{equation}
        \begin{aligned}
            M^\tau_f(x+h) = \frac{1}{2\tau}\|x+h-p(x+h)\|^2 + f(p(x+h))
            \leq \frac{1}{2\tau}\|x+h-p(x)\|^2 + f(p(x))\\
            M^\tau_f(x) = \frac{1}{2\tau}\|x-p(x)\|^2 + f(p(x)).
        \end{aligned}
    \end{equation}
    Subtracting the two yields
    \begin{equation}
        \begin{aligned}
            M^\tau_f(x+h) - M^\tau_f(x)
            \leq& \frac{1}{2\tau}(\|x+h-p(x)\|^2 - \|x-p(x)\|^2)\\
            =& \frac{1}{2\tau}(\|x+h\|^2 - \|x\|^2 -2\inner{h}{p(y)})\\
            =& \inner{h}{\frac{1}{\tau}(x-p(x))} + \frac{1}{2\tau}\|h\|^2
        \end{aligned}
    \end{equation}
    or, equivalently
    \begin{equation}
        \begin{aligned}
            M^\tau_f(x+h) - M^\tau_f(x) - \inner{h}{\frac{1}{\tau}(x-p(x))}
            \leq & \frac{1}{2\tau}\|h\|^2
        \end{aligned}
    \end{equation}
    If we can bound
    \begin{equation}\label{proximal:eq:moreau_diff}
        \phi(h)\coloneqq M^\tau_f(x) - M^\tau_f(y) - \inner{x-y}{\frac{1}{\tau}(y-p(y))}
    \end{equation}
    similarly from below, we can conclude. Note that $\phi$ is convex and $\phi(0) = 0$, thus, $0 = \phi(0) = \phi(\frac{1}{2}h + \frac{1}{2}(-h))\leq \frac{1}{2}(\phi(h) + \phi(-h))$, \ie,
    \begin{equation}
        \phi(h)\geq -\phi(-h) \geq -\frac{1}{2\tau}\|h\|^2.
    \end{equation}
    In total we find
    \begin{equation}
        \begin{aligned}
            \left|M^\tau_f(x+h) - M^\tau_f(x) - \inner{h}{\frac{1}{\tau}(x-p(x))}\right|
            \leq & \frac{1}{2\tau}\|h\|^2
        \end{aligned}
    \end{equation}
    implying $\nabla M^\tau_f(x) = \frac{1}{\tau}(x-p(x))$. To prove Lipschitz continuity of the gradient, we note that
    \begin{equation}\label{proximal:eq:moreau_diff2}
        \begin{aligned}
            \|\nabla M^\tau_f(x) - \nabla M^\tau_f(y)\|^2
            =& \inner{\nabla M^\tau_f(x) - \nabla M^\tau_f(y)}{\frac{1}{\tau}(x-p(x)) - \frac{1}{\tau}(y-p(y))}\\
            =& \frac{1}{\tau}\inner{\nabla M^\tau_f(x) - \nabla M^\tau_f(y)}{x - y}
            - \frac{1}{\tau} \inner{\nabla M^\tau_f(x) - \nabla M^\tau_f(y)}{p(x)-p(y)}.
        \end{aligned}
    \end{equation}
    Note that $\nabla M^\tau_f(x) = \frac{1}{\tau}(x-p(x))\in \partial f(x)$ by the optimality conditions for the prox so that the second inner product in~\eqref{proximal:eq:moreau_diff2} is non-negative by monotonicity of the subgradient leading to 
    \begin{equation}\label{proximal:eq:moreau_diff2}
        \begin{aligned}
            \|\nabla M^\tau_f(x) - \nabla M^\tau_f(y)\|^2
            \leq & \frac{1}{\tau}\inner{\nabla M^\tau_f(x) - \nabla M^\tau_f(y)}{x - y}
            \leq\frac{1}{\tau}\|\nabla M^\tau_f(x) - \nabla M^\tau_f(y)\|\|x-y\|
        \end{aligned}
    \end{equation}
    concluding the proof.
\end{proof}

\begin{rem}
    With the above representation of the gradient of the Moreau envelope we can rewrite the prox as 
    \begin{equation}
        \prox_{\tau f}(x) 
        = x-\tau \bigg(\frac{1}{\tau}(x-\prox_{\tau f}(x))\bigg)
        = x-\tau \nabla M^\tau_f(x).
    \end{equation}
    That is, a proximal step is equivalent to a gradient step on the Moreau envelope!
\end{rem}

\begin{example}
    Projection, 1 norm, 2 norm
\end{example}

We will now proof convergence of the proximal gradient method. First we derive an essential lemma which states that for an $L$-smooth function we can bound the error between the function and its linear approximation by a square from above.

\begin{lemma}\label{proximal_gradient:lemma:Lsmoothness}
    Let $F:\R^d\rightarrow \R$ be continuously differentiable with $L$-Lipschitz gradient. Then it holds true that 
    \begin{equation}
        F(y)\leq F(x) + \inner{\nabla F(x)}{y-x} + \frac{L}{2}\|y-x\|^2.
    \end{equation}
\end{lemma}
\begin{proof}
    By the fundamental theorem of calculus and $L$-smoothness we have
    \begin{equation}
        \begin{aligned}
            F(y)-F(x) 
            =& \int_0^1 \frac{\dd}{\dd t}F(x+t(y-x))\dd t\\
            =& \int_0^1\inner{\nabla F(x+t(y-x))}{y-x}\dd t\\
            =& \int_0^1\inner{\nabla F(x+t(y-x)) - \nabla F(x)}{y-x}\dd t + \int_0^1\inner{\nabla F(x)}{y-x}\dd t\\
            \leq & \int_0^1L t \|y-x\|^2\dd t + \inner{\nabla F(x)}{y-x}\\
            = & \frac{L}{2} \|y-x\|^2 + \inner{\nabla F(x)}{y-x}
        \end{aligned}
    \end{equation}
\end{proof}

For the proof of convergence of the proximal-gradient method we introduce the following notation
\begin{equation}
    T_\tau(x) = \frac{1}{\tau}(x-\prox_{\tau g}(x-\tau \nabla f(x))).
\end{equation}
With this function we may write the update of the proximal-gradient method as 
\begin{equation}
    x_{k+1} = x_k - \tau_k T_{\tau_k}(x_k)
\end{equation}

\begin{theorem}
    Let $f:\R^d\rightarrow\R$ be convex and $L$-smooth and $g:\R^d\rightarrow(-\infty,\infty]$ proper, closed, and convex. Moreover, assume $\tau_k\leq L^{-1}$ for all $k$. Then, with $x^*\in\arg\min_x F(x)$, the proximal gradient algorithm satisfies
    \begin{equation}
        F(x_{n})-F(x^*)\leq \frac{\|x_{0}-x^*\|^2}{2\tau_{\min} n}.
    \end{equation}
    Moreover, there exists $x^*\in\arg\min_x F(x)$ such that $x_k\rightarrow x^*$.
\end{theorem}
\begin{proof}
    We first note that by definition of $T_\tau$ we have
    \begin{equation}
        T_{\tau_k}(x_k)-\nabla f(x_k)\in \partial g(x_k-\tau_k T_{\tau_k}(x_k)) = \partial g(x_{k+1}).
    \end{equation}
    Using this, the fact that $\nabla f(x)\in \partial f(x)$,~\cref{proximal_gradient:lemma:Lsmoothness}, and $\tau_k\leq L$, we find for any $z\in\R^d$
    \begin{equation}
        \begin{aligned}
            F(x_{k+1})
            \leq& f(x_k) -\tau_k\inner{\nabla f(x_k)}{T_{\tau_k}(x_k)} + \frac{\tau_k}{2}\|T_{\tau_k}(x_k)\|^2\\
            &+g(z) - \inner{T_{\tau_k}(x_k)-\nabla f(x_k)}{z - (x_k-\tau_k T_{\tau_k}(x_k))}\\
            \leq& f(z) + g(z)\\
            &- \inner{\nabla f(x_k)}{z-x_k}-\tau_k\inner{\nabla f(x_k)}{T_{\tau_k}(x_k)} + \frac{\tau_k}{2}\|T_{\tau_k}(x_k)\|^2\\
            &- \inner{T_{\tau_k}(x_k)-\nabla f(x_k)}{z - (x_k-\tau_k T_{\tau_k}(x_k))}\\
            \leq& F(z) - \inner{T_{\tau_k}(x_k)}{z - x_k} -\frac{\tau_k}{2}\|T_{\tau_k}(x_k)\|^2.
        \end{aligned}
    \end{equation}
    Inserting $z=x_k$ yields
    \begin{equation}
        F(x_{k+1})\leq F(x_k) - \frac{\tau_k}{2}\|T_{\tau_k}(x_k)\|^2
    \end{equation}
    which implies that the proximal-gradient method is a descent method with strict decrease except for the case $T_{\tau_k}(x_k)=0$ which, however, implies that $x_k$ is optimal. On the other hand, inserting an optimal point $z=x^*$ and noting that $T_{\tau_k} = \tau_k^{-1}(x_k-x_{k+1})$, it follows
    \begin{equation}
        \begin{aligned}
            0\leq F(x_{k+1})-F(x^*)
            \leq& - \tau_k^{-1} \inner{x_k-x_{k+1}}{x^* - x_k} -\frac{\tau_k^{-1}}{2}\|x_k-x_{k+1}\|^2\\
            \leq& -\frac{\tau_k^{-1}}{2} \left( \|x_k-x_{k+1}\|^2 - \inner{x_k-x_{k+1}}{x_k-x^*}\right)\\
            \leq& -\frac{\tau_k^{-1}}{2} \left( \|x_{k+1} - x^*\|^2 - \|x_k-x^*\|^2\right)\\
            \leq& \frac{\tau_k^{-1}}{2} \left( \|x_k-x^*\|^2 - \|x_{k+1} - x^*\|^2\right).
        \end{aligned}
    \end{equation}
    In particular, note that $ \|x_k-x^*\|^2 - \|x_{k+1} - x^*\|^2\geq0$ and, thus, $(\|x_{k+1} - x^*\|^2)_k$ is decreasing as well.
    As usual, summing over $k$ yields
    \begin{equation}
        \begin{aligned}
            \sum_{k=1}^n F(x_{k})-F(x^*)
            \leq& \sum_{k=1}^n\frac{\tau_k^{-1}}{2} \left( \|x_{k-1}-x^*\|^2 - \|x_{k} - x^*\|^2\right)\\
            \leq& \frac{\tau_{\min}^{-1}}{2}\sum_{k=1}^n \left( \|x_{k-1}-x^*\|^2 - \|x_{k} - x^*\|^2\right)\\
            =& \frac{\tau_{\min}^{-1}}{2} \left( \|x_{0}-x^*\|^2 - \|x_{n} - x^*\|^2\right)
        \end{aligned}
    \end{equation}
    By the fact that $(F(x_{k})-F(x^*))_k$ is also decreasing we can deduce
    \begin{equation}
        n(F(x_{n})-F(x^*))\leq \sum_{k=1}^n F(x_{k})-F(x^*)
        \leq \frac{\tau_{\min}^{-1}}{2} \|x_{0}-x^*\|^2
    \end{equation}
    and, thus, 
    \begin{equation}
        F(x_{n})-F(x^*)\leq \frac{\|x_{0}-x^*\|^2}{2\tau_{\min} n}.
    \end{equation}
    Lastly, we want to show convergence of the sequence $(x_k)_k$. Note that the above derivations have shown that $(\|x_k-x^*\|)_k$ is monotonically decreasing for every minimizer $x^*$ of $F$. In particular, $(x_k)_k$ is bounded, and thus, admits a convergent subsequence. 
    Let $\hat x$ be an arbitrary accumulation point of $(x_k)_k$, that is $\hat x=\lim_k x_{n(k)}$ for some subsequence. Since $F(x_k)\rightarrow \min F$ we have by lower semi-continuity
    \[
        F(\hat x))\leq \liminf_k F(x_{n(k)}) = \lim F(x_k) = \min F,
    \]
    that is, $\hat x$ is a minimizer of $F$ and $(\|x_k-\hat{x}\|)_k$ is, therefore, decreasing. As a consequence there exists $c\in\R$ such that $\|x_k-\hat{x}\|\rightarrow c$. However, this implies
    \begin{equation}
        c = \lim_k \|x_k-\hat{x}\| = \lim_k \|x_{n(k)}-\hat{x}\| = 0
    \end{equation}
    meaning that already the original sequence converges to $\hat{x}$ concluding the proof.
\end{proof}
\chapter{Acceleration}

By constructing a specifically difficult function to optimize, one can show that for any first order method which satisfies
\begin{equation}
    x_k\in x_0+\mathrm{span}\{\nabla f(x_0),\dots,\nabla f(x_{k-1})\}
\end{equation}
there exists $L$-smooth $f$ such that
\begin{equation}
    f(x_k)-\min f\geq\frac{3L\|x_0-x^*\|}{32(k+1)^2}.
\end{equation}
Note that up until now we have not reached convergence better than $\Oc(1/k)$. Closing the gap to $\Oc(1/k^2)$ is the goal of this section.

\section{Polyak's heavy ball method}

Note that we can interpret gradient descent for minimizing the (unconstrained) problem
\begin{equation}
    \min_x f(x)
\end{equation}
as a discretization of the \emph{gradient flow}
\begin{equation}
    \dot x  = -\nabla f(x).
\end{equation}
Interpreting the function $f$ as a potential, $\nabla f$ corresponds to the force enacted by this potential on a particle (with mass one). It is now natural to include friction into this model. The force caused by friction is typically modelled proportional to the velocity leading to
\begin{equation}\label{acceleration:eq:friction_ode}
    \ddot x = - \gamma \dot x -\nabla f(x).
\end{equation}
Let us discretize this ODE via $\dot x (t)\approx \frac{x(t)-x(t-h)}{h}$ and 
\begin{equation}
    \ddot x(t)
    \approx \frac{\dot x(t+h)-\dot x(t)}{h}
    \approx \frac{x(t+h)-2x(t)+x(t-h)}{h^2}.
\end{equation}
Denoting $x_{t-h}=x_{k-1}$, $x(t)=x_k$ and $x(t+h)=x_{k+1}$ we find via inserting into~\cref{acceleration:eq:friction_ode}
\begin{equation}
    x_{k+1} - 2x_k + x_{k-1} = -h\gamma (x_{k}-x_{k-1}) - h^2\nabla f(x_k)
\end{equation}
which yields after rearranging the update
\begin{equation}
    x_{k+1} = x_{k} + (1-\gamma h)(x_{k}-x_{k-1}) - h^2\nabla f(x_k).
\end{equation}
Relabelling the constants and allowing them also to be iteration depdendent we arrive at \emph{Polyak's heavy ball method}
\begin{equation}
    x_{k+1} = x_{k} + \beta_k(x_{k}-x_{k-1}) - \tau_k\nabla f(x_k).
\end{equation}
Note that even without this derivation using the ODE perspective, the update rule is rather intuitive: The update contains in addition to the descent direction for $f$ also an additional \emph{inertia} term, which adds a compontent in the same direction as the previous step.

The convergence proof of the heavy ball method will rely on the following basic results from linear algebra which we include here for the sake of completeness.

\begin{lemma}\label{acceleration:lemma:spectral_radius}
    For any $A\in \R^{d\times d}$ we have $\lim_{n\rightarrow\infty}\|A^n\|^{1/n} = \rho(A)$ where 
    \begin{equation}
        \rho(A) = \max_{i}|\lambda_i(A)|
    \end{equation}
    with $\lambda_i(A)$ the eigenvalues of $A$.
\end{lemma}
\begin{proof}
    First of all it is easy to see that $\liminf_{n\rightarrow\infty}\|A^n\|^{1/n} \geq \rho(A)$. Indeed, let $v$ be a normalized eigenvector for the largest (in absolute value) eigenvalue of $A$, then we have
    \begin{equation}
        \|A^n\|^{1/n} \geq \|A^n v\|^{1/n} = \rho(A)
    \end{equation}
    and thus also $\liminf_{n\rightarrow\infty}\|A^n\|^{1/n} \geq \rho(A)$.
    The converse inequality is a little more subtle.
    Recall from linar algebra that for every square matrix we can derive the Jordan normal form
    \begin{equation}
        A = B^{-1}D B
    \end{equation}
    with 
    \begin{equation}
        D = \begin{pmatrix}
            \lambda_1& \rho_1 & \dots&\dots&\\
            0 & \lambda_2& \rho_2&\dots&\\
            &&\ddots&\ddots&\\
            & &\dots & \lambda_{d-1}&\rho_{d-1}&\\
            & & &\dots&\lambda_d\\
        \end{pmatrix}
    \end{equation}
    with $\lambda_i$ the eigenvalues (with multiplicity) and $\rho_i\in \{0,1\}$ and which we write as $D=\Lambda + J$ where $\Lambda$ contains the diagonal entries and $J$ the off diagonals. Note that $J^k=0$ whenever $k$ is larger than the largest Jordan block. In particular, $J^k=0$ for $k\geq d-1$. Since $\Lambda$ and $J$ commute, we, thus, find
    \begin{equation}
        \begin{aligned}
            \|D^n\| = 
            \|\sum_{k=0}^{d-2}
            \begin{pmatrix}
                n\\
                k
            \end{pmatrix}\Lambda^{n-k}J^k\|
            \leq \sum_{k=0}^{d-2}
            \begin{pmatrix}
                n\\
                k
            \end{pmatrix}\|\Lambda^{n-k}\|
            \leq& \sum_{k=0}^{d-2} \frac{n^k}{k!}\|\Lambda^{n-k}\|\\
            \leq& \sum_{k=0}^{d-2} \frac{n^k}{k!}\rho(A)^{n-k}\\
            =& \rho(A)^{n}\sum_{k=0}^{d-2} \frac{n^k}{k!}\rho(A)^{-k}\\
        \end{aligned}
    \end{equation}
    which implies
    \begin{equation}
        \begin{aligned}
            \|D^n\|^{1/n}
            \leq& \rho(A)\left(\sum_{k=0}^{d-2} \frac{n^k}{k!}\rho(A)^{-k}\right)^{1/n}
            \leq& \rho(A)\left((d-1)n^{d-2}\max_{k=0,\dots, d-2}\rho(A)^{-k}\right)^{1/n}
        \end{aligned}
    \end{equation}
    where we note that tha maximizing $k$ above is either $k=0$ or $k=d-2$ when $\rho(A)\leq 1$ or $\rho(A)>1$, respectively. Since $\lim_{n\rightarrow \infty} n^{1/n}=1$ it follows \begin{equation}
        \begin{aligned}
            \limsup_n\|D^n\|^{1/n}
            \leq \rho(A).
        \end{aligned}
    \end{equation}
    Lastly, we conclude
    \begin{equation}
        \limsup_n\|A^n\|^{1/n} 
        = \limsup_n \|B^{-1}D^nB\|^{1/n}
        \leq \limsup_n\|B^{-1}\|^{1/n}\|D^n\|^{1/n} \|B\|^{1/n}\leq \rho(A)
    \end{equation}
\end{proof}

\begin{cor}\label{acceleration:cor:spectral_radius}
    It holds that $\lim_k A^k=0$ if and only if $\rho(A)<1$. Moreover, in this case for every $\epsilon>0$ there exists $c(\epsilon)>0$ such that $\|A^k\|\leq c(\epsilon)(\rho(A)+\epsilon)^k$.
\end{cor}
\begin{proof}
    The \enquote{if and only if}-part of the theorem follows from directly from~\cref{acceleration:lemma:spectral_radius}.
    Moreover, by~\cref{acceleration:lemma:spectral_radius} we can choose $n$ sufficiently large such that $\|A^k\|^{1/k}\leq \rho(A) + \epsilon$ for $k\geq n$.
    For such $k$ it follows
    \begin{equation}
        \|A^k\|\leq (\rho(A)+\epsilon)^k.
    \end{equation}
    Since there also exists some $c>0$ such that 
    \begin{equation}
        \|A^k\|\leq c(\rho(A)+\epsilon)^k.
    \end{equation}
    for $k\leq n$ the result follows.
\end{proof}

We can now prove convergence of the heavy ball method. The technique we use is quite standard: We analyse the eigenvalues of the linearization of the update rule.

\begin{theorem}[Convergence of heavy ball]
    Let 
    \begin{equation}
        0\leq \beta <1, \quad 0<\tau <2(1+\beta)/L,\quad \mu \Id\preceq \nabla^2 f(x^*)\preceq L\Id.
    \end{equation}
    There exists $\epsilon>0$ such that for $x_0,x_1\in B_\epsilon(x^*)$ it holds that $x_k\rightarrow x^*$. More specifically, 
    \begin{equation}
        \|x_k-x^*\|\leq c(\delta)(q+\delta)^k,\quad 0\leq q< 1, 0<\delta<1-q.
    \end{equation}
    Moreover, we can derive the optimal parameters (\ie, the minimal value of $q$) as
    \begin{equation}
        q = \frac{\sqrt{L}-\sqrt{\mu}}{\sqrt{L}+\sqrt{\mu}},
        \quad \tau = \frac{4}{(\sqrt{L} + \sqrt{\mu})^2}
        ,\quad \beta = \left(\frac{\sqrt{L}-\sqrt{\mu}}{\sqrt{L}+\sqrt{\mu}}\right)^2
    \end{equation}
\end{theorem}
\begin{proof}
    Note that, as is, the scheme is, in fact, a recursion of depth two. We, thus, rewrite the update as
    \begin{equation}
        \begin{aligned}
            \begin{bmatrix}
                x_{k+1}\\
                x_k
            \end{bmatrix}
            &= 
            \begin{bmatrix}
                x_{k} + \beta(x_{k}-x_{k-1}) - \tau\nabla f(x_k)\\
                x_k
            \end{bmatrix}\\
            &=
            \begin{pmatrix}
                1+\beta - \tau \nabla f&-\beta\\
                1&0
            \end{pmatrix}
            \begin{bmatrix}
                x_k\\
                x_{k-1}
            \end{bmatrix}.
        \end{aligned}
    \end{equation}
    Moreover, we can subtract the minimizer $x^*$, \ie, 
    \begin{equation}
        \begin{aligned}
            \begin{bmatrix}
                x_{k+1}-x^*\\
                x_k-x^*
            \end{bmatrix}
            &=
            \begin{pmatrix}
                1+\beta - \tau \nabla f&-\beta\\
                1&0
            \end{pmatrix}
            \begin{bmatrix}
                x_k-x^*\\
                x_{k-1}-x^*
            \end{bmatrix}.
        \end{aligned}
    \end{equation}
    Moreover, we can \emph{linearize} the update as follows: By the fundamental theorem of calculus we have
    \begin{equation}
        \begin{aligned}
            \nabla f(x_k) - \nabla f(x^*) 
            =& \int_0^1 \nabla^2 f(x^* + t(x_k-x^*))\dd t (x_k-x^*)\\
            =& \nabla^2 f(x^*)(x_k-x^*) + \int_0^1 \nabla^2 f(x^* + t(x_k-x^*))-\nabla^2 f(x^*)\dd t (x_k-x^*).
        \end{aligned}
    \end{equation}
    Thus, defining $z_k = \int_0^1 \nabla^2 f(x^* + t(x_k-x^*))-\nabla^2 f(x^*)\dd t (x_k-x^*)$ we have
    \begin{equation}
        \begin{aligned}
            \begin{bmatrix}
                x_{k+1}-x^*\\
                x_k-x^*
            \end{bmatrix}
            &=
            \underbrace{\begin{pmatrix}
                1+\beta - \tau \nabla^2 f(x^*)&-\beta\\
                1&0
            \end{pmatrix}}_{\mathcal{H}}
            \begin{bmatrix}
                x_k-x^*\\
                x_{k-1}-x^*
            \end{bmatrix}
            + \begin{bmatrix}
                z_k\\
                0
            \end{bmatrix}
        \end{aligned}
    \end{equation}
    In view of~\cref{acceleration:cor:spectral_radius}, we analyze the eigenvalues of $\mathcal{H}$. Let $(v,w)$ be an eigenvector of $\mathcal{H}$ with eigenvalue $\lambda$. We find
    \begin{equation}
        \begin{cases}
            (1+\beta -\tau \nabla^2 f(x^*))v - \beta w &= \lambda v\\
            v &= \lambda w.
        \end{cases}
    \end{equation}
    Inserting the second into the first equation it follows
    \begin{equation}
        (1+\beta -\tau \nabla^2 f(x^*))v = \lambda v + \beta\lambda^{-1} v.
    \end{equation}
    This, in turn, implies that there exist an eigenvalue $\eta$ of $\nabla^2f(x^*)$ such that
    \begin{equation}
        1+\beta -\tau \eta = \lambda + \beta\lambda^{-1},
    \end{equation}
    respectively
    \begin{equation}
        \lambda^2 - \lambda(1+\beta -\tau \eta) = -\beta
    \end{equation}
    by completing the square we obtain
    \begin{equation}\label{acceleration:eq:heavy_ball1}
        \left(\lambda - \frac{1+\beta -\tau \eta}{2}\right)^2 = \left(\frac{1+\beta -\tau \eta}{2}\right)^2 - \beta.
    \end{equation}
    Our goal is to ensure $|\lambda|<1$. Note that~\eqref{acceleration:eq:heavy_ball1} shows that $\lambda$ is contained within a circle in the complex plane with center $c \coloneqq \frac{1+\beta -\tau \eta}{2}$ and radius $\sqrt{c^2 - \beta}$. It is easy to see that it suffices to estimate the values of $\lambda$ for maximal and minimal $c$.
    A strict upper bound for $c$ is obtained by the estimate $\tau\eta>0$ yielding $c = (1+\beta)/2$ and 
    \begin{equation}
        \begin{aligned}
            |\lambda|\leq c+c^2-\beta = \frac{3+\beta^2}{4}<1
        \end{aligned}
    \end{equation}
    due to $0\leq \beta<1$.
    Conversely, a strict lower bound for $c$ is obtained for $\tau = 2(1+\beta)/L$ and $\eta = L$ leading to $c = -(1+\beta)$ and a similar estimate as above. Thus, we find $\rho(\mathcal{H})<1$. Noting that $z_k = o(x_k-x^*)$ the result follows where we leave the remaining details as an exercise.
\end{proof}


\section{Nesterov acceleration}
Polyak's method is in fact optimal in the sense that it achieves the best possible convergence rate. However, this is only true for strongly convex and twice continuously differentiable $f$. The next acceleration we consider achieves optimal convergence under significantly weaker conditions. We consider now again problems of the form 
\begin{equation}
    \min_x F(x)\coloneqq f(x) + g(x)
\end{equation}
where only $f$ is $L$-smooth. The update of the the \enquote{fast iterative shrinkage-thresholding algorithm} (FISTA, or fast proximal gradient method) reads as
\begin{equation}
    \begin{cases}
        t_{k+1} = \frac{1+\sqrt{1+4t_k^2}}{2}\\
        \beta_k = \frac{t_k-1}{t_{k+1}}\\
        y_k = x_k + \beta_k(x_k-x_{k-1})\\
        x_{k+1} = \prox_{\tau g}(y_k - \tau \nabla f(y_k)).
    \end{cases}
\end{equation}
In the case $g\equiv 0$ the method closely resebles the heavy ball algorithm. However, the gradient is now evaluated after adding the inertia term. Moreover the convergence relies on a subtle adaptive choice of the inertia parameter.

\begin{lemma}
    Set $t_1 = 1$. It holds true that 
    \begin{equation}
        t_k \geq \frac{k+1}{2}
    \end{equation}
    for all $k$.
\end{lemma}
\begin{proof}
    The proof follows by straightforward induction. The assertion holds true for $k=1$ and, using the induction hypothesis, we find
    \begin{equation}
        t_{k+1} = \frac{1+\sqrt{1+4t_k^2}}{2}\geq \frac{1+\sqrt{1+(k+1)^2}}{2}\geq \frac{1+(k+1)}{2}.
    \end{equation}
\end{proof}

\begin{lemma}[Fundamental prox-grad inequality]
    Denote 
    \begin{equation}
        P_\tau(x) = \prox_{\tau g}(x-\tau \nabla f(x)).
    \end{equation}
    For any $\tau\leq \frac{1}{L}$ it holds true that 
    \begin{equation}
        F(x) - F(P_\tau (y))\geq \frac{1}{2\tau}\|x-P_\tau (y)\|^2 - \frac{1}{2\tau}\|x-y\|^2 + \ell_f(x,y)
    \end{equation}
    where $\ell_f(x,y) = f(x) - f(y) - \langle x-y,\nabla f(y)\rangle$
\end{lemma}
\begin{proof}
    First note that we can write
    \begin{equation}
        \begin{aligned}
            P_\tau (y) 
            =& \arg\min_z \frac{1}{2\tau}\|z-(y-\tau \nabla f(y))\|^2 + g(z)\\
            =& \arg\min_z f(y) + \langle z-y,\nabla f(y)\rangle + \frac{1}{2\tau}\|z-y\|^2 + g(z).
        \end{aligned}
    \end{equation}
    As a sidenote, this shows that the proximal-gradient method can be interpreted as a proximal method of $g$ plus a linear approximation of $f$. Let us denote 
    \begin{equation}
        \phi(z) = f(y) + \langle z-y,\nabla f(y)\rangle + \frac{1}{2\tau}\|z-y\|^2 + g(z).
    \end{equation}
    By convexity of and $g$, $\phi$ is $1/\tau$-strongly convex and we find
    \begin{equation}
        \phi(x) - \phi(P_\tau (y))\geq \frac{1}{2\tau}\|x-P_\tau (y)\|^2
    \end{equation}
    Moreover, by convexity and $L$-smoothness of $f$ and the fact that $\frac{1}{\tau}\geq L$
    \begin{equation}
        f(z) \leq f(y) + \langle z-y,\nabla f(y)\rangle + \frac{L}{2}\|y-z\| \leq f(y) + \langle z-y,\nabla f(y)\rangle + \frac{1}{2\tau}\|z-y\|^2
    \end{equation}
    and thus
    \begin{equation}
        \phi(x) - F(P_\tau (y))\geq \frac{1}{2\tau}\|x-P_\tau (y)\|^2
    \end{equation}
    Pugging in $\phi$ yields
    \begin{equation}
        f(y) + \langle x-y,\nabla f(y)\rangle + \frac{1}{2\tau}\|x-y\|^2 + g(x) - F(P_\tau (y))\geq \frac{1}{2\tau}\|x-P_\tau (y)\|^2
    \end{equation}
    concluding the proof.
\end{proof}
\begin{theorem}[Convergence of FISTA]
    Let the iterates of FISTA be initialized as $x_0=x_1$ and $t_1=1$ and the step size $\tau\leq \frac{1}{L}$ with $L$ the Lipschitz constant of $\nabla f$. Then we obtain the convergence 
    \begin{equation}
        F(x_k) - F(x^*)\leq \frac{C}{k^2 \tau}
    \end{equation}
    with a constnat $C$ depending on the initialization.
\end{theorem}
\begin{proof}
    Choosing $x = t_{k+1}^{-1}x^* + (1-t_{k+1}^{-1})x_k$ and $y=y_k$ in the fundamental prox-grad inequality and noting that by convexity of $f$ $\ell_f(x,y)\geq 0$, we obtain
    \begin{equation}\label{eq:fista1}
        \begin{aligned}
            F(t_{k+1}^{-1}x^* + &(1-t_{k+1}^{-1})x_k) - F(x_{k+1})\\
            \geq& \frac{1}{2\tau}\|t_{k+1}^{-1}x^* + (1-t_{k+1}^{-1})x_k-x_{k+1}\|^2 - \frac{1}{2\tau}\|t_{k+1}^{-1}x^* + (1-t_{k+1}^{-1})x_k-y_k\|^2\\
            \geq& \frac{1}{2\tau t_{k+1}^2}\|x^* + (t_{k+1}-1)x_k-t_{k+1} x_{k+1}\|^2 - \frac{1}{2\tau t_{k+1}^2}\|x^* + (t_{k+1}-1)x_k-t_{k+1}y_k\|^2.
        \end{aligned}
    \end{equation}
    Again, by convexity and the fact that by induction one easily verifies $t_k\geq 1$ for all $k$, it holds true that
    \begin{equation}\label{eq:fista2}
        \begin{aligned}
            F(t_{k+1}^{-1}x^* + (1-t_{k+1}^{-1})x_k) - F(x_{k+1})
            \leq& t_{k+1}^{-1} F(x^*) + (1-t_{k+1}^{-1}) F(x_k) - F(x_{k+1})\\
            \leq& (1-t_{k+1}^{-1}) (F(x_k)-F(x^*)) - (F(x_{k+1})-F(x^*)).
        \end{aligned}
    \end{equation}
    Since $y_k = x_k + \frac{t_k-1}{t_{k+1}}(x_k-x_{k-1})$ we have
    \begin{equation}\label{eq:fista3}
        \begin{aligned}
            \|x^* + (t_{k+1}-1)x_k-t_{k+1}y_k\|^2
            = \|x^* - t_k x_k + (t_k-1)x_{k-1}\|^2
        \end{aligned}
    \end{equation}
    Combining~\eqref{eq:fista1}, \eqref{eq:fista2}, and \eqref{eq:fista3} yields
    \begin{equation}
        \begin{aligned}
            t_{k+1}^2(1-t_{k+1}^{-1}) &(F(x_k)-F(x^*)) - t_{k+1}^2(F(x_{k+1})-F(x^*))\\
            \geq& \frac{1}{2\tau}\|x^* - t_{k+1} x_{k+1}+ (t_{k+1}-1)x_k\|^2 - \frac{1}{2\tau}\|x^* - t_k x_k + (t_k-1)x_{k-1}\|^2.
        \end{aligned}
    \end{equation}
    By the update rule of FISTA, it holds true that $t_{k+1}^2(1-t_{k+1}^{-1}) = t_{k}^2$ and we have
    \begin{equation}
        \begin{aligned}
            t_{k}^2 &(F(x_k)-F(x^*)) - t_{k+1}^2(F(x_{k+1})-F(x^*))\\
            \geq& \frac{1}{2\tau}\|x^* - t_{k+1} x_{k+1}+ (t_{k+1}-1)x_k\|^2 - \frac{1}{2\tau}\|x^* - t_k x_k + (t_k-1)x_{k-1}\|^2
        \end{aligned}
    \end{equation}
    hence,
    \begin{equation}
        \begin{aligned}
            t_{k}^2 &(F(x_k)-F(x^*)) + \frac{1}{2\tau}\|x^* - t_k x_k + (t_k-1)x_{k-1}\|^2\\
            \geq& t_{k+1}^2(F(x_{k+1})-F(x^*)) + \frac{1}{2\tau}\|x^* - t_{k+1} x_{k+1}+ (t_{k+1}-1)x_k\|^2
        \end{aligned}
    \end{equation}
    and by iterating over $k$ and noting that $x_0=x_1$
    \begin{equation}
        \begin{aligned}
            t_{2}^2 &(F(x_2)-F(x^*)) + \frac{1}{2\tau}\|t_2 (x^* - x_2) + (t_2-1)(x_{1}-x^*)\|^2\\
            \geq& t_{k}^2(F(x_{k})-F(x^*)) + \frac{1}{2\tau}\|x^* - t_{k} x_{k}+ (t_{k}-1)x_{k-1}\|^2
        \end{aligned}
    \end{equation}
    Moreover, the fundamental prox-grad inequality with $x=x^*$ and $y=y_1=x_0=x_1$ yields
    \begin{equation}
        F(x^*) - F(x_2)\geq \frac{1}{2\tau}\|x^* - x_2\|^2 - \frac{1}{2\tau}\|x^* - x_1\|^2.
    \end{equation}
    Thus, we find
    \begin{equation}
        \begin{aligned}
            F(x_{k})-F(x^*)
            \leq t_k^{-2}\left(\frac{t_2^2}{2\tau}\|x^* - x_1\|^2 + \frac{1}{2\tau}\|t_2 (x^* - x_2) + (t_2-1)(x_{1}-x^*)\|^2\right)
        \end{aligned}
    \end{equation}
    Since $t_k\geq \frac{k+1}{2}$ the result follows.
\end{proof}
\chapter{Duality}\label{chapter:duality}
In this section $V$ will be a general vector space. We assume, however, that $V^{**} \cong V$. Note that this is trivially true for finite dimensional and for Hilbert spaces, where $V^*\cong V$.

\section{Fenchel duality}

\begin{defn}[Fenchel/convex conjugate]
    Let $f:V\rightarrow [-\infty,\infty]$. We define the Fenchel or convex conjugate as
    \begin{equation}
        \begin{aligned}
            f^*:V^*&\rightarrow[-\infty,\infty]\\
            f^*(x^*) &= \sup_{x\in V} \inner{x^*}{x} - f(x).
        \end{aligned}
    \end{equation}
\end{defn}

\begin{example}
    \begin{enumerate}
        \item Indicator functions: Let $f=\delta_C$ with $C\subset V$ convex. Then $f^*(x^*) = \sup_{x\in C}\inner{x}{x^*}$ is the so-called support function.
        \item Norms: Let $f = \|\emptyarg\|$ some norm. Then $f^* = \delta_{\overline{B}_{\|\emptyarg\|_*,1}(0)}$. That is, $f^*$ is the indicator function on the closed 1-ball with respect to the dual norm.
        \item Conversely, the conjugate of $f = \delta_{\overline{B}_{\|\emptyarg\|_*,1}(0)}$ is $f^* = \|\emptyarg\|_*$.
    \end{enumerate}
\end{example}


\begin{lemma}
    Let $f$ be proper and convex, then $f^*$ is proper, convex, and lsc.
\end{lemma}
\begin{proof}
    If $f$ is proper, there exists $x_0\in V$ such that $f(x_0)\in \R$ implying that for any $x^*$
    \begin{equation}
        f^*(x^*)\geq \inner{x^*}{x_0} - f(x_0)> -\infty.
    \end{equation}
    Since $f$ is proper and convex, there exists at least one $\hat x\in V$ such that $\partial f(\hat x) \neq \emptyset$ (\cf~\cite[Corollary 3.19]{beck2017first}). Let $g\in \partial f(\hat x)$. We have by definition of the subgradient
    \begin{equation}
        f(\hat x) + \inner{g}{x-\hat x}\leq f(x)
    \end{equation}
    and, thus,
    \begin{equation}
        \begin{aligned}
            f^*(\hat x)
            =& \sup_{x\in V}\inner{g}{x} - f(x)\\
            \leq& \sup_{x\in V}f(x) - f(\hat x) + \inner{g}{\hat x} - f(x)\\
            \leq& \sup_{x\in V}- f(\hat x) + \inner{g}{\hat x} <\infty.
        \end{aligned}
    \end{equation}
    Convexity and lower semi-continuity follow directly by the fact that the convex conjugate is a pointwise supremum of linear, and thus convex and lsc, functions\footnote{We have only shown this property for convexity. For lsc it is left as an exercise.}. 
\end{proof}

\begin{lemma}[Fenchel inequality]
    Let $f$ be proper. Then
    \begin{equation}
        f^*(x^*) + f(x)\geq \inner{x^*}{x}.
    \end{equation}
\end{lemma}
\begin{proof}
    Exercise.
\end{proof}

We can \emph{iterate} the process of conjugation and consider the \emph{bijonjugate}. We will denote for simplicity $f^{**} \coloneq (f^*)^*$. Since we assume that $V^{**}=V$, the biconjugate is again defined on the original space $V$.

\begin{lemma}
    It holds true that $f^{**}(x)\leq f(x)$ for any $x\in V$.
\end{lemma}
\begin{proof}
    Note first that for the conjugate we have for any $x$
    \begin{equation}
        f^*(x^*)\geq \inner{x^*}{x}-f(x).
    \end{equation}
    It follows
    \begin{equation}
        \begin{aligned}
            f^{**}(x) 
            =& \sup_{x^*\in V^*}\inner{x^*}{x} - f^*(x^*)\\
            =& \sup_{x^*\in V^*}\inf_y \inner{x^*}{x-y} + f(y)\\
            \leq& \sup_{x^*\in V^*} \inner{x^*}{x-x} + f(y)\\
            =& f(x).
        \end{aligned}
    \end{equation}
\end{proof}

\begin{lemma}
    Let $f$ be proper, convex, and lsc. Then $f^{**} = f$.
\end{lemma}
\begin{proof}
    Since we always have $f^{**}\leq f$ it only remains to show that $f^{**}\geq f$. Assume to the contrary, there exists $x_0\in V$ such that
    \begin{equation}
        f^{**}(x_0) < f(x_0).
    \end{equation}
    This means that $(x_0, f^{**}(x_0))\not\in \epi(f)$. Since $\epi(f)$ is closed and convex by assumption, we may apply Hahn-Banach to find $(g,\alpha)\in V^*\times \R$ and $c_1,c_2\in \R$ such that
    \begin{equation}
        \inner{(g,\alpha)}{(x_0,f^{**}(x_0))}<c_1<c_2<\inner{(g,\alpha)}{(x,t)},\quad (x,t)\in \epi(f).
    \end{equation}
    Plugging in $x=x_0$ and $t$ large enough we find that $\alpha\geq 0$. 
    Now assume $\alpha>0$. Then we may divide by $\alpha$ to obtain
    \begin{equation}
        \inner{\frac{g}{\alpha}}{x_0-x} + f^{**}(x_0)-f(x)\leq \frac{c_1-c_2}{\alpha}<0.
    \end{equation}
    The above is true for every $x\in \dom(f)$ and also trivially for every $x$ with $f(x)=\infty$. Taking the supremum over $x$ yields
    \begin{equation}
        \inner{\frac{g}{\alpha}}{x_0} + f^{**}(x_0)+f^*(-\frac{g}{\alpha})<0
    \end{equation}
    which contradicts Fenchel's inequality.
    In the remaining case, that is, if $\alpha=0$ we have
    \begin{equation}
        \inner{g}{x_0-x} < c_1-c_2<0.
    \end{equation}
    Now take any $\hat y\in \dom(f^*)$. We obtain for $\hat g = g-\epsilon \hat y$
    \begin{equation}
        \begin{aligned}
            \inner{\hat g}{x_0-x} + \epsilon(f^{**}(x_0)-f(x))
            \leq& \inner{g}{x_0-x} + \epsilon( \inner{-\hat y}{x_0-x} + f^{**}(x_0)-f(x))\\
            \leq& c_1-c_2 + \epsilon( \inner{\hat y}{-x_0} + f^{**}(x_0)+f^*(\hat y)).
        \end{aligned}
    \end{equation}
    Note that $\hat y$ and $x_0$ are fixed and we may choose $\epsilon>0$ sufficiently small so that the right-hand side remains strictly negativ. Denote $c \coloneqq c_1-c_2 + \epsilon( \inner{-\hat y}{x_0} + f^{**}(x_0)+f^*(\hat y))<0$. Dividing by $\epsilon$ it follows
    \begin{equation}
        \begin{aligned}
            \inner{\frac{\hat g}{\epsilon}}{x_0-x} + f^{**}(x_0)-f(x)\frac{c}{\epsilon}
        \end{aligned}
    \end{equation}
    and once again taing the supremum over all $x$ yields 
    \begin{equation}
        \begin{aligned}
            \inner{\frac{\hat g}{\epsilon}}{x_0} + f^{**}(x_0)+f^*(-\frac{\hat g}{\epsilon})\frac{c}{\epsilon}<0
        \end{aligned}
    \end{equation}
    contradicting Fenchel's inequality. Thus, in any case the assumption $f^{**}(x_0)<f(x_0)$ leads to a contradiction, implying $f^{**}(x_0)\geq f(x_0)$ and concluding the proof.
\end{proof}

\begin{rem}
    When $f$ is not convex and closed, $f^{**}$ is, in fact, a convex and closed relaxation of $f$, referred to as $\Gamma$-regularization. More precisely, denote
    \begin{equation}
        \Gamma(V)=\{\ell:V\rightarrow [-\infty,\infty]\;|\; \ell\text{ is proper, closed, and convex}\}
    \end{equation}
    then
    \begin{equation}
        f^{**}(x) = \sup_{\substack{\ell\in \Gamma(V),\\
        \ell\leq f}}\ell (x).
    \end{equation}
\end{rem}

The conjugate admits an intricate relation with the subgradient as the following lemma shows.
\begin{lemma}\label{duality:lemma:subdiff_conjugate}
    Let $f$ be proper, convex, and closed. Then the following statements are equivalent:
    \begin{enumerate}
        \item $f(x) + f^*(x^*) = \inner{x^*}{x}$.\label{duality:item1}
        \item $x^*\in \partial f(x)$\label{duality:item2}
        \item $x\in \partial f^*(x^*)$\label{duality:item3}
    \end{enumerate}
\end{lemma}
\begin{proof}\
    \begin{enumerate}
        \item[\ref{duality:item1}$\rightarrow$\ref{duality:item2}:] \ref{duality:item1} implies that for any $z\in V$
        \begin{equation}
            \begin{aligned}
                -f(x) + \inner{x^*}{x} = f^*(x^*) \geq \inner{x^*}{z} - f(z)
            \end{aligned}
        \end{equation}
        which by definition implies $x^*\in \partial f(x)$.
        \item[\ref{duality:item2}$\rightarrow$\ref{duality:item1}:]
        Conversely, \ref{duality:item2} yields for any $z$
        \begin{equation}
            \begin{aligned}
                -f(x) + \inner{x^*}{x} \geq \inner{x^*}{z} - f(z).
            \end{aligned}
        \end{equation}
        By taking the supremum over $z$ we obtain $\inner{x^*}{x} \geq f(x) + f^*(x^*)$ which together with Fenchel's inequality shows the result.
        \item[\ref{duality:item1}$\leftrightarrow$\ref{duality:item3}:]
        Since $f$ is proper, convex, and closed, $\ref{duality:item1}$ is equivalent to
        \begin{equation}
            f^{**}(x) + f^*(x^*) = \inner{x^*}{x}.
        \end{equation}
        The proof follows then by repeating the arguments from above.
    \end{enumerate}
\end{proof}

\begin{theorem}[Moreau identity]\label{duality:thm:Moreau:id}
    Let $V$ be an inner product space and $f:V\rightarrow [-\infty,\infty]$ proper, lsc, and convex. Then 
    \begin{equation}
        x = \prox_f(x) + \prox_{f^*}(x).
    \end{equation}
\end{theorem}
\begin{proof}
    First of all, note that both proximal maps are well-defined and single-valued (why?). We denote $p = \prox_f(x)$ for simplicity. By the optimality conditions we have $x-p\in\partial f(p)$ and by~\cref{duality:lemma:subdiff_conjugate}, thus, $p = x - (x-p)\in \partial f^*(x-p)$. The latter implies $\prox_{f^*}(x) = x-p$ concluding the proof.
\end{proof}

\begin{theorem}[Fenchel-Rockafellar]\label{duality:eq:fenchel_duality}
    Let $f:X\rightarrow [-\infty,\infty]$, $g:Y\rightarrow[-\infty,\infty]$ both proper, convex and, closed and $K:X\rightarrow Y$ a linear, bounded operator. Moreover, assume that there exists $x_0 \in \dom(f) \cap \dom(g\circ K)$ with $K x_0 \in \dom(g)^\circ$. Assume the (primal) problem 
    \begin{equation}
        \min_{x} f(x) + g(Kx).
    \end{equation}
    admits a solution $\hat x$. Then also the dual problem
    \begin{equation}
        \begin{aligned}
            \max_{x^*} -f^*(-K^*x^*) - g^*(x^*).
        \end{aligned}
    \end{equation}
    admits a solution $\hat x^*$ and we have strong duality, that is,
    \begin{equation}
        \begin{aligned}
            \min_{x} f(x) + g(Kx)
            = \max_{x^*} -f^*(-K^*x^*) - g^*(x^*).
        \end{aligned}
    \end{equation}
    Moreover, the primal und dual solutions satisfy
    \begin{equation}
        \begin{cases}
            -K^*\hat x^*&\in \partial f(\hat x),\\
            \hat x^* &\in \partial g(K\hat x).
        \end{cases}
    \end{equation}
\end{theorem}
\begin{proof}
    First of all one trivially finds
    \begin{equation}\label{duality:eq:fenchel_rocka}
        \begin{aligned}
            \inf_{x} f(x) + g(Kx)
            =& \inf_{x} f(x) + g^{**}(Kx)\\
            =& \inf_{x}\sup_{x^*} f(x) + \inner{Kx}{x^*} - g^*(x^*)\\
            \geq & \sup_{x^*} \inf_{x} f(x) + \inner{Kx}{x^*} - g^*(x^*)\\
            = & \sup_{x^*} -\sup_{x} -f(x) + \inner{x}{-K^*x^*} - g^*(x^*)\\
            = & \sup_{x^*} -f^*(-K^*x^*) - g^*(x^*).
        \end{aligned}
    \end{equation}
    So it remains only to proof the converse inequality.
    For the minimizer $\hat x$ it holds true that $0 \in \partial f(\hat x) + K^*\partial g(K\hat x)$. In other words, there exists $\hat x^*\in \partial g(K\hat x)$ such that $-K^*\hat x^*\in \partial f (\hat x)$. By~\cref{duality:lemma:subdiff_conjugate} this implies
    \begin{equation}
        \begin{aligned}
            g(K\hat x) + g^*(\hat x^*) =& \inner{\hat x^*}{K\hat x},\\
            f(\hat x) + f^*(-K^*\hat x^*) 
            =& -\inner{K^*\hat x^*}{\hat x} = -\inner{\hat x^*}{K\hat x}.
        \end{aligned}
    \end{equation}
    Therefore, 
    \begin{equation}
        \begin{aligned}
            \inf_x f(x) + g(Kx) 
            =& f(\hat x) + g(K\hat x) \\
            =& \big(-f^*(-K^*\hat x^*) - \inner{\hat x^*}{K\hat x}\big) + \big(-g^*(\hat x^*) + \inner{\hat x^*}{K\hat x}\big)\\
            =& -f^*(-K^*\hat x^*) -g^*(\hat x^*)\\
            \leq& \sup_{x^*} -f^*(-K^*x^*) - g^*(x^*)
        \end{aligned}
    \end{equation}
    which implies equality in~\eqref{duality:eq:fenchel_rocka} as well as that $\hat x^*$ is the solution of the dual problem.
\end{proof}

\section{Primal-dual optimization}

\Cref{duality:eq:fenchel_duality} opens up the possibility for several methods to find solutions of problems of the form
\begin{equation}\label{duality:qe:problem}
    \min_x f(x) + g(Kx).
\end{equation}
Indeed, by strong duality we may tackle the primal, the dual, or the saddle point problem
\begin{equation}\label{duality:qe:saddlepoint}
    \min_x \max_y \Lc(x,y)\coloneqq f(x) -g^*(y) + \inner{Kx}{y}
\end{equation}
The main point of switching between primal/dual formulations is to \emph{transfer} the operator $K$. Indeed, if for instance $g$ is non-differentiable so that we would like to apply proximal steps to $g$, the operator $K$ in~\eqref{duality:qe:problem} will render $\prox_{g\circ K}$ inexplicit in general. Conversely, in the saddle-point formulation~\eqref{duality:qe:saddlepoint} the operator $K$ appears only in the scalar product. Note, moreover, that by~\cref{duality:thm:Moreau:id} the prox of a conjugate function is as easy or difficult to compute as the prox of the original function.

We define the primal-dual gap as
\begin{equation}
    \Gc(x,y) = (f(x) + g(Kx)) - (-f^*(y) - g^*(-K^*y)).
\end{equation}

\begin{lemma}
    We have $\Gc(x,y)\geq 0$ and $\Gc(\hat x,\hat y) = 0$ if and only if $\hat x$ solves the primal problem and $\hat y$ solves the dual problem. The gap admits the representation
    \begin{equation}
        \Gc(x,y) = \sup_{y'}\Lc(x,y') - \inf_{x'}\Lc(x',y).
    \end{equation}
    In particular, $(\hat x,\hat y)$ are optimal for the primal and dual problem, respectively, if and only if for any $x,y$
    \begin{equation}
        \Lc(\hat x,y)\leq \Lc(\hat x,\hat y)\leq \Lc(x,\hat y).
    \end{equation}
\end{lemma}
\begin{proof}
    Exercise.
\end{proof}

Performing alternating optimization in~\eqref{duality:qe:saddlepoint} with proximal steps with respect to $f$ and $g$ and gradient steps with respect to the inner product we obtain the following update rule known as the Arrow-Hurwicz method
\begin{equation}\label{duality:eq:AH}
    \begin{cases}
        x_{k+1} =& \prox_{\tau f}(x_k-\tau K^*y_k)\\
        y_{k+1} =& \prox_{\sigma g^*}(y_k+\sigma Kx_{k+1}).
    \end{cases}
\end{equation}
It turns out that in this form, the method admits worse properties. In order to derive a better update we rewrite the algorithm in a clever way. Using that $\prox_{\tau f}(x) = (\Id + \tau \partial f)^{-1}(x)$ one can easily verify that~\eqref{duality:eq:AH} amounts to
\begin{equation}
    \begin{cases}
        x_{k+1} +\tau \partial f(x_{k+1}) =& x_k - \tau K^*y_k\\
        y_{k+1} +\sigma \partial g^*(y_{k+1}) =& y_k + \sigma Kx_{k+1}.
    \end{cases}
\end{equation}
Dividing by $\sigma$ and $\tau$ and using the notation $z = (x,y)$ we may rewrite this as
\begin{equation}
    \begin{bmatrix}
        \frac{1}{\tau}\Id+ \partial f&0\\
        -K&\frac{1}{\sigma}\Id + \partial g^*
    \end{bmatrix}
    z_{k+1} = 
    \begin{bmatrix}
        \frac{1}{\tau}\Id& - K^*\\
        0& \frac{1}{\sigma}\Id
    \end{bmatrix}
    z_k
\end{equation}
Even without thinking about intricacies of a proof, one may expect that a more \emph{symmetrical} update could be preferred (in particular, due to favourable properties of symmetric matrices). Thus, we instead consider the following scheme which does not affect potential fixed points
\begin{equation}
    \begin{bmatrix}
        \frac{1}{\tau}\Id+ \partial f&0\\
        -2K&\frac{1}{\sigma}\Id + \partial g^*
    \end{bmatrix}
    z_{k+1} = 
    \begin{bmatrix}
        \frac{1}{\tau}\Id& -K^*\\
        -K & \frac{1}{\sigma}\Id
    \end{bmatrix}
    z_k
\end{equation}
leading to a symmetric matrix on the right-hand side which is positive definite as long as $\sigma\tau \|K\|^2<1$.
Going back to proximal operators this so-called primal-dual hybrid gradient method, often referred to as the Chambolle-Pock method due to~\cite{chambolle2011first} reads as
\begin{equation}\label{duality:eq:pdhg}
    \begin{cases}
        x_{k+1} =& \prox_{\tau f}(x_k-\tau K^*y_k)\\
        y_{k+1} =& \prox_{\sigma f}(y_k+\sigma K(2x_{k+1}-x_k)).
    \end{cases}
\end{equation}

For the analysis we will rewrite the PDHG in an abstract way as
\begin{equation}\label{duality:eq:pdhg2}
    \begin{cases}
        x^+ =& \arg\min_x \frac{1}{2\tau}|x-x^-|^2 + \inner{x}{K^*\bar y} + f(x)\\
        y^+ =& \arg\min_x \frac{1}{2\sigma}|y-y^-|^2 - \inner{y}{K\bar x} + g^*(y)
    \end{cases}
\end{equation}
where the plus superscript denotes the updated variables, the minus superscript the old iterates and the bar denotes the value of $x$ used for the $y$ update and the value of $y$ used in the $x$ update, respectively.

\begin{lemma}
    The update rule~\eqref{duality:eq:pdhg2} satisfies
    \begin{equation}
        \begin{aligned}
            \Lc(x^+,y) - \Lc(x,y^+)\leq& \frac{1}{2\tau}|x-x^-|^2 + \frac{1}{2\sigma}|y-y^-|^2\\
            &-\frac{1}{2\tau}|x-x^+|^2 - \frac{1}{2\sigma}|y-y^+|^2
            -\frac{1}{2\tau}|x^+-x^-|^2 - \frac{1}{2\sigma}|y^+-y^-|^2\\
            &+\inner{x^+-x}{K^*(y-\bar{y})}  - \inner{y^+-y}{K(x-\bar{x})}.
        \end{aligned}
    \end{equation}
\end{lemma}
\begin{proof}
    By strong convexity the updates $x^+$ and $y^+$ satisfy
    \begin{equation}
        \begin{aligned}
            \frac{1}{2\tau}|x^+-x^-|^2 + \inner{x^+}{K^*\bar y} + f(x^+) + \frac{1}{2\tau}|x-x^+|^2
            \leq& \frac{1}{2\tau}|x-x^-|^2 + \inner{x}{K^*\bar y} + f(x)\\
            \frac{1}{2\sigma}|y^+-y^-|^2 - \inner{y^+}{K\bar x} + g^*(y^+) + \frac{1}{2\tau}|y-y^+|^2
            \leq& \frac{1}{2\sigma}|y-y^-|^2 - \inner{y}{K\bar x} + g^*(y).
        \end{aligned}
    \end{equation}
    Summing the two inequalities and rearranging terms leads to the desired result.
\end{proof}

\begin{theorem}[Convergence of PDHG]
    Consider the updates~\eqref{duality:eq:pdhg} and assume the step sizes satisfy $\tau\sigma\leq \|K\|^{-2}$. Define the ergodic means $X_K = \frac{1}{K}\sum_{k=0}^{K-1}x_k$ and $Y_K = \frac{1}{K}\sum_{k=0}^{K-1}y_k$. Then it holds true that 
    \begin{equation}
        \Lc(X_{K},y) - \Lc(x,Y_{K})\leq \frac{1}{K}\bigg[\frac{1}{2\tau}|x-x_0|^2 + \frac{1}{2\sigma}|y-y_0|^2 - \inner{y-y_0}{K(x-x_{0})}\bigg].
    \end{equation}
\end{theorem}
\begin{rem}
    Note that the step sizes do not depend on the functions $f$ and $g$ in any way but only on the operator $K$.

\end{rem}
\begin{proof}
    \begin{equation}
        \begin{aligned}
            \Lc(x_{k+1},y) - \Lc(x,y_{k+1})
            \leq& \frac{1}{2\tau}|x-x_k|^2 + \frac{1}{2\sigma}|y-y_k|^2\\
            &-\frac{1}{2\tau}|x-x_{k+1}|^2 - \frac{1}{2\sigma}|y-y_{k+1}|^2
            -\frac{1}{2\tau}|x_{k+1}-x_k|^2 - \frac{1}{2\sigma}|y_{k+1}-y_k|^2\\
            &+\inner{x_{k+1}-x}{K^*(y-y_k)}  - \inner{y_{k+1}-y}{K(x-x_{k+1})} + \inner{y_{k+1}-y}{K(x_{k+1}-x_k)}\\
            \leq& \bigg[\frac{1}{2\tau}|x-x_k|^2 + \frac{1}{2\sigma}|y-y_k|^2 - \inner{y-y_k}{K(x-x_{k})}\bigg]\\
            &-\bigg[\frac{1}{2\tau}|x-x_{k+1}|^2 + \frac{1}{2\sigma}|y-y_{k+1}|^2 - \inner{y-y_{k+1}}{K(x-x_{k+1})}\bigg]\\
            &-\bigg[ \frac{1}{2\tau}|x_{k+1}-x_k|^2 + \frac{1}{2\sigma}|y_{k+1}-y_k|^2 - \inner{x_{k+1}-x_k}{K^*(y_{k+1}-y_{k})}\bigg].
        \end{aligned}
    \end{equation}
    Note that all expressions within the parentheses are non-negative due to $\inner{Kx}{y}\leq \|Kx\|\|y\|\leq\|K\| (\frac{\delta}{2}\|x\|^2 + \frac{1}{2\delta}\|y\|^2)\leq \frac{1}{2\tau}\|x\|^2 + \frac{1}{2\sigma}\|y\|^2$
    where the last inequality follows with the choice $\delta = \frac{1}{\tau\|K\|}$ using that $\|K\|^2\leq \frac{1}{\tau\sigma}$.
    Summing over $k$ yields
    \begin{equation}
        \sum_{k=0}^{K-1} \Lc(x_{k},y) - \Lc(x,y_{k}) \leq \bigg[\frac{1}{2\tau}|x-x_0|^2 + \frac{1}{2\sigma}|y-y_0|^2 - \inner{y-y_0}{K(x-x_{0})}\bigg].
    \end{equation}
    Since $(x,y)\mapsto \Lc(x,y') - \Lc(x',y)$ is convex, we obtain
    \begin{equation}
        \Lc(X_{K},y) - \Lc(x,Y_{K})\leq \frac{1}{K}\sum_{k=0}^{K-1} \Lc(x_{k},y) - \Lc(x,y_{k}) \leq \frac{1}{K}\bigg[\frac{1}{2\tau}|x-x_0|^2 + \frac{1}{2\sigma}|y-y_0|^2 - \inner{y-y_0}{K(x-x_{0})}\bigg].
    \end{equation}
\end{proof}

\section{Alternating direction method of multipliers (ADMM)}
Consider the constrained optimization problem
\begin{equation}\label{eq:admm1}
    \min_{Ax + By = z} f(x) + g(y).
\end{equation}
We can reformulate this problem as an unconstrained saddle point problem of the \emph{augmented Lagrangian}, that is, for any $\gamma>0$ we consider
\begin{equation}\label{eq:admm2}
    \min_{x,y}\max_{\lambda} L_\gamma (x,y,\lambda) \coloneqq f(x) + g(y) + \inner{\lambda}{Ax + By - z} + \frac{\gamma}{2}\|Ax + By - z\|^2
\end{equation}
One can easily show the following:
\begin{lemma}
    The problems~\eqref{eq:admm1} and \eqref{eq:admm2} are equivalent.
\end{lemma}
\begin{proof}
    Exercise!
\end{proof}
A straightforward way to solve the saddle-point problem is by performing in an alternating fashion an optimization with respect to $x$, and $y$, and then a gradient ascent step for $\lambda$.
\begin{equation}\label{duality:eq:admm_algo}
    \begin{cases}
        x_{k+1} = \arg\min_{x} f(x) + \inner{\lambda_k}{Ax} + \frac{\gamma}{2}\|Ax + By_k - z\|^2\\
        y_{k+1} = \arg\min_{x} g(y) + \inner{\lambda_k}{By} + \frac{\gamma}{2}\|Ax_{k+1} + By - z\|^2\\
        \lambda_{k+1} = \lambda_k + \gamma (Ax_{k+1} + By_{k+1} - z).
    \end{cases}
\end{equation}
Note that we used $\gamma$ as the step size for the gradient ascent. An intuitive motivation for this choice is the fact that by by definition of the update we have
\begin{equation}\label{duality:eq:admm3}
    \begin{aligned}
        0 = \nabla g(y_{k+1}) + B^* \lambda_k + \gamma B^*(Ax_{k+1} + By_{k+1}-z).
    \end{aligned}
\end{equation}
Multiplying the $\lambda$-update by $B^*$ and inserting~\eqref{duality:eq:admm3} we find
\begin{equation}
    B^*\lambda_{k+1} = -\nabla g(y_{k+1}).
\end{equation}
This, however, impliey that $\nabla_y L_0(x_{k+1}, y_{k+1},\lambda_{k+1})=0$. If we would perform the $x$ and $y$ optimization in~\eqref{duality:eq:admm_algo} jointly, we would obtain the same with respect to $x$ implying that every iterate of the dual variable $\lambda$ is a critical point.

We can show the following very basic convergence result for ADMM.
\begin{theorem}
    Let $f,g:\R^d\rightarrow [-\infty,\infty]$ be proper, closed, and convex. Assume that there exists a saddle point of $L_0(x,y,\lambda)$. Then it holds true that 
    \begin{equation}
        f(x_k)+g(y_k) \rightarrow \min_{Ax + By = z} f(x) + g(y).
    \end{equation}
\end{theorem}
\begin{proof}
    The proof uses an important technique, namely that of \emph{Lyapunv functionals}: We define a functional $V(y,\lambda)$ which is non-negative (lower-bounded) and decreasing along the iteration. Essentially, the Lyapunov functional serves as a substitute for the objective when proving a descent of the objective itself is not possible. In this specific case, we will use the Lyapunov functional 
    \begin{equation}
        V(y,\lambda) = \gamma \|B(y-y^*)\|^2 + \gamma^{-1} \|\lambda-\lambda^*\|^2
    \end{equation}
    where $(x^*,y^*,\lambda^*)$ is a saddle point for $L$. The descent of the Lyapunov functional will yield convergence using typical telescope sum arguments. Showing this descent, however, requires some technical estimates. For simplicity we denote the residuals as $r_k \coloneqq Ax_k + By_k- z$ and the objective value as $p_k \coloneqq f(x_k) + g(y_k)$ and similarly $r^*$ and $p^*$ for the saddle point. The proof is split into three steps.
    \paragraph{Step 1} 
    Since $(x^*,y^*,\lambda^*)$ is a saddle-point for $L_0$ it holds for any $k$, $L_0(x^*,y^*,\lambda^*)\leq L_0(x_{k+1},y_{k+1},\lambda^*)$. This is equivalent to
    \begin{equation}\label{duality:eq:proof_admm1}
        p^*-p_{k+1}\leq \inner{\lambda^*}{r_{k+1}}.
    \end{equation}
    \paragraph{Step 2}
    On the other hand-side, we find by optimality of $x_{k+1}$, $y_{k+1}$ in~\eqref{duality:eq:admm_algo}
    \begin{equation}\label{duality:eq:proof_admm2}
        \begin{aligned}
            0 &\in \partial f(x_{k+1}) + A^*\lambda_k + \gamma A^*(r_{k+1} + B(y_k-y_{k+1}))\\
            0 &\in \partial g(y_{k+1}) + B^*\lambda_k + \gamma B^*r_{k+1}.
        \end{aligned}
    \end{equation}
    Note, moreover, that the $\lambda$ update reads as $\lambda_{k+1} = \lambda_k + \gamma r_{k+1}$ so that
    \begin{equation}
        \begin{aligned}
            A^* \lambda_{k+1} &= A^*\lambda_k + \gamma A^*r_{k+1}\\
            B^* \lambda_{k+1} &= B^*\lambda_k + \gamma B^*r_{k+1}
        \end{aligned}
    \end{equation}
    which, inserted into~\eqref{duality:eq:proof_admm2} yields
    \begin{equation}
        \begin{aligned}
            0 &\in \partial f(x_{k+1}) + A^*(\lambda_{k+1} + \gamma B(y_k-y_{k+1}))\\
            0 &\in \partial g(y_{k+1}) + B^*\lambda_{k+1}.
        \end{aligned}
    \end{equation}
    This implies that $x_{k+1}$ minimizes
    \begin{equation}
        x\mapsto f(x) + \inner{A x}{\lambda_{k+1} + \gamma B(y_k-y_{k+1})}
    \end{equation}
    and $y_{k+1}$ minimizes
    \begin{equation}
        y\mapsto g(y) + \inner{By}{\lambda_{k+1}}.
    \end{equation}
    Consequently, we have by optimality
    \begin{equation}
        \begin{aligned}
            f(x_{k+1}) + &\inner{A x_{k+1}}{\lambda_{k+1} + \gamma B(y_k-y_{k+1})} + g(y_{k+1}) + \inner{By_{k+1}}{\lambda_{k+1}}\\
            &\leq
            f(x^*) + \inner{Ax^*}{\lambda_{k+1} + \gamma B(y_k-y_{k+1})} + g(y^*) + \inner{By^*}{\lambda_{k+1}}
        \end{aligned}
    \end{equation}
    Since the saddle point necessarily satisfies $Ax^* + By^* = z$ by also rearranging a little bit, this implies
    \begin{equation}
        \begin{aligned}
            p_{k+1} - p^*
            \leq \gamma \inner{A (x^*-x_{k+1})}{B(y_k-y_{k+1})} - \inner{r_{k+1}}{\lambda_{k+1}}
        \end{aligned}
    \end{equation}
    Lastly, we may insert $A(x^* - x_{k+1}) = z - By^* - (r_{k+1} - z - By_{k+1}) = -r_{k+1} + By_{k+1}$ to obtain
    \begin{equation}\label{duality:eq:proof_admm3}
        \begin{aligned}
            p_{k+1} - p^*
            \leq -\gamma \inner{-r_{k+1} + B(y_{k+1}-y^*)}{B(y_{k+1}-y_k)} - \inner{r_{k+1}}{\lambda_{k+1}}.
        \end{aligned}
    \end{equation}
    \paragraph{Step 3}
    Summing~\eqref{duality:eq:proof_admm1} and~\eqref{duality:eq:proof_admm3} and multiplying by $2$ yields
    \begin{equation}\label{duality:eq:proof_admm4}
        -2\gamma \inner{r_{k+1}}{B(y_{k+1}-y_k)} + 2\gamma\inner{B(y_{k+1}-y^*)}{B(y_{k+1}-y_k)} + 2\inner{r_{k+1}}{\lambda_{k+1}-\lambda^*}\leq 0.
    \end{equation}
    We will manipulate this inequality. We can rewrite the third term using that $\lambda_{k+1} = \lambda_k + \gamma r_{k+1}$ leading to 
    \begin{equation}
        2\inner{r_{k+1}}{\lambda_{k+1}-\lambda^*}
        = \gamma\|r_{k+1}\|^2 + \gamma\|r_{k+1}\|^2 + 2\inner{r_{k+1}}{\lambda_k -\lambda^*}.
    \end{equation}
    Substituting $r_{k+1} = (\lambda_{k+1}-\lambda_k)/\gamma$ in the last two terms yields
    \begin{equation}
        \gamma\|r_{k+1}\|^2 + \gamma^{-1}\|\lambda_{k+1}-\lambda_k\|^2 + 2\gamma^{-1}\inner{\lambda_{k+1}-\lambda_k}{\lambda_k -\lambda^*}.
    \end{equation}
    By replacing above $\lambda_{k+1}-\lambda_k = \lambda_{k+1}-\lambda^* + \lambda^* - \lambda_k$ this is equivalent to
    \begin{equation}
        \gamma\|r_{k+1}\|^2 + \gamma^{-1}\big(\|\lambda_{k+1}-\lambda^*\|^2 - \|\lambda_{k}-\lambda^*\|^2\big).
    \end{equation}
    Thus,~\eqref{duality:eq:proof_admm4} is equivalent to 
    \begin{equation}\label{duality:eq:proof_admm5}
        \begin{aligned}
            0\geq&-2\gamma \inner{r_{k+1}}{B(y_{k+1}-y_k)} + 2\gamma\inner{B(y_{k+1}-y^*)}{B(y_{k+1}-y_k)} + \gamma\|r_{k+1}\|^2 \\
            &+ \gamma^{-1}\big(\|\lambda_{k+1}-\lambda^*\|^2 - \|\lambda_{k}-\lambda^*\|^2\big).
        \end{aligned}
    \end{equation}
    We will now reqrite the first three terms 
    \begin{equation}\label{duality:eq:proof_admm6}
        -2\gamma \inner{r_{k+1}}{B(y_{k+1}-y_k)} + 2\gamma\inner{B(y_{k+1}-y^*)}{B(y_{k+1}-y_k)} + \gamma\|r_{k+1}\|^2
    \end{equation}
    Noting that 
    \begin{equation}
        -2\inner{r_{k+1}}{B(y_{k+1}-y_k)} + \|r_{k+1}\|^2 
        = \|r_{k+1} - B(y_{k+1}-y_k)\|^2 - \|B(y_{k+1}-y_k)\|^2
    \end{equation}
    we find that~\eqref{duality:eq:proof_admm5} is equivalent to 
    \begin{equation}
        \gamma \|r_{k+1} - B(y_{k+1}-y_k)\|^2 - \gamma\|B(y_{k+1}-y_k)\|^2 + 2\gamma\inner{B(y_{k+1}-y^*)}{B(y_{k+1}-y_k)}
    \end{equation}
    which is equivalent to
    \begin{equation}
        \gamma \|r_{k+1} - B(y_{k+1}-y_k)\|^2 + \gamma\big( \|B(y_{k+1}-y^*)\|^2 - \|B(y_{k}-y^*)\|^2\big).
    \end{equation}
    In total,~\eqref{duality:eq:proof_admm5} then yields
    \begin{equation}
        0\geq \gamma \|r_{k+1} - B(y_{k+1}-y_k)\|^2 + \gamma\big( \|B(y_{k+1}-y^*)\|^2 - \|B(y_{k}-y^*)\|^2\big) + \gamma^{-1}\big(\|\lambda_{k+1}-\lambda^*\|^2 - \|\lambda_{k}-\lambda^*\|^2\big)
    \end{equation}
    Lastly, we note that
    \begin{equation}
        \|r_{k+1} - B(y_{k+1}-y_k)\|^2
        =\|r_{k+1}\|^2 + \|B(y_{k+1}-y_k)\|^2 + 2\inner{r_{k+1}}{B(y_{k+1}-y_k)}
    \end{equation}
    We want to show that the term
    \begin{equation}
        \inner{r_{k+1}}{B(y_{k+1}-y_k)} = \inner{r_{k+1}}{B(y_{k+1}-y_k)}
    \end{equation}
    is non-negative. Recall that $y_k$ minimizes $g(y) + \inner{By}{\lambda_k}$. Thus, we have
    \begin{equation}
        \begin{aligned}
            g(y_k) + \inner{By_k}{\lambda_k}
            &\leq g(y_{k+1}) + \inner{By_{k+1}}{\lambda_k}\\
            g(y_{k+1}) + \inner{By_{k+1}}{\lambda_{k+1}}
            &\leq g(y_k) + \inner{By_k}{\lambda_{k+1}}
        \end{aligned}
    \end{equation}
    and summing both inequalities yields
    \begin{equation}
        \begin{aligned}
            \inner{B(y_{k+1}-y_k)}{\lambda_{k+1}-\lambda_k}
            \leq 0 
        \end{aligned}
    \end{equation}
    and inserting $\lambda_{k+1}-\lambda_k = \gamma r_{k+1}$ yields the desired sign. Thus, we have
    \begin{equation}
        V(y_{k+1},\lambda_{k+1}) + \|r_{k+1}\|^2 + \|B(y_{k+1}-y_k)\|^2 \leq V(y_{k},\lambda_{k}).
    \end{equation}
    Summing over $k$ yields
    \begin{equation}
        \sum_{k=0}^{K-1}\|r_{k+1}\|^2 + \|B(y_{k+1}-y_k)\|^2 \leq V(y_{0},\lambda_{0}).
    \end{equation}
    By the descent of the Lyapunov functional we immediately obtain that $\lambda_k$ and $By_k$ are bounded sequences.
    Moreover, it follows
    \begin{equation}\label{duality:eq:proof_admm6}
        \begin{cases}
            r_k&\rightarrow 0\\
            B(y_{k+1}-y_k)&\rightarrow 0.
        \end{cases}
    \end{equation}
    By~\eqref{duality:eq:proof_admm1} and~\eqref{duality:eq:proof_admm3} we have
    \begin{equation}
        \begin{aligned}
            -\inner{\lambda^*}{r_{k+1}}\leq p_{k+1}-p^*\leq \inner{-r_{k+1} + B(y_{k+1}-y^*)}{B(y_{k+1}-y_k)} - \inner{r_{k+1}}{\lambda_{k+1}}.
        \end{aligned}
    \end{equation}
    Then~\eqref{duality:eq:proof_admm6} and boundedness of $By_k$ imply that also $p_{k}\rightarrow p^*$.
\end{proof}

\chapter{Stochastic Gradient Descent}
Assume we want to solve
\begin{equation}
    \min_x f(x),
\end{equation}
but we do have access only to random but unbiased estimates of $f$, respectively $\partial f$. That is, we assume we can evaluate $G(x,z)$ which satisfies
\begin{equation}\label{stochastic:eq:gradient}
    \E[G(x,Z)] \in \partial f(x)
\end{equation}
where $Z$ is some random variable.
It turns out that many of the convergence proofs can be transferred to this stochastic setting with almost no extra effort. Before showing this, let us motivate this setting.

In modern machine learning, most learning problems can be formulated as
\begin{equation}
    \min_x f(x)\coloneqq \sum_{i=1}^N f_i(x)
\end{equation}
very $N$ is usually the number of training samples. For instance, if we want to learn a parametrized map $f_\theta$ which should approximate some relation $x\mapsto y$ based on the training data $(x_i,y_i)_{i=1}^N$, the corresponding training problem could be
\begin{equation}
    \min_\theta \sum_{i=1}^N \|y_i-f_\theta(x_i)\|^2.
\end{equation}
Due to the encountered sizes of training data, it is often not possible to directly compute gradients of $f$ because of memory limitations.
If $f_i$ are differentiable a possible remedy is the stochastic update
\begin{equation}
    \begin{cases}
        \text{Choose $i\in\Ic$}\\
        x_{k+1} = x_k - \tau \nabla f_i(x_k)
    \end{cases}
\end{equation}
where $\Ic=\{1,\dots,N\}$.
Defining $Z_k$ as a sequence of iid uniformly distributed random variables in $\Ic$ the above can be written as
\begin{equation}
    x_{k+1} = x_k - \tau G(x_k,Z_k)
\end{equation}
with $G(x,z) \coloneqq \nabla f_z(x)$. One can easily check that such $G$ satisfies~\eqref{stochastic:eq:gradient}.

\begin{theorem}
    Let $f$ be $L$-Lipschitz and assume the step sizes satisfy
    \[
        \frac{\sum_{k=1}^n\tau_k}{\sum_{k=1}^n\tau_k^2}\rightarrow \infty
    \]
    as $n\rightarrow\infty$. 
    Moreover, assume the stochastic gradients are unbiased and of bounded variance, specifically, for all $x\in \R^d$ we have
    \begin{equation}
        \begin{cases}
            \E[G(x,Z)] \in &\partial f(x)\\
            \E[\|G(x,Z)\|^2] - \|\E[G(x,Z)]\|^2 &\leq \sigma^2<\infty
        \end{cases}
    \end{equation}
Then it holds $\E[f^n_{\mathrm{best}}-f(x^*)]\rightarrow 0$.
\end{theorem}
\begin{proof}
    \begin{equation}
        \begin{aligned}
            \|x_{k+1}-x^*\|^2
            &= \|\proj_C(x_{k}-\tau_k G(x_k,Z_k))-\proj_C(x^*)\|^2\\
            &\leq \|x_{k}-\tau_k G(x_k,Z_k)-x^*\|^2\\
            &= \|x_{k}-x^*\|^2 - 2\tau_k\inner{x_k-x^*}{G(x_k,Z_k)} + \tau_k^2\|G(x_k,Z_k)\|^2.
        \end{aligned}
    \end{equation}
    Note that
    \begin{equation}
        \begin{aligned}
            \E[\inner{x_k-x^*}{G(x_k,Z_k)}]
            &= \E[\E[\inner{x_k-x^*}{G(x_k,Z_k)}|x_k]]\\
            &= \E[\inner{x_k-x^*}{\E[G(x_k,Z_k)|x_k]}]\\
            &\geq \E[f(x_k) - f(x^*)]\\
            &= \E[f(x_k)] - f(x^*)
        \end{aligned}
    \end{equation}
    as well as
    \begin{equation}
        \E[\|G(x_k,Z_k)\|^2] 
        = \E[\|G(x_k,Z_k)\|^2] - \|E[G(x_k,Z_k)]\|^2 + \|E[G(x_k,Z_k)]\|^2
        \leq \sigma^2 + L^2
    \end{equation}
    Thus, taking expectation leads to 
    \begin{equation}
        \begin{aligned}
            \E[\|x_{k+1}-x^*\|^2]
            \leq \E[\|x_{k}-x^*\|^2] - 2\tau_k (\E[f(x_k)] - f(x^*)) + \tau_k^2(\sigma^2 + L^2).
        \end{aligned}
    \end{equation}
    The remaining proof is as in the deterministic case. Summing over $k$ leads to
    \begin{equation}
        \begin{aligned}
            \sum_{k=0}^{K-1}\tau_k (\E[f(x_k)] - f(x^*))\leq \frac{1}{2}\E[\|x_{0}-x^*\|^2] + (\sigma^2 + L^2)\sum_{k=0}^{K-1}\tau_k^2
        \end{aligned}
    \end{equation}
    Defining
    \begin{equation}
        \sigma_n = \frac{\sum_{k=0}^{K-1}\tau_k}{\sum_{k=0}^{K-1}\tau_k^2}
    \end{equation}
    we obtain
    \begin{equation}
        \begin{aligned}
            \sigma_n \E[f^n_{\mathrm{best}}-f(x^*)]
            =& \frac{\sum_{k=0}^{K-1}\tau_k\E[f^n_{\mathrm{best}}-f(x^*)]}{\sum_{k=0}^{K-1}\tau_k^2}\\
            \leq& \frac{\sum_{k=0}^{K-1}\tau_k\E[f(x_k)-f(x^*)]}{\sum_{k=0}^{K-1}\tau_k^2}\\
            \leq& \frac{\E[\|x_{0}-x^*\|^2]}{2\sum_{k=0}^{K-1}\tau_k^2} + \frac{(\sigma^2 + L^2)\sum_{k=0}^{K-1}\tau_k^2}{\sum_{k=0}^{K-1}\tau_k^2}\\
        \end{aligned}
    \end{equation}
    and thus
    \begin{equation}
        \E[f^n_{\mathrm{best}}-f(x^*)] \leq \frac{\E[\|x_{0}-x^*\|^2]}{2\sum_{k=0}^{K-1}\tau_k} + \frac{(\sigma^2 + L^2)\sum_{k=0}^{K-1}\tau_k^2}{\sigma_n}
    \end{equation}
    which goes to zero as $n\rightarrow\infty$.
\end{proof}
\newabbreviation{ot}{OT}{optimal transport}

\chapter{Excursion on Optimal Transport}
While not fitting perfectly into this lecture, we finish with a an excursion into \gls{ot}, a subject with immense importance in modern mathematics and machine learning. We emphasize at this point that the exposition here will not be rigorous and in particular we will omit a lot of details regarding specific properties of the underlying spaces.

\paragraph{Notation}
In the following $X$, $Y$ will be measurable spaces\footnote{usually \emph{Polish} spaces, \ie, separable, complete, metric spaces}, $\Pc(X)$, $\Pc(Y)$ will be the space of all probability measures on $X$, respectively $Y$ and $\Mc(X)$, $\Mc(Y)$ the spaces of all finite measures.
Recall that a (signed) measure $\mu$ on a set $X$ is a function $\mu:\Sigma\rightarrow \R$ where $\Sigma\subset 2^X$ is a subset\footnote{specifically, $\Sigma$ needs to be a $\sigma$-algebra, that is, a subset of $2^X$ with certain properties} of the power set of $X$ such that $\mu$ satisfies the following
\begin{enumerate}
    \item $\mu(\emptyset)=0$
    \item For $(A_n)_{n\in \N}\subset\Sigma$ all disjoint it holds $\mu(\bigcup_n A_n) = \sum_n \mu(A_n)$.
\end{enumerate}
A probability measure is a measure with $\mu(A)\geq 0$ for any $A$ and $\mu(X)=1$.

\section{Monge's and Kantorovich's optimal transport}

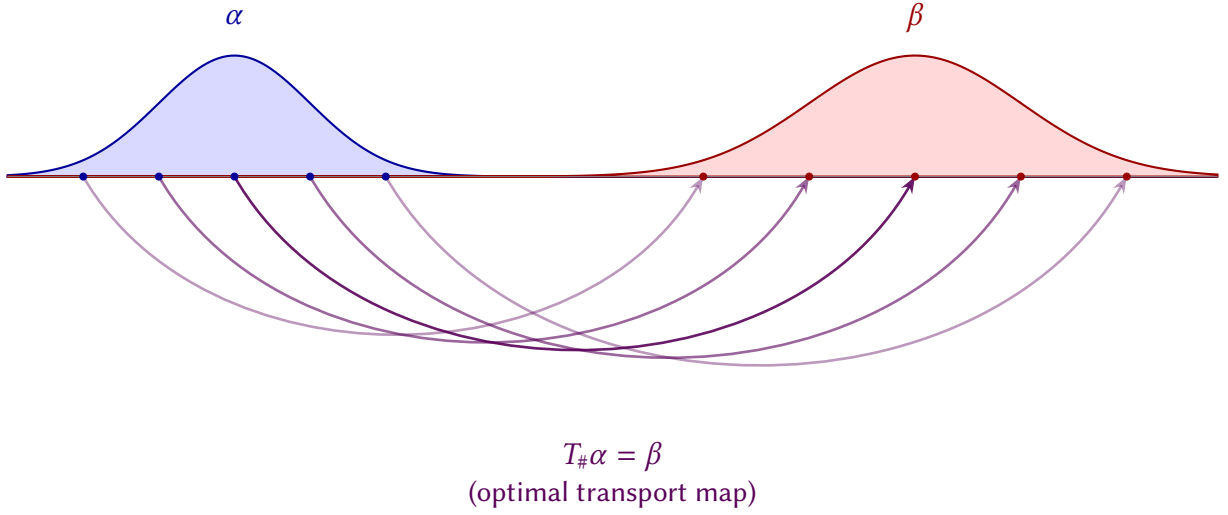
\begin{figure}
\begin{tikzpicture}[
  >=Stealth,
  font=\sffamily
]

\def\y{0}
\def\xmin{0}
\def\xmax{16}
\def\amean{3}
\def\astd{1}
\def\bmean{12}
\def\bstd{1.4}

\draw[domain=\xmin:\xmax, smooth, samples=150, variable=\x,
      blue!60!black, thick, fill=blue!15]
  plot ({\x}, {\y + 1.6*exp(-0.5*((\x-\amean)/\astd)^2)})
  -- (\xmax,\y) -- (\xmin,\y) -- cycle;

\draw[domain=\xmin:\xmax, smooth, samples=150, variable=\x,
      red!60!black, thick, fill=red!15]
  plot ({\x}, {\y + 1.6*exp(-0.5*((\x-\bmean)/\bstd)^2)})
  -- (\xmax,\y) -- (\xmin,\y) -- cycle;

\draw[gray] (\xmin,\y) -- (\xmax,\y);

\node[blue!60!black] at (\amean, 2.1) {$\alpha$};
\node[red!60!black]  at (\bmean, 2.1) {$\beta$};

\foreach \x/\opacity in {1/0.4, 2/0.6, 3/0.9, 4/0.6, 5/0.4} {
  \pgfmathsetmacro{\Tx}{\bmean + (\bstd/\astd)*(\x-\amean)}
  \draw[-{Stealth[length=2.2mm]}, line width=1pt,
        violet!70!black, opacity=\opacity]
    (\x,\y) to[out=-60, in=-120] (\Tx,\y);
  \filldraw[blue!60!black] (\x,\y) circle (1.3pt);
  \filldraw[red!60!black] (\Tx,\y) circle (1.3pt);
}

\node[violet!70!black, fill=white, fill opacity=0.85, text opacity=1,
      inner sep=2pt] at (8, -3.7) {$T_{\#}\alpha=\beta$};
\node[violet!70!black, font=\sffamily\small, fill=white, fill opacity=0.85,
      text opacity=1, inner sep=2pt] at (8, -4.2) {(optimal transport map)};
\end{tikzpicture}
\caption{Illustration of~\gls{ot}. The map $T$ transports the \emph{pile} from $\alpha$ to $\beta$.}
\label{ot:fig:ot}
\end{figure}

The typical motivation for \gls{ot} is the following. Imagine we want to carry a pile of sand from one spot to another. By conservation of mass we know that the two piles have the same total mass which we assume without loss of generality to be one. Thus, we may describe the initial pile by a probability distribution $\alpha\in\Pc(X)$ and the final pile by a probability distribution $\beta\in\Pc(Y)$. The question to be answered in \gls{ot} is:
\begin{quote}
    \centering
    \emph{How can we transport the pile $\alpha$ to $\beta$ at minimal cost?}
\end{quote}
Careful readers might notice, that it is unclear what the \emph{costs} are. We assume we have a cost function $c$ such that $c(x,y)$ is the cost of transporting one unit of mass from $x\in X$ to $y\in Y$. Each possible transport from $\alpha$ to $\beta$ is modelled by a function $T$. The condition of transporting $x$ to $y$ can mathematically be formulated by requiring that for any set $A$, $\alpha(T^{-1}(A)) = \beta(A)$. We make the following appropriate definition.
\begin{defn}[Push-forward]
    Let $\alpha\in \Pc(X)$ and $T:X\rightarrow Y$. The push-forward measure $T_\sharp\alpha\in \Pc(Y)$ is defined via $T_\sharp \alpha(A) = \alpha(T^{-1}(A))$ for any measurable $A\subset Y$.
\end{defn}

We can no formally introduce the Monge~\gls{ot} problem:
\begin{defn}[Monge \gls{ot}]
    Let $X$ and $Y$ be measurable spaces and $\alpha\in \Pc(X)$ and $\beta\in \Pc(Y)$. The Monge \gls{ot} problem is 
    \begin{equation}\label{ot:eq:monge}
        \min_{T:X\rightarrow Y} \int c(x,T(x))\dd \alpha(x),\quad\text{s.t. }T_\sharp\alpha = \beta.
    \end{equation}
\end{defn}
The Monge \gls{ot} problem admits a few downsides. For instance, one may quickly realize that the function $T$ can map each $x$ only to \emph{one} $y$ which means we cannot distribute mass located at a point $x$ to different points $y$. Therefore, whenever $\alpha$ admits non-zero mass at a specific point but $\beta$ is absolutely continuous with respect to the Lebesgue measure\footnote{that is, $\beta$ has no point masses} there is no admissible plan $T$. Moreover, depending on the choice of $c$, the problem is in general difficult to analyze as the unknown $T$ appears within the cost $c$. 

A formulation of the problem that is more general and easier to analyze is Kantorovich's \gls{ot}. Instead of a map $T:X\rightarrow Y$, we model our transportation via a distribution as well. That is, a transport plan is a probability distribution $\gamma$ on $X\times Y$ which admits $\alpha$ and $\beta$ as its marginals, that is, $\gamma(A\times Y) = \alpha(A)$ and $\gamma(X\times B) = \beta(B)$ for all $A\subset X$, $B\subset Y$. Figuratively speaking, for any $A\subset X$, $B\subset Y$, $\gamma(A\times B)$ is the mass transported from $A$ to $B$. We denote the set of all such distributions with marginals $\alpha$, $\beta$ as $\Pi(\alpha,\beta)$ and refer to them as couplings. We can now define Kantorovich's \gls{ot} problem.
\begin{defn}[Kantorovich's \gls{ot}]
    Let $X$ and $Y$ be measurable spaces and $\alpha\in \Pc(X)$ and $\beta\in \Pc(Y)$. The Kantorovich \gls{ot} problem is 
    \begin{equation}\label{ot:eq:K_ot}
        \min_{\gamma\in\Pi(\alpha,\beta)} \int c(x,y)\dd\gamma(x,y).
    \end{equation}
\end{defn}
Note that~\eqref{ot:eq:K_ot} is highly favorable over~\eqref{ot:eq:monge}. In particular, the objective is now linear.

\begin{rem}
    Note that we can formulate~\eqref{ot:eq:K_ot} very differently using random variables: Every coupling $\gamma$ can be~\emph{realized} as a random vector $(X,Y)\sim\gamma$ where $X\sim \alpha$ and $Y\sim \beta$. Then the cost satisfies
    \begin{equation}
        \int c(x,y)\dd \gamma(x,y) = \E[c(X,Y)].
    \end{equation}
\end{rem}

We state the following result.
\begin{theorem}[Existence of solutions]
    Let $c:X\times Y\rightarrow \R\cup\{\infty\}$ be bounded from below and lower semi-continuous. Then there exists an optimal transport plan $\gamma$, that is, a solution to~\eqref{ot:eq:K_ot}.
\end{theorem}
\begin{proof}[Proof sketch]
    For a detailed proof, see for instance~\cite[Theorem 4.1]{villani2009optimal}. Essentially, the proof boils down to the direct method: Assume $(\gamma_n)_n$ is a minimizing sequence. By extracting a subsequence we may assume $\lim_n \int c(x,y)\dd \gamma_n(x,y) = \inf_\gamma \int c(x,y)\dd \gamma(x,y)>-\infty$. We obtain boundedness of the sequence more or less for free, as the $\gamma_n$ are probability measures. More specifically, one can show that the set $\Pi(\alpha,\beta)$ is compact with respect to the weak topology. Therefore, by extracting another subsequence we obtain $\hat\gamma\in\Pi(\alpha,\beta)$ such that $\gamma_n\rightharpoonup \hat \gamma$. Since, moreover, $\gamma\mapsto \int c(x,y)\dd \gamma(x,y)$ is weakly lower semi-continuous, $\hat\gamma$ is a minimizer concluding the proof.
\end{proof}

An important special case of optimal transport surely has to be highlighted.

\begin{defn}[Wasserstein distances]
Let $X=Y$ and $c(x,y) = \|x-y\|^p$, $p\in[1,\infty)$. Then the expression
\begin{equation}
    W_p(\alpha,\beta) \coloneqq \bigg(\inf_{\gamma\in\Pi(\alpha,\beta)}\int c(x,y)\dd \gamma(x,y)\bigg)^{1/p}
\end{equation}
is referred to as the Wasserstein-$p$-distance.
\end{defn}
In particular, $W_p$ is a metric on the space of probability measures with finite $p$-th.

\section{Kantorovich duality}
From~\cref{chapter:duality} we are already familiar with the concept of duality. Duality plays a crucial role in~\gls{ot} and we obtain the following famous result.
\begin{theorem}[Kantorovich duality]
    Let $c:X\times Y\rightarrow \R\cup\{\infty\}$ be lower semi-continuous. We have the following duality
    \begin{equation}
        \begin{aligned}
            \inf_{\gamma\in\Pi(\alpha,\beta)}\int c(x,y) \dd \gamma(x,y) 
            = \sup_{\phi(x)+\psi(y)\leq c(x,y)}\int \phi(x)\dd\alpha(x) + \int \psi(y)\dd \beta(y).
        \end{aligned}
    \end{equation}
\end{theorem}
\begin{proof}[Proof sketch]
    We begin with the easy inequlity: Let $(\phi,\psi)$ be such that $\phi(x) + \psi(y)\leq c(x,y)$ and $\gamma\in\Pi(\alpha,\beta)$ arbitrary. We find
    \begin{equation}
        \begin{aligned}
            \int \phi(x)\dd\alpha(x) + \int \psi(y)\dd \beta(y)
            =& \int \phi(x)\dd\gamma(x,y) + \int \psi(y)\dd \gamma(x,y)\\
            =& \int \phi(x)+\psi(y)\dd\gamma(x,y)\\
            \leq& \int c(x,y) \dd \gamma(x,y).
        \end{aligned}
    \end{equation}
    Since $\gamma$ and $(\phi,\psi)$ were arbitrary we may take the infimum on the right and the supremum on the left to obtain
    \begin{equation}
        \begin{aligned}
            \sup_{\phi(x)+\psi(y)\leq c(x,y)}\int \phi(x)\dd\alpha(x) + \int \psi(y)\dd \beta(y)
            \leq& \inf_{\gamma\in\Pi(\alpha,\beta)}\int c(x,y) \dd \gamma(x,y).
        \end{aligned}
    \end{equation}
    The converse inequality turns out to be more intricate.
    First we may rewrite~\eqref{ot:eq:K_ot} as 
    \begin{equation}
        \inf \int c(x,y)\dd \gamma + \delta_{\Pi(\alpha,\beta)}(\gamma)
    \end{equation}
    with $\delta_{\Pi(\alpha,\beta)}$ the indicator function as usual. Note that we can write
    \begin{equation}\label{ot:eq:duality}
        \begin{aligned}
            \delta_{\Pi(\alpha,\beta)}(\gamma)
            =& \sup_{\phi,\psi}\int \phi(x)\dd\alpha(x) + \int \psi(y)\dd \beta(y) - \int \phi(x)+\psi(y)\dd \gamma(x,y)\\
        \end{aligned}
    \end{equation}
    To see this, note that whenever $\gamma\in\Pi(\alpha,\beta)$, the right-hand side evaluates to zero for every $(\phi,\psi)$. On the other hand-side, when $\gamma\notin\Pi(\alpha,\beta)$ then without loss of generality $\gamma_x\neq \alpha$ where $\gamma_x$ denotes the $x$-marginal of $\gamma$. But by duality, this means that there exists $\phi$ such that
    \begin{equation}
        \int \phi \dd(\alpha - \gamma_x) \neq 0.
    \end{equation}
    Considering $t\phi$ with $t\rightarrow \pm \infty$ shows the result.

    Therefore, assuming strong duality, that is, that we may exchange infimum and supremum, we obtain
    \begin{equation}
        \begin{aligned}
            \inf_{\gamma\in\Pi(\alpha,\beta)}\int c(x,y) \dd \gamma(x,y)
            =& \inf_{\gamma\in\Mc_+(X\times Y)}\sup_{\phi,\psi} \int c(x,y) \dd \gamma(x,y)\\
            &+\int \phi(x)\dd\alpha(x) + \int \psi(y)\dd \beta(y) - \int \phi(x)+\psi(y)\dd \gamma(x,y)\\
            =& \sup_{\phi,\psi}\int \phi(x)\dd\alpha(x) + \int \psi(y)\dd \beta(y) \\
            &+ \inf_{\gamma\in\Mc_+(\alpha,\beta)} \int c(x,y)- \phi(x)-\psi(y) \dd \gamma(x,y)\\
            =& \sup_{\phi(x)+\psi(y)\leq c(x,y)}\int \phi(x)\dd\alpha(x) + \int \psi(y)\dd \beta(y)
        \end{aligned}
    \end{equation}
    where in the last equality we again recognize the indicator function of the set $\{(\phi,\psi)\;|\; \phi(x) + \psi(y)\leq c(x,y)\}$ similarly as in~\eqref{ot:eq:duality}.
\end{proof}

\section{Discretizing optimal transport: The Sinkhorn-Knopp algorithm}
To discretize~\eqref{ot:eq:K_ot} we model $\alpha$ and $\beta$ as discrete measures, that is,
\begin{equation}
    \alpha = \sum_{i=1}^n a_i \delta_{x_i}
\end{equation}
for some $a_i\in \Delta_n$ and $x_i\in X$ and similarly 
\begin{equation}
    \beta = \sum_{j=1}^m b_j \delta_{y_j}.
\end{equation}
Moreover, we define $C = [C_{i,j}]_{i,j} = [c(x_i,y_j)]_{i,j}$. A transport plan then is a matrix $P\in \R^{n\times m}$ such that 
\begin{equation}
    \begin{cases}
        \sum_{j=1}^m P_{i,j} &= a_i\\
        \sum_{i=1}^n P_{i,j} &= b_j\\
        P_{i,j}&\geq 0.
    \end{cases}
\end{equation}
The marginal constraints can be compactly written as $P\1=a$, $P^t\1 = b$. The discrete \gls{ot} problem then reads as
\begin{equation}
    \begin{aligned}
        \min_{P\in \R^{n\times m}}\inner{C}{P} = \sum_{i,j}C_{i,j}P_{i,j}\\
        \text{s.t. } P\1=a,\;P^t\1 = b\; P_{i,j}\geq 0.
    \end{aligned}
\end{equation}
In practice, the problem is often regularized. That is, the non-negativity constraints implicitly handled by adding an entropic regularization
\begin{equation}
    E(P) \coloneq \sum_{i,j} P_{i,j}(\log(P_{i,j})-1),
\end{equation}
and considering the problem
\begin{equation}\label{ot:eq:disc_entropic}
    \begin{aligned}
        \min_{P\in \R^{n\times m}}\inner{C}{P} + \epsilon E(P)
        \quad \text{s.t. } P\1=a,\;P^t\1 = b.
    \end{aligned}
\end{equation}

\begin{lemma}
    The solution of the problem~\eqref{ot:eq:disc_entropic} is unique and admits the representation
    \begin{equation}
        P_{i,j} = u_i K_{i,j}v_j
    \end{equation}
    where $u_i,v_j>0$ and $K_{i,j}=\exp\left(-\frac{C_{i,j}}{\epsilon}\right)$.
\end{lemma}
\begin{proof}
    Uniqueness is left as an exercise.
    The Lagrangian for the problem reads as
    \begin{equation}
        L(P,\lambda,\mu)
        = \inner{C}{P} + \epsilon E(P) + \inner{\lambda}{P\1-a} + \inner{\mu}{P^t\1-b}.
    \end{equation}
    The corresponding optimality conditions are
    \begin{equation}
        \begin{cases}
            0 =& C_{i,j} + \epsilon\log(P_{i,j}) + \lambda_i + \mu_j\\
            0 =& P\1-a\\
            0 =& P^t\1-b.
        \end{cases}
    \end{equation}
    The first condition implies 
    \begin{equation}
        P_{i,j} = \exp\left(-\frac{\mu_j}{\epsilon}\right)\exp\left(-\frac{C_{i,j}}{\epsilon}\right)\exp\left(-\frac{\lambda_i}{\epsilon}\right)
    \end{equation}
    concluding the proof.
\end{proof}
Since we additionally still have the constraints $P\1 = \diag(u) K \diag(v)\1 = \diag(u) K v = a$ and similarly $\diag(v)K^t u\1 = b$ this motivates the Sinkhorn-Knopp algorithm where we update for $k=1,2,\dots$
\begin{equation}
    \begin{cases}
        u_{k+1} = a \oslash Kv_k\\
        v_{k+1} = b\oslash K^t u_{k+1}
    \end{cases}
\end{equation}
where $\oslash$ denotes element-wise division.

\backmatter
\printbibliography

\end{document}